\documentclass[10pt, reqno]{amsart}
\usepackage{amsmath,amssymb,amscd,mathrsfs,amscd}
\usepackage{graphics,verbatim}
\newdimen\hoogte    \hoogte=12pt    % hoogte  van hokje
\newdimen\breedte   \breedte=14pt   % breedte van hokje
\newdimen\dikte     \dikte=0.5pt    % dikte lijn

\newenvironment{Young}{\begingroup
       \def\vr{\vrule height0.8\hoogte width\dikte depth 0.2\hoogte}
       \def\fbox##1{\vbox{\offinterlineskip
                    \hrule height\dikte
                    \hbox to \breedte{\vr\hfill##1\hfill\vr}
                    \hrule height\dikte}}
       \vbox\bgroup \offinterlineskip \tabskip=-\dikte \lineskip=-\dikte
            \halign\bgroup &\fbox{##\unskip}\unskip  \crcr }
       {\egroup\egroup\endgroup}
\def\diagram#1{\relax\ifmmode\vcenter{\,\begin{Young}#1\end{Young}\,}\else%
              $\vcenter{\,\begin{Young}#1\end{Young}\,}$\fi}

\theoremstyle{plain}
\newtheorem{thm}{Theorem}[subsection]
\newtheorem{lem}[thm]{Lemma}
\newtheorem*{thme}{Theorem}

\newtheorem{prp}[thm]{Proposition}
\newtheorem{dfn}[thm]{Definition}
\newtheorem{cor}[thm]{Corollary}

\newtheorem{rmk}[thm]{Remark}

\newtheorem{exa}[thm]{Example}

\newenvironment{pff}{{\em Proof:}}{\QED}

\newcommand{\af}{\alpha}
\newcommand{\bt}{\beta}
\newcommand{\gm}{\gamma}
\newcommand{\dt}{\delta}
\newcommand{\ep}{\varepsilon}
\newcommand{\zt}{\zeta}

\newcommand{\te}{\theta}
\newcommand{\ld}{\lambda}
\newcommand{\sm}{\sigma}
\newcommand{\kp}{\kappa}
\newcommand{\ph}{\varphi}

\newcommand{\om}{\omega}

\newcommand{\Dt}{\Delta}

\newcommand{\Ph}{\Phi}

\newcommand{\Om}{\Omega}
\newcommand{\del}{\partial}

\newcommand{\Q}{{\mathbb{Q}}}
\newcommand{\Z}{{\mathbb{Z}}}
\newcommand{\R}{{\mathbb{R}}}
\newcommand{\C}{{\mathbb{C}}}

\newcommand{\Sl}{{\mathfrak{sl}}}

\newcommand{\gl}{{\mathfrak{gl}}}
\newcommand{\g}{{\mathfrak{g}}}
\newcommand{\h}{{\mathfrak{h}}}
\newcommand{\n}{{\mathfrak{n}}}

\newcommand{\q}{{\mathfrak{q}}}
\newcommand{\p}{{\mathfrak{p}}}
\renewcommand{\b}{{\mathfrak{b}}}

\newcommand{\He}{{\mathcal{H}^{\mathrm{aff}}}}
\newcommand{\Cl}{{\mathcal{C}\ell}}
\newcommand{\Se}{{\mathcal{S}}}
\newcommand{\A}{{\mathcal{A}}}
\newcommand{\ASe}{{\mathcal{H}_{\Cl}^{\mathrm{aff}}}}
\renewcommand{\P}{{\mathcal{P}}}
\renewcommand{\L}{{\mathcal{L}}}
\newcommand{\M}{{\mathcal{M}}}

\newcommand{\U}{{\mathcal{U}}}
\newcommand{\curlyQ}{{\mathcal{Q}}}
\newcommand{\F}{{\mathcal{F}}}
\newcommand{\W}{{\mathcal{W}}}

\newcommand{\str}{\mathrm{str}}
\newcommand{\ch}{\operatorname{ch}}
\newcommand{\ind}{\operatorname{Ind}}
\newcommand{\res}{\operatorname{Res}}
\newcommand{\coind}{\operatorname{Coind}}

\newcommand{\Ker}{\operatorname{Ker}}
\newcommand{\Hom}{\operatorname{Hom}}

\newcommand{\End}{\operatorname{End}}

\newcommand{\Rep}{\operatorname{Rep}}

\newcommand{\Rad}{\operatorname{Rad}}
\newcommand{\height}{\mathrm{ht}}

\renewcommand{\Pr}{{P^+_{\mathrm{rat}}}}
\newcommand{\Pt}{{P^{++}}}
\newcommand{\Pp}{{P_{\mathrm{poly}}^+}}
\newcommand{\Prt}{{P^{++}_{\mathrm{rat}}}}
\newcommand{\Ppt}{{P^{++}_{\mathrm{poly}}}}
\newcommand{\Ppos}{{P_{\geq0}}}

\newcommand{\zero}{{\bar{0}}}
\newcommand{\one}{{\bar{1}}}
\newcommand{\oi}{{\bar{i}}}
\newcommand{\oj}{{\bar{j}}}
\newcommand{\ok}{{\bar{k}}}
\newcommand{\e}{{\bar{e}}}
\newcommand{\f}{{\bar{f}}}

\newcommand{\ui}{{\underline{i}}}
\newcommand{\uj}{{\underline{j}}}

\newcommand{\la}{{\langle}}
\newcommand{\ra}{{\rangle}}

\newcommand{\va}{{\mathbf{1}}}

\newcommand{\andeqn}{\,\,\,\,\,\, {\mbox{and}} \,\,\,\,\,\,}
\newcommand{\QED}{\rule{0.4em}{2ex}}

\newcommand{\fq}{\ensuremath{\mathfrak{q}}}
\newcommand{\fg}{\ensuremath{\mathfrak{g}}}
\newcommand{\fb}{\ensuremath{\mathfrak{b}}}
\newcommand{\fh}{\ensuremath{\mathfrak{h}}}
\newcommand{\fnminus}{\ensuremath{\mathfrak{n}^{-}}}
\newcommand{\even}{\ensuremath{\bar{0}}}
\newcommand{\odd}{\ensuremath{\bar{1}}}
\numberwithin{equation}{subsection}

\input xy
\xyoption{all}

\begin{document}
\title{Degenerate Affine Hecke-Clifford Algebras and Type $Q$ Lie Superalgebras}

\author{David Hill }
\address{Department of Mathematics \\
            University of California, Berkeley \\
            Berkeley, CA 94720-3840}
\email{dhill1@math.berkeley.edu}
\author{Jonathan R. Kujawa}
\address{Department of Mathematics \\
            University of Oklahoma \\
            Norman, OK 73019}
\email{kujawa@math.ou.edu}
\author{Joshua Sussan}
\address{Department of Mathematics \\
            University of California, Berkeley \\
            Berkeley, CA 94720-3840}
\email{sussan@math.berkeley.edu}
\thanks{Research of the second author was partially supported by NSF grant
DMS-0734226. Research of the first and third author was partially supported by
NSF EMSW21-RTG grant DMS-0354321}\
\date{\today}

\begin{abstract} We construct the finite dimensional simple integral modules for the (degenerate) affine Hecke-Clifford algebra (AHCA), $\ASe(d)$. Our construction includes an
analogue of Zelevinsky's segment representations, a complete combinatorial
description of the simple calibrated $\ASe(d)$-modules, and a classification of the simple
integral $\ASe(d)$-modules. Our main tool is an analogue of the
Arakawa-Suzuki functor for the Lie superalgebra $\q(n)$.
\end{abstract}

\maketitle

\section{Introduction}\label{S:Intro}
\subsection{} Throughout this paper, we will work over the ground field $\C$. As is well
known,
 the symmetric group, $S_d$, has a non-trivial \emph{central extension}:
\[
\xymatrix{1\ar[r]&\Z/2\Z\ar[r]&\widehat{S}_d\ar[r]&S_d\ar[r]&1}.
\]
The double cover $\widehat{S}_d$ is generated by elements
$\zt,\hat{s}_1,\ldots,\hat{s}_{d-1}$, where $\zt$ is central, $\zt^2=1$, and
the $\hat{s}_i$ satisfy the relations
$\hat{s}_i\hat{s}_{i+1}\hat{s}_i=\hat{s}_{i+1}\hat{s}_i\hat{s}_{i+1}$ and
$\hat{s}_j\hat{s}_i=\zt\hat{s}_i\hat{s}_j$ for admissible $i$ and $j$
satisfying $|i-j|>1$. The \emph{projective} or $\emph{spin}$ representations of
$S_d$ are the linear representations of $\widehat{S}_d$ which factor through $\C\widehat{S}_d/(\zt+1)$.
This paper is a study of some structures arising from the
projective representation theory of symmetric groups.

The double cover $\widehat{S}_d$ suffers a defect: it is difficult to define
parabolic induction, see \cite[Section 4]{stem}. Since the inductive approach
to the study of linear representations of the symmetric group is so effective,
it is preferable to study the \emph{Sergeev algebra} $\Se(d)$ introduced in
\cite{s,n}, which provides a natural fix to this problem. As a vector space,
$\Se(d)=\Cl(d)\otimes\C S_d$, where $\Cl(d)$ is the $2^d$-dimensional Clifford
algebra with generators $c_1,\ldots,c_d$ subject to the relations $c_i^2=-1$
and $c_ic_j=-c_jc_i$ for $i\neq j$, and $\C S_d$ is the group algebra of $S_d$. Let
$s_i=(i,i+1)\in S_d$ be the $i$th basic transposition, and identify $\Cl(d)$
and $\C S_d$ with the subspaces $\Cl(d)\otimes 1$ and $1\otimes\C S_d$
respectively. Multiplication is defined so that $\Cl(d)$ and $\C S_d$ are
subalgebras, and $wc_i=c_{w(i)}w$ for all $1\leq i\leq d$ and $w\in S_d$. The
Sergeev algebra admits a natural definition of parabolic induction and the
projective representation theory of the symmetric group can be recovered from
that of $\Se(d)$, \cite[Theorem 3.4]{bk1}.

Additionally, the Sergeev algebra is a \emph{superalgebra}, and plays the role of the
symmetric group for a super version of Schur-Weyl duality known as Sergeev
duality in honor of A. N. Sergeev who extended the classical theorem of Schur
and Weyl \cite{s}.  If $V=\C^{n|n}$ is the standard representation of the Lie
superalgebra $\q(n)$, then both $\Se(d)$ and $\q(n)$ act on the tensor product
$V^{\otimes d}$ and each algebra is the commutant algebra of the other. In particular, there exists an isomorphism of superalgebras
\[
\Se(d)\rightarrow\End_{\q(n)}(V^{\otimes d}).
\]

The algebra $\Se(d)$ admits an affinization, $\ASe(d)$, called the (degenerate) affine Hecke-Clifford algebra (AHCA). The affine Hecke-Clifford
algebra was introduced by Nazarov in \cite{n} and studied in \cite{n,bk2,kl,w}. As a
vector space, $\ASe(d)=\P_d[x]\otimes\Se(d)$, where
$\P_d[x]=\C[x_1,\ldots,x_d]$. We identify $\P_d[x]$ and $\Se(d)$ with the
subspaces $\P_d[x]\otimes 1$ and $1\otimes\Se(d)$. Multiplication is defined so
that these are subalgebras, $c_ix_j=x_jc_i$ if $j\neq i$, $c_ix_i=-x_ic_i$,
$s_ix_j=x_js_i$ if $j\neq i,i+1$, and
\[
s_ix_i=x_{i+1}s_i-1+c_ic_{i+1}.
\]
In addition to $\Se (d)$ being a subalgebra of $\ASe(d)$, there also exists a natural
surjection $\ASe(d)\twoheadrightarrow\Se(d)$ obtained by mapping $x_1\mapsto
0$, $c_i\mapsto c_i$ and $s_i\mapsto s_i$. Therefore, the representation theory
of the AHCA contains that of the Sergeev algebra.

Surprisingly little is explicitly known about the representation theory of
$\ASe(d)$, in contrast with its linear counterpart, the \emph{(degenerate)
affine Hecke algebra} $\He(d)$. The most significant contribution to the
projective theory is from \cite{bk2,kl}, which describe the Grothendieck group
of the full subcategory of \emph{integral} $\ASe(d)$-modules in terms of the
crystal graph associated to a maximal nilpotent subalgebra of $\b_\infty$ (or,
more generally, $A_{2\ell}^{(2)}$ if working over a field of odd prime
characteristic $2\ell-1$). We will return to this important topic later
on.

The algebra $\He(d)$ has been studied for many years. Of particular interest
are those modules for $\He(d)$ which admit a generalized weight space
decomposition with respect to the polynomial generators. It is known that among
these modules it is enough to consider those for which the generalized
eigenvalues of the polynomial generators are integers, cf. \cite[$\S7.1$]{kl}.
These are known as \emph{integral modules}. As discovered in \cite{n}, the appropriate
analogue of integral modules for $\ASe(d)$ are those which admit a generalized
weight space decomposition with respect to the $x_i^2$, and the generalized
eigenvalues of the $x_i^2$ are of the form $q(a):=a(a+1)$, $a\in\Z$.

The finite dimensional, irreducible, integral modules for $\He(d)$ were
classified by Zelevinsky in \cite{z} via combinatorial objects known as
multisegments. A segment is an interval $[a,b]\in\Z$. To each segment $[a,b]$
with $d=b-a+1$, Zelevinsky associates a 1-dimensional $\He(d)$-module
$\C_{[a,b]}$ defined from the trivial representation of $\C S_d$ by letting
$x_1$ act by the scalar $a$. A multisegment may be regarded as a pair of
compositions $(\bt,\af)=((b_1,\ldots,b_n),(a_1,\ldots,a_n))\in\Z^n\times\Z^n$,
with $d_i=b_i-a_i\geq 0$. If $d=d_1+\cdots+d_n$, Zelevinsky associates to the
multisegment $(\bt,\af)$ a \emph{standard cyclic} $\He(d)$-module
\[
\M(\bt,\af) =\ind_{\He(d_1)\otimes\cdots\otimes\He(d_n)}^{\He(d)}
\C_{[a_1,b_1-1]}
    \boxtimes\cdots\boxtimes\C_{[a_n,b_n-1]}.
\]
To explain the classification, let $P=\Z^n$ be the weight lattice associated to
$\gl_n(\C)$, $P^+$ the dominant weights, and $\rho=(n-1,\ldots,1,0)$. Additionally, define the weights
\[
\Ppos(d)=\{\mu\in\Z^n_{\geq0}\mid \mu_1+\cdots+\mu_n=d\}\andeqn P^+[\ld]=\{\mu\in P\mid \mu_i\geq\mu_j\mbox{ whenever }\ld_i=\ld_j\}.
\]
Given
$\ld\in P^+$, let
\begin{align}\label{E:Bsubd}
\mathcal{B}_d[\ld]=\{\mu\in P^+[\ld]\mid \ld-\mu\in\Ppos(d)\},
\end{align}
and
\[
\mathcal{A}_d=\{(\ld,\mu)\mid \ld\in P^+,\mbox{ and
}\mu\in\mathcal{B}_d[\ld+\rho]\}.
\]
Then, the set $\{\L(\bt,\af)\mid (\bt,\af)\in\A_d\}$ is a complete list of
irreducible integral $\He(d)$-modules.

In the case of $\ASe(d)$, the situation is more subtle. To describe this, fix a
segment $[a,b]$. The obvious analogue of the trivial representation of $\C S_d$
is the $2^d$-dimensional basic spin representation $\Cl_d=\Cl(d).1$ of
$\Se(d)$. If $a=0$, the action of $\ASe(d)$ factors through $\Se(d)$ and it can
be checked that $\Cl_d$ is the desired segment representation. If $a\neq 0$, it
is not immediately obvious how to proceed. Inspiration comes from a \emph{rank
1} application of the functor described below. We define a module structure on
the \emph{double} of $\Cl_d$: $\hat{\Ph}_{[a,b]}=\Ph_a\otimes\Cl_d$, where
$\Ph_a$ is a 2-dimensional Clifford algebra. The module $\hat{\Ph}_{[a,b]}$ is
not irreducible, but decomposes as a direct sum of irreducibles
$\Ph_{[a,b]}^+\oplus\Ph_{[a,b]}^-$, where $\Ph_{[a,b]}^+$ and $\Ph_{[a,b]}^-$
are isomorphic via an \emph{odd} isomorphism.  Let  $\Ph_{[a,b]}$ denote one of these simple summands. Now, given a multisegment
$(\ld,\mu)$, with $\ld_i-\mu_i=d_i$ and $d=d_1+\cdots+d_n$, we define the
standard cyclic module
\[
\M(\ld,\mu)=\ind_{\ASe(d_1)\otimes\cdots\otimes\ASe(d_n)}^{\ASe(d)}
    \Ph_{[\mu_1,\ld_1-1]}\circledast\cdots\circledast\Ph_{[\mu_n,\ld_n-1]},
\]
where $\circledast$ is an analogue of the outer tensor product adapted for
superalgebras, see section \ref{S:Prelim} below.

A weight $\ld\in P$ is called typical if $\ld_i+\ld_j\neq0$ for all $i\neq j$.
Let
\[
\Pt=\{\ld\in P^+\mid \lambda_{1}\geq \dotsb  \geq \lambda_{n}, \text{ and } \ld_i+\ld_j\neq 0\mbox{ for all }i\neq j\}
\]
be the set of
dominant typical weights. We prove

\begin{thme}
Assume that $\ld\in\Pt$ and $\mu\in\mathcal{B}_d[\ld]$. Then, $\M(\ld,\mu)$ has
a unique simple quotient, denoted $\L(\ld,\mu)$.
\end{thme}

In the special case where the multisegment $(\lambda,\mu)$ corresponds to skew
shapes (i.e. $\lambda,\mu \in P^+$), the associated $\He(d)$-modules are called
calibrated. The calibrated representations may also be characterized as those
modules on which the polynomial generators act semisimply, and were originally classified by Cherednik in \cite{ch0}. In \cite{ram}, Ram
gives a complete combinatorial description of the calibrated representations of
$\He(d)$ in terms of skew shape tableaux and provides a complete classification
(see also \cite{kr} for another combinatorial model).

The projective analogue of the skew shapes are the shifted skew shapes which
have appeared already in \cite{s2,stem} and correspond to when $\ld$ and $\mu$
are \emph{strict} partitions. As in the linear case, these are the modules for which the $x_i$ act semisimply. In the
spirit of \cite{ram}, we prove that

\begin{thme}
For each shifted skew shape $\lambda/\mu$, where $\lambda$ and $\mu$ are strict
partitions such that $ \lambda $ contains $ \mu, $ there is an irreducible
$\ASe(d)$-module $H^{\lambda/\mu}$.  Every irreducible, calibrated
$\ASe(d)$-module is isomorphic to exactly one such $H^{\lambda/\mu}$.
\end{thme}

The $H^{\ld/\mu}$ are constructed directly using the combinatorics of shifted skew shapes.
Furthermore, we show that $H^{\ld/\mu}\cong\L(\ld,\mu)$. We would also like to point out that Wan, \cite{wan}, has recently obtained a classification of the calibrated representations for $\ASe(d)$ over any arbitrary algebraically closed field of characteristic not equal to 2.

The appearance of the weight lattice for $\gl_n(\C)$ in the representation
theory of $\He(d)$ is explained by a work of Arakawa and Suzuki who introduced
in \cite{as} a functor from the BGG category $ \mathcal{O}(\mathfrak{gl}_n) $
to the category of finite dimensional representations of $\He(d)$.  The authors
proved that the functor maps Verma modules to the standard modules or zero.
Using the Kazdhan-Lusztig conjecture together with the results of \cite{ginz},
they proved that simple objects in $\mathcal{O}(\mathfrak{gl}_n)$ are mapped by
the functor to simple modules or zero. In \cite{su1}, Suzuki avoided the
Kazdhan-Lusztig conjecture, and proved that the functor maps simples to simples
using Zelevinsky's classification together with the existence of a nonzero
$\He(d)$-contravariant form on certain standard modules, see \cite{r}. In
\cite{su2}, Suzuki was able to avoid the results of Zelevinsky and independently reproduce
the classification via a careful analysis of the standard modules. For a
complete explanation of the functor in type $A$, we refer the reader to
\cite{or}.

The functor and related constructions have had numerous applications in various areas of representation theory.  This includes the study of affine Braid groups and Hecke algebras \cite{or}, Yangians \cite{KN}, the centers of parabolic category $\mathcal{O}$ for $\mathfrak{gl}_{n}$ \cite{b2}, finite W-algebras \cite{bk4}, and the proof of Brou\'e's abelian defect conjecture for symmetric groups by Chuang and Rouquier via $\mathfrak{sl}_{2}$ categorification \cite{CR}.

We define an analogous functor from the category $\mathcal{O}(\q(n))$ to the
category of finite dimensional modules for $\ASe(d)$. The contruction of this
functor relies on the following key result:

\begin{thme}
Let $M$ be a $\q(n)$-supermodule. Then, there exists a homomorphism
\[
\ASe(d)\rightarrow\End_{\fq (n)}(M\otimes V^{\otimes d}).
\]
\end{thme}

To define the functor, let $\q(n)=\n^+\oplus\h\oplus\n^-$ be the triangular
decomposition of $\q(n)$. For each $\ld\in P$, the functor
\[
F_\ld:\mathcal{O}(\q(n))\rightarrow\ASe(d)\mbox{-mod}
\]
is defined by
\begin{equation*}
F_\ld M=\{\,m\in M \mid \n^+.m=0 \text{ and } hv =\lambda(h)v \text{ for all } h\in \h \}.
\end{equation*}
The functor $F_\ld$ is exact when $\ld\in\Pt$.

The dimension of the highest weight space of a Verma module in
$\mathcal{O}(\q(n))$ is generally greater than one. A consequence of this is
that the functor maps a Verma module to a direct sum of the same standard
module. A simple object in $ \mathcal{O}(\mathfrak{q}(n)) $ is mapped to a
direct sum of the same simple module or else zero.  Determining when a simple
object is mapped to something non-zero is a more difficult question than in the
non-super case and we have only partial results in this direction. The main
difficulty is a lack of information about the category $\mathcal{O}(\q(n))$.
The category of finite dimensional
representations of $\q(n)$ has been studied by Penkov and Serganova
\cite{p,ps,ps2}; they give a character formula for all finite dimension simple $\q(n)$-modules.
Using other methods, Brundan \cite{b} has also studied this
category, and has even obtained some (conjectural) information about the whole category
$\mathcal{O}(\q(n))$ via the theory of crystals. The most useful information, however, comes from Gorelik
\cite{g}, who defines the Shapovalov form for Verma modules and calculates the
linear factorization of its determinant.

In various works by Ariki, Grojnowski, Vazirani, and Kleshchev
\cite{ar,gr,v,kl} it was shown that there is an action of $U(\gl_\infty)$ on the direct sum of Grothendieck groups of the categories of integral $\He(d)$-modules, for all $d$. This gives another type of classification of the simple integral modules as
nodes on the crystal graph associated to a maximal nilpotent subalgebra of
$\gl_\infty$. In \cite{bk1}, Brundan and Kleshchev show there is a
classification of the simple integral modules for $\ASe(d)$ parameterized by
the nodes of the crystal graph associated to a maximal nilpotent subalgebra of
$\b_\infty$, see also \cite{kl}.

In \cite{lec}, Leclerc studied dual canonical bases of the quantum group
$\U_q(\g)$ for various finite dimensional simple Lie algebras $\g$ via
embeddings of the quantized enveloping algebra $\U_q(\n)$ of a maximal
nilpotent subalgebra $\n\subseteq\g$ in the \emph{quantum shuffle algebra}. To
describe the quantum shuffle algebra associated to $\g$ of rank $r$, let
$\mathcal{F}$ be the free associative algebra on the letters
$[0],\ldots,[r-1]$, and let $[i_1,i_2,\ldots,i_k]:=[i_1]\cdot[i_2]\cdots[i_k]$.
Then, the quantum shuffle algebra is the algebra $(\mathcal{F},*)$, where
\[
[i_1,\ldots,i_k]*[i_{k+1},\ldots,i_{k+\ell}]=\sum_\sm
q^{-e(\sm)}[i_\sm(1),\ldots,i_{\sm(k+\ell)}],
\]
where the sum is over all minimal length coset representatives in
$S_{k+\ell}/(S_k\times S_\ell)$, and $e(\sm)$ is some explicit function of $\sm$.
There exists an \emph{injective} homomorphism $\Psi:\U_q(\n)\hookrightarrow\F$
satisfying $\Psi(xy)=\Psi(x)*\Psi(y)$ for all $x,y\in\U_q(\n)$. Let
$\mathcal{W}=\Psi(\U_q(\n))$.

The ordering $[0]<[1]<\cdots<[r-1]$ yields two total ordering on words in $\F$:
One the standard lexicographic ordering reading from \emph{left to right}, and
the other the \emph{costandard} lexicographic ordering reading from \emph{right
to left}. These orderings give rise to special words in $\F$ called Lyndon
words, and every word has a canonical factorization as a non-increasing product
of Lyndon words. In \cite{lec}, Leclerc uses the standard ordering, while we
use the costandard ordering. It is easy to translate between results using one
ordering as opposed to the other. However, in our situation, choosing the
costandard ordering leads to some significant differences in the \emph{shape}
of Lyndon words. We will explain this shortly.

Bases for $\W$ are parameterized by certain words called \emph{good words}.
 A \emph{good word} is a nonincreasing product of
\emph{good Lyndon word} which have been studied in \cite{lr, ro1,ro2,ro3}. The
good Lyndon words are in 1-1 correspondence with the positive roots, $\Dt^+$,
of $\g$, and the (standard or costandard) lexicographic ordering on good Lyndon
words gives rise to a convex ordering on $\Dt^+$. The convex ordering on
$\Dt^+$ gives rise to a PBW basis for $\U_q(\n)$, which in turn gives a
multiplicative basis $\{E^*_g=(E^*_{l_k})*\cdots*(E_{l_1}^*)\}$ for $\W$
labeled by good words $g=l_1\cdots l_k$, where $l_1\geq\cdots \geq l_k$ are
good Lyndon words. Additionally, the bar involution on $\U_q(\n)$ gives rise to
a bar involution on $\W$, and hence, a \emph{dual canonical basis} $\{b^*_g\}$
labeled by good words. The transition matrix between the basis $\{E^*_g\}$ and
$\{b^*_g\}$ is triangular and, in particular, $b^*_l=E^*_l$ for each good
Lyndon word $l$. In what follows, let $\underline{w}$ denote the specialization
at $q=1$ of an element $w\in\W$.

For $\g$ of type $A_\infty=\underrightarrow{\lim}A_r$, good Lyndon words are
labelled by segments $[a,b]$, and there is no difference between the standard
and costandard ordering. In this case, for a good Lyndon word $l$,
$\underline{E^*_l}=l$. The Mackey theorem for $\He(d)$ (see section
\ref{SS:Mackey}) implies that the formal character of a standard module
$\M(\bt,\af)$ is given by $\underline{E^*_g}$, where $g$ is the good word
$[\af_1,\ldots,\bt_1-1,\ldots,\af_n,\ldots,\bt_n-1]$. A much deeper fact,
proved by Ariki in \cite{ar}, is that the character of the simple module
$\L(\bt,\af)$ is given by the dual canonical basis element $\underline{b^*_g}$.

Leclerc also studied the Lie algebra $\b_r$ of type $B_r$, and hence that of
type $B_\infty=\underrightarrow{\lim}B_r$. The good Lyndon words for $\b_r$
with respect to the standard ordering are segments $[i,\ldots,j]$, $0\leq i\leq
j<r$, and \emph{double segments} $[0,\ldots,j,0,\ldots,k]$, $0\leq j<k<r$ (cf.
\cite[$\S8.2$]{lec}). In this case, when $l=[i,\ldots,j]$ is a segment,
$\underline{b_l^*}=[i,\ldots,j]=\ch\Ph_{[i,j]}$. However, when
$l=[0,\ldots,j,0,\ldots,k]$ is a double segment
\begin{align}\label{E:StdDblSeg}
\underline{b^*_l}=2[0]\cdot([0,\ldots,j]*[1,\ldots,k]).
\end{align}
When we adopt the costandard ordering, the picture becomes much more familiar.
Indeed, the good Lyndon words are of the form $[i,\ldots,j]$ $0\leq i<j<r$ and
$[j,\ldots,0,0,\ldots,k]$, $0\leq j<k<r$! In particular, they correspond to
weights of the segment representations $\Phi_{[i,j]}$ and $\Phi_{[-j-1,k]}$
respectively. Moreover, for $l=[j,\ldots,0,0,\ldots,k]$
\[
\underline{b^*_{l}}=2[j,\ldots,0,0,\ldots,k]=\ch \Phi_{[-j-1,k]}.
\]

Leclerc conjectures \cite[Conjecture 52]{lec} that for each good word $g$ of
\emph{principal degree} $d$, there exists a simple $\ASe(d)$-module with
character given by $b^*_g$. We are not yet able to confirm the conjecture for
general good words. However, the combinatorial construction of $H^{\ld/\mu}$
immediately implies Leclerc's conjecture for calibrated representations (cf.
\cite[Proposition 51]{lec} and Corollary \ref{C:characters}). Additionally, for
each good Lyndon word $l$ (with respect to the costandard ordering), there is a
simple module with character $b^*_l$.

Also, an application of the functor $F_\ld$ gives a representation theoretic
interpretation of \eqref{E:StdDblSeg} above. Indeed, let $\ld=(k+1,j+1)$ and
$\af=(1,-1)$. Then,
\[
\ch\L(\ld,-\af)=2[0]\cdot([0,\ldots,j]*[1,\ldots,k]).
\]

Finally, the analysis of good Lyndon words leads to a classification of simple
integral modules for $\ASe(d)$. Indeed, recall the set \eqref{E:Bsubd}, and let
\[
\mathcal{B}_d=\{(\ld,\mu)\mid \ld\in\Pt,\mbox{ and }\mu\in\mathcal{B}_d(\ld)\}.
\]
Then,

\begin{thme} The following is a complete list of pairwise non-isomorphic simple modules
for $\ASe(d)$:
\[
\{\,\L(\ld,\mu) \mid (\ld,\mu)\in \mathcal{B}_d\,\}.
\]
\end{thme}

We believe this paper may serve as a starting point for future investigations
into categorification theories associated to non-simply laced Dynkin diagrams. In
particular, we hope that the functor introduced here will play a role in
showing that the 2-category for $\mathfrak{b}_\infty$, introduced by Khovanov-Lauda  and independently by Rouquier, acts on
$\mathcal{O}(\mathfrak{q}(n))$, see \cite{khl1,khl2,khl3,rq}.
Additionally, in \cite{wz}, Wang and Zhao initiated a study of super analogues of $W$-algebras.
This functor should be useful for studying these $W$-superalgebras along the lines of
\cite{bk3, bk4}.

In \cite{b}, Brundan studied the category of finite dimensional modules for
$\q(n)$ via Kazhdan-Lusztig theory. Among the finite dimensional
$\q(n)$-modules are the polynomial representations, which correspond under our
functor to calibrated representations. Other modules in this category are those
associated to \emph{rational} weights, i.e.\  strict partitions with negative
parts allowed. The functor should map these modules to interesting
$\ASe(d)$-modules. These should be investigated. It would also be interesting
to compare the Kazhdan-Lusztig polynomials in \cite{b} to those appearing in
\cite{lec}.

We now briefly outline the paper. In section~\ref{S:Prelim}, we review some
basic notion of super representation theory.  In section ~\ref{S:ASA} we define
the degenerate AHCA and review some of its properties which
may also be found in \cite{kl}.  The standard modules and their irreducible
quotients are introduced in section \ref{S:standardreps}.  The classification
of the calibrated representations are given in section ~\ref{S:Calibrated}. In
section ~\ref{S:Lie algebras} we review some basic notions about category $
\mathcal{O}(\mathfrak{q}(n)) $ which may be found in \cite{b,g}. Next, in
section ~\ref{S:LieTheoreticConstr} the functor is developed along with its
properties. Finally, in section \ref{S:Classification} a classification of
simple modules is obtained.

\subsection{Acknowlegments}\label{SS:acknowlegements}  The work presented in this paper was begun while the second author visited the Mathematical Sciences Research Institute in Berkeley, CA.  He would like to thank the administration and staff of MSRI for their hospitality and especially the organizers of the ``Combinatorial Representation Theory'' and ``Representation Theory of Finite Groups and Related Topics'' programs for providing an exceptionally stimulating semester.

We would like to thank Mikhail Khovanov for suggesting we consider a super analogue of the Arakawa-Suzuki functor. We would also like to thank Bernard Leclerc for pointing out \cite{lec}, as well as Monica Vazirani and Weiqiang Wang for some useful comments.

\section{(Associative) Superalgebras and Their Modules}\label{S:Prelim} We now review
some basics of the theory of superalgebras, following \cite{bk1,bk2,kl}. The
objects in this theory are $\Z_2$-graded. Throughout the exposition, we will
make definitions for homogeneous elements in this grading. These definitions
should always be extended by linearity. Also, we often choose to not write the
prefix \emph{super}. As the paper progresses this term may be dropped; however,
we will always point out when we are explicitly ignoring the $\Z_2$-grading.

A vector superspace is a $\Z_2$-graded $\C$-vector space $V=V_{\zero}\oplus
V_{\one}$. Given a nonzero homogeneous vector $v\in V_{\oi}$, let $p(v) = \oi  \in\Z_2$
be its \emph{parity}. Given a superspace $V$, let $\Pi V$ be the superspace
obtained by reversing the parity. That is, $\Pi V_\oi=V_{\oi+1}$. A
supersubspace of $V$ is a \emph{graded} subspace $U\subseteq V$. That is,
$U=(U\cap V_{\zero})\oplus(U\cap V_{\one})$. Observe that $U$ is a
supersubspace if, and only if, $U$ is stable under the map
$v\mapsto(-1)^{p(v)}v$ for homogeneous vectors $v\in V$.

Given two superspaces $V,W$, the direct sum $V\oplus W$
and tensor product $V\otimes W$ satisfy $(V\oplus
W)_{\oi}=V_\oi\oplus W_\oi$ and
\[
(V\otimes W)_\oi=\bigoplus_{\oj+\ok=\oi}V_\oj\otimes W_\ok.
\]
We may regard $\Hom_\C(V,W)$ as a superspace by setting
$\Hom_\C(V,W)_\oi$ to be the set of all homogeneous linear maps of
degree $\oi$. That is, linear maps $\ph:V\rightarrow W$ such that
$\ph(V_\oj)\subseteq W_{\oj+\oi}$. Finally,
$V^*=\Hom_\C(V,\C)$ is a superspace, where $\C=\C_\zero$.

Now, a superalgebra is a vector superspace $A$ that has the
structure of an associative, unital algebra such that $A_\oi
A_\oj\subseteq A_{\oi+\oj}$. A superideal of $A$ is a two sided
ideal of $A$ that is also a supersubspace of $A$. A superalgebra
homomorphism $\ph:A\rightarrow B$ is an even (i.e.\ grading preserving) linear map which is
also an algebra homomorphism. Observe that since $\ph$ is even, its
kernel, $\ker\ph$, is a superideal of $A$. Finally, given
superalgebras $A$ and $B$, their tensor product $A\otimes B$ is a
superalgebra with product given by
\begin{equation}\label{tensor product rule-algebra}
(a\otimes b)(a'\otimes b')=(-1)^{p(a')p(b)}(aa'\otimes bb').
\end{equation}

We now turn our attention to supermodules. Given a superalgebra $A$, let
$A$-smod denote the category of all finite dimensional $A$-supermodules, and $A$-mod be the category
of $A$-modules in the usual ungraded sense. An object in $A$-smod is a
$\Z_2$-graded left $A$-module $M=M_\zero\oplus M_\one$ such that $A_\oi
M_\oj\subseteq M_{\oi+\oj}$. A homomorphism of $A$-supermodules $M$ and $N$ is
a map of vector superspaces $f:M\rightarrow N$ satisfying
$f(am)=(-1)^{p(a)p(f)}af(m)$ when $f$ is homogeneous. A submodule of an
$A$-supermodule $M$ will always be a supersubspace of $M$. An $A$-supermodule
$M$ is called irreducible if it contains no proper nontrivial subsupermodules.

The supermodule $M$ may or may not remain irreducible when regarded
as an object in $A$-mod. If $M$ remains irreducible as an
$A$-module, it is called \emph{absolutely irreducible}, and if it
decomposes, it is called \emph{self associate}. Alternatively,
absolutely irreducible supermodules are said to be irreducible of
type \texttt{M}, while self associate supermodules are irreducible
of type \texttt{Q}. When $M\in A$-smod is self associate, there
exists an odd $A$-smod homomorphism $\te_M$ which interchanges the
two irreducible components of $M$ as an object in $A$-mod.

Now, let $A$ and $B$ be superalgebras, $M\in A$-smod and $N\in
B$-smod. The vector superspace $M\otimes N$ has the
structure of an $A\otimes B$-supermodule via the action is given
by
\begin{eqnarray}\label{tensor product rule-module}
(a\otimes b)(m\otimes n)=(-1)^{p(b)p(m)}(am\otimes bn)
\end{eqnarray}
for homogeneous $b\in B$ and $m\in M$. This is called the outer
tensor product of $M$ and $N$ and is denoted $M\boxtimes N$.

Unlike the classical situation, it may happen that the outer tensor product of
irreducible supermodules is no longer irreducible. This only happens when both
modules are self associate. To see this, let $M\in A$-smod and $N\in B$-smod be
self associate, and recall the odd homomorphisms $\te_M$ and $\te_N$. Then,
$\te_M\otimes\te_N:M\boxtimes N\rightarrow M\boxtimes N$, is an even
automorphism of $M\boxtimes N$ that squares to $-1$. Hence $M\boxtimes N$
decomposes as direct sum of two $A\otimes B$-supermodules, namely the
$(\pm\sqrt{-1})$-eigenspaces. These two summands are absolutely irreducible and
isomorphic under the odd isomorphism $\Theta_{M,N}:=\te_M\otimes\mathrm{id}_N$,
see \cite[Lemma 2.9]{bk1} and \cite[Section 2-b]{bk2}. When $M$ and $N$ are
irreducible, define the (irreducible) $A\otimes B$-module $M\circledast N$ by
the formula
\begin{equation}\label{E:startensor}
M\boxtimes N
= \begin{cases} M\circledast N, & \text{if either $M$ or $N$ is of type \texttt{M};}\\
                     (M\circledast N)\oplus\Theta_{M,N}(M\circledast N),& \text{if both $M$ and $N$ are of type \texttt{Q}.}
\end{cases}
\end{equation}
When $M=M'\oplus M''$, define $M\circledast N=(M'\circledast
N)\oplus(M''\circledast N)$.

Finally, let $A-\mbox{smod}_{\mbox{ev}}$ be the abelian subcategory of
$A-\mbox{smod}$ with the same objects, but only \emph{even} morphisms. Then,
the Grothendieck group $K(A-\mbox{smod})$ is the quotient of the Grothendieck
group $K(A-\mbox{smod}_{\mbox{ev}})$ modulo the relation $M-\Pi M$ for every
$A$-supermodule $M$. We would like to emphasize again that we allow odd
morphisms and, therefore, $M\cong\Pi M$ in the original category.

\section{The Degenerate Affine Hecke-Clifford Algebra}\label{S:ASA}

In this section we define the algebra which is the principle object of study in
this paper and summarize the results we will require in what follows. Many of
the results may be found in \cite{kl}, however, we include them here in an
effort to make this paper self contained and readable to a wider audience.

\subsection{The Algebra}\label{SS:Saffdef} Let $\Cl(d)$ denote the Clifford algebra
over $\C$ with generators $c_1,\ldots,c_d$, and relations
\begin{eqnarray}\label{c}
c_i^2=-1,\;\;\; c_ic_j=-c_jc_i\;\;\; 1\leq i\neq j\leq d.
\end{eqnarray}  Then $\Cl (d)$ is a superalgebra by declaring the generators
$c_{1}, \dotsc , c_{d}$ to all be of degree $\odd$.

Let $S_d$ be the symmetric group on $d$ letters with Coxeter generators
$s_1,\ldots, s_{d-1}$ and relations
\begin{eqnarray}\label{s}
s_i^2=1\;\;\; s_is_{i+1}s_i=s_{i+1}s_is_{i+1}\;\;\;s_is_j=s_js_i
\end{eqnarray}
for all admissible $i$ and $j$ such that $|i-j|>1$.  The group algebra of the
symmetric group, $\C S_{d}$, is a superalgebra by viewing it as concentrated in
degree $\even$; that is, $(\C S_{d})_{\even}= \C S_{d}$.

The \emph{Sergeev algebra} is given by setting
\[
\Se(d)= \Cl(d)\otimes \C S_d
\] as a vector superspace and declaring
$\Cl(d) \cong \Cl(d)\otimes 1$ and $\C S_d \cong 1\otimes\C S_d$ to be
subsuperalgebras. The Clifford generators $c_1,\ldots,c_d$ and Coxeter
generators $s_1,\ldots,s_{d-1}$ are subject to the mixed relation
\begin{eqnarray}\label{c&s}
s_ic_i=c_{i+1}s_i,\;\;\;s_ic_{i+1}=c_is_i,\;\;\; s_ic_j=c_js_i,
\end{eqnarray}
for all admissible $i$ and $j$ such that $j\neq i,i+1$.

The algebra of primary interest in this paper is the \emph{(degenerate) affine Hecke-Clifford algebra}, AHCA.  It is given as
\[
\ASe(d) = \P_d[x]\otimes\Se(d)
\]
as a vector superspace, where $\P_d[x]:=\C[x_1,\ldots,x_d]$ is the polynomial
ring in $d$ variables and is viewed as a superalgebra concentrated in degree
$\even$.  Multiplication is defined so that $\Se (d) \cong 1\otimes\Se (d) $ and $\P_{d}[x] \cong \P_{d}[x]
\otimes 1$ are subsuperalgebras.  The generators of these two
subalgebras are subject to the mixed relations
\begin{eqnarray}\label{c&x}
c_ix_i=-x_ic_i,\;\;\;c_jx_i=x_ic_j,\;\;\;1\leq i\neq j\leq d,
\end{eqnarray}
and
\begin{eqnarray}\label{s&x}
s_ix_i=x_{i+1}s_i-1+c_ic_{i+1},\;\;\;s_ix_j=x_js_i
\end{eqnarray}
for $1\leq i\leq d-1$, $1\leq j\leq d$, $j\neq i,i+1$.

Note that relation \eqref{s&x} differs from the corresponding relation in
\cite{bk2,kl}. This is because in \eqref{c} we choose $c_{i}^{2}=-1$, following
\cite{o,s,s2}, whereas in \emph{loc. cit.} the authors take $c_{i}^{2}=1$.  The
resulting algebras are isomorphic and the only effect of this convention is
that this change of sign has to be taken into account when comparing formulae.

It will be useful to
consider another decomposition
\begin{equation}\label{E:AlternateDecomp}
\ASe(d) \cong\A(d)\otimes\C S_d,
\end{equation}
where $A(d)$ is the subalgebra generated by $\Cl(d)$ and $\P_{d}[x]$.  As a
superspace
\begin{equation}\label{E:Adef}
A(d) \cong \P_d[x]\otimes\Cl(d).
\end{equation}

We have the following
PBW-type theorem for $\ASe(d)$.  Given
$\af=(\af_1,\ldots,\af_d)\in\Z_{\geq0}^d$ and
$\ep=(\ep_1,\ldots,\ep_d)\in\Z_2^d$, set $x^\af=x_1^{\af_1}\cdots x_d^{\af_d}$
and $c^\ep=c_1^{\ep_1}\cdots c_d^{\ep_d}$. Then,

\begin{thm}\cite[Theorem 14.2.2]{kl} The set $\{\,x^\af c^\ep
w\,|\,\af\in\Z_{\geq0}^d,\,\ep\in\Z_2^d,\,w\in S_d\}$ forms a basis for
$\ASe(d)$.
\end{thm}

\subsection{Some (Anti)Automorphisms}\label{SS:alghomoms}  The superalgebra $\ASe (d)$
admits an automorphism $\sm:\ASe(d)\rightarrow\ASe(d)$ given by
\begin{equation}\label{E:sigmadef}
\sm(s_i)=-s_{d-i},  \hspace{.25in} \sm(c_i)=c_{d+1-i},  \hspace{.25in}
\sm(x_i)=x_{n+1-i}.
\end{equation}

 It also admits an antiautomorphism $\tau:\ASe(d)\rightarrow\ASe(d)$ given by
\[
\tau(s_i)=s_i, \hspace{.25in} \tau(c_i)=-c_i, \hspace{.25in} \tau(x_i)=x_i.
\]
Note that, for superalgebras, antiautomorphism means that, for any homogeneous
$x,y \in \ASe (d)$,
\begin{equation}\label{E:taudef}
\tau(xy) = (-1)^{p(x)p(y)}\tau(y) \tau(x).
\end{equation}

\subsection{Weights and Integral Modules}\label{SS:weights}  We now introduce the class
of integral $\ASe (d)$-modules. It is these modules which are the main focus
of the paper.  To this end, for each $a\in\C$, define
\begin{equation}\label{E:qdef}
q(a)=a(a+1).
\end{equation}
By \cite[Theorem 14.3.1]{kl}, the center of $\ASe(d)$ consists of symmetric polynomials in $x_1^2,\ldots,x_d^2$.
Let $\P_{d}[x^{2}]=\C[x_1^2,\ldots,x_d^2]\subset\P_d[x]$. A \emph{weight} is an
algebra homomorphism
\[
\zt:\P_d[x^2]\rightarrow\C.
\]
It is often convenient to identify a weight $\zt$ with the $d$-tuple of complex
numbers $\zt=(\zt(x_1^2),\ldots,\zt(x_d^2))\in\C^d$.

Given an $\ASe(d)$-supermodule $M$ and a weight $\zt$, define the \emph{$\zeta$
weight space},
\[
M_\zeta=\left\{ m\in\M \mid x_i^2m =q\left(\zeta\left( x_i^2\right)\right)m
\text{ for all $i=1,\ldots,d$} \right\},
\]
and the \emph{generalized $\zeta$ weight space},
\[
M_\zeta^{\mathrm{gen}} =\left\{ m\in\M \mid \left( x_i^2-q(\zeta\left(x_i^2
\right)\right)^km=0 \text{ for $k\gg 0$  and all $i=1,\ldots,d$} \right\}.
\]
Observe that if $M_\zt^{\text{gen}}\neq 0$, then $M_\zt\neq0$.

Following \cite{bk2}, say that an $\ASe(d)$-module $M$ is \emph{integral} if
\[
M=\bigoplus_\zt M_\zt^{\text{gen}}
\]
and $M^{\text{gen}}_\zt\neq0$ implies $\zt\left( x_i^2\right)\in\Z$ for
$i=1,\ldots,d$.

Let $\Rep\ASe(d)$ denote the full subcategory of $\ASe(d)$-smod of finite
dimensional \emph{integral} modules for the degenerate AHCA.
Unless stated otherwise, all $\ASe(d)$-modules will be integral by assumption.

\subsection{The Mackey Theorem}\label{SS:Mackey} In this section we review the
Mackey Theorem for integral $\ASe$-modules. Refer to \cite{kl} for details.

Let $\mu=(\mu_1,\ldots,\mu_k)$ be a composition of $d$. Define the parabolic
subgroup $S_\mu=S_{\mu_1}\times\cdots\times S_{\mu_k}\subseteq S_d$, and
parabolic subalgebra
$\ASe(\mu):=\ASe(\mu_1)\otimes\cdots \otimes \ASe(\mu_k)\subseteq\ASe(d)$. Define the
functor
\[
\ind_\mu^d:\Rep\ASe(\mu)\rightarrow\Rep\ASe(d),\;\;\;
\ind_\mu^dM=\ASe(d)\otimes_{\ASe(\mu)}M.
\]
This functor is left adjoint to
$\res_\mu^d:\Rep\ASe(d)\rightarrow\Rep\ASe(\mu)$. Also, given a composition
$\nu=(\nu_1,\ldots,\nu_\ell)$ of $d$, which is a refinement of $\mu$ (i.e.\
there exist $0=i_1\leq\ldots\leq i_{k+1}=\ell$ such that
$\nu_{i_j}+\ldots+\nu_{i_{j+1}-1}=\mu_j$), define $\ind_\nu^\mu$ and
$\res_\nu^\mu$ in the obvious way.

Now, let $\mu$ and $\nu$ be compositions of $d$, and let $D_{\mu,\nu}$ denote
the set of minimal length $S_\mu\backslash S_d/S_\nu$-double coset
representatives and $D_\nu=D_{(1^d),\nu}$. Let $w\in D_{\mu,\nu}$. The
following lemma is standard.

\begin{lem}\label{L:MinCosetReps} Let $\nu=(\nu_1,\ldots,\nu_n)$ be a composition
of $d$, and set $a_i=\nu_1+\cdots+\nu_{i-1}+1$ and $b_i=\nu_1+\cdots+\nu_i$. If
$w\in D_\nu$ and $a_i\leq k<k'\leq b_i$ for some $i$, then $w(k)<w(k')$.
\end{lem}

It is known that $S_\mu\cap wS_\nu w^{-1}$ and $w^{-1}S_\mu w\cap S_\nu$ are
parabolic subgroups of $S_d$. Hence we may define compositions $\mu\cap w\nu$
and $w^{-1}\mu\cap\nu$ by the formulae
\[
S_\mu\cap w^{-1}S_\nu w=S_{\mu\cap w\nu}\andeqn w^{-1}S_\mu w\cap
S_\nu=S_{w^{-1}\mu\cap\nu}.
\]
Moreover, the map $\sm\mapsto w\sm w^{-1}$ induces a length preserving
isomorphism $S_{\mu\cap w\nu}\rightarrow S_{w^{-1}\mu\cap\nu}$.

Using this last fact, it can be proved that for each $w\in D_{\mu,\nu}$ there
exists an algebra isomorphism
\[
\ph_{w^{-1}}:\ASe(\mu\cap w\nu)\rightarrow\ASe(w^{-1}\mu\cap\nu)
\]
given by $\ph_{w^{-1}}(\sm)=w^{-1}\sm w$, $\ph_{w^{-1}}(c_i)=c_{w^{-1}(i)}$ and
$\ph_{w^{-1}}(x_i)=x_{w^{-1}(i)}$ for $1\leq i\leq d$ and $\sm\in S_{\mu\cap
w\nu}$. If $M$ is a left $\ASe(\mu\cap w\nu)$-supermodule, let $^wM$ denote the
$\ASe(w^{-1}\mu\cap\nu)$-supermodule obtained by twisting the action with the
isomorphism $\ph_{w^{-1}}$. We have the following ``Mackey Theorem'':

\begin{thm}\label{Mackey}\cite[Theorem 14.2.5]{kl} Let $M$ be an
$\ASe(\nu)$-supermodule. Then $\res_\mu^d\ind_\nu^dM$ admits a filtration with
subquotients isomorphic to
\[
\ind_{\mu\cap w\nu}^\mu{}^w(\res_{w^{-1}\mu\cap\nu}^\nu M),
\]
one for each $w\in D_{\mu,\nu}$. Moreover the subquotients can be taken in any
order refining the Bruhat order on $D_{\mu,\nu}$. In particular,
$\ind_{\mu\cap\nu}^\mu\res_{\mu\cap\nu}^\nu M$ appears as a subsupermodule.
\end{thm}

\subsection{Characters}\label{SS:characters}  Following \cite[Chapter 16]{kl},
we now describe the notion of characters for integral $\ASe(d)$-supermodules.

Recall the subsuperalgebra $\A(d)\subseteq\ASe(d)$ defined in \eqref{E:Adef}.
When $d=1$ and $a\in\Z$ there exists a $2$-dimensional simple $\A(1)$-module
\[
\L(a)=\Cl(1)1_a=\C1_a\oplus\C c_1.1_a,
\]
which is free as a $\Cl (1)$-module satisfying
\[
x_1.1_a=\sqrt{q(a)}1_a.
\]
The $\Z_{2}$-grading on $\L(a)$ is given by setting $p(1_{a})=\even$.

Observe that $\L(a)\cong\L(-a-1)$ and that by replacing $\sqrt{q(a)}$ with
$-\sqrt{q(a)}$ in the action of $x_1$ yields an isomorphic supermodule under
the odd isomorphism $1_a\mapsto c_1.1_a$. A direct calculation verifies that
this module is of type \texttt{M} if $a\neq 0$ and of type \texttt{Q} if $a=0$.

Now, $\A(d) \cong \A(1)\otimes\cdots\otimes\A(1)$. Hence, applying
\eqref{E:startensor} we obtain a simple $\A (d)$-module
$\L(a_1)\circledast\cdots\circledast\L(a_d)$. Given $(a_{1}, \dotsc ,
a_{d})\in\Z^d_{\geq0}$, let
\begin{equation}\label{E:gammazerodef}
\gamma_{0}(a_{1}, \dotsc, a_{d})=|\{ i \mid a_i=0 \}|.
\end{equation}
We have

\begin{lem}\label{A(d) irreducibles}\cite[Lemma 16.1.1]{kl} The set
\[
\left\{\L(a_1)\circledast\cdots\circledast \L(a_d) \mid
(a_1,\ldots,a_d)\in\Z_{\geq0}^d \right\}
\] is a complete set of pairwise non-isomorphic irreducible integral $\A (d)$-modules.

The module $\L(a_1)\circledast\cdots\circledast\L(a_d)$ is of type \texttt{M}
if $\gm_0$ is even and of type \texttt{Q} if $\gm_0$ is odd. Moreover,
\[
\dim\L(a_1)\circledast\cdots\circledast\L(a_d)=2^{n-\lfloor\gm_0/2\rfloor}
\]
where $\gm_0=\gm_0(a_1,\ldots,a_d)$ as above.
\end{lem}

Restriction to the subalgebra $A(d)=\ASe((1^d))\subseteq\ASe(d)$ defines a
functor from $\Rep\ASe(d)$ to $\A (d)$-mod.  The map obtained by applying this
functor and passing to the Grothendieck group of the category $\A(d)$-mod
yields a map
\[
\ch:\Rep\ASe(d)\rightarrow K(\A(d)\mbox{-mod})
\]
defined by
\[
\ch M=\left[ \res^{d}_{1^d}M \right]
\]
where $[X]$ is the image of an $\A(d)$-module, $X$, in $K(\A(d)\mbox{-mod})$.
The image $\ch M$ is called the \emph{formal character} of the $\ASe(d)$-module
$M$.

The following fundamental result is given in \cite[Theorem 17.3.1]{kl}.
\begin{lem}\label{L:independenceofcharacters} The induced map on Grothendeick rings
\[
\ch : K(\Rep\ASe(d)) \to K(\A (d)\text{-mod})
\] is injective.

\end{lem}

For convenience of notation, set
\[
[a_1,\ldots,a_d]=[\L(a_1)\circledast\cdots\circledast \L(a_d)].
\]
The following lemma describes how to calculate the character of $M \circledast
N$ in terms of the characters of $M$ and $N$, and is a special case of the
Mackey Theorem:

\begin{lem}\label{L:ShuffleLemma}\cite[Shuffle Lemma]{kl} Let $K\in\ASe(k)$
and $M\in\ASe(m)$ be simple, and assume that
\[
\ch K=\sum_{\ui\in\Z_{\geq0}^k}r_{\ui}[i_1,\ldots,i_k]\andeqn \ch
M=\sum_{\uj\in\Z_{\geq0}^m}s_{\uj}[j_1,\ldots,j_m].
\]
Then,
\begin{eqnarray*}
\ch\ind_{m,k}^{m+k}K\circledast M
=\sum_{\ui,\uj}r_{\ui}s_{\uj}[i_1,\ldots,i_k]*[j_1,\ldots,j_m]
\end{eqnarray*}
where
\[
[i_1,\ldots,i_k]*[i_{k+1},\ldots,i_{k+m}] =\sum_{w\in
D_{(m,k)}}[w(i_1),\ldots,w(i_{k+m})].
\]
\end{lem}

\subsection{Duality}\label{SS:duality}
Now, given an $\ASe(d)$-module $M$, we obtain a new module $M^\sm$ by twisting
the action of $\ASe(d)$ by $\sm$. That is, define a new action, $*$, on $M$ by
$x*m=\sm(x).m$ for all $x\in\ASe(d)$. We have

\begin{lem}\label{sm twisted action}\cite[Lemma 14.6.1]{kl} If $M$ is an
$\ASe(k)$-module and $N$ is an $\ASe(\ell)$-module, then
\[
(\ind_{k,\ell}^{k+\ell}M\circledast N)^\sm
\cong\ind_{k,\ell}^{k+\ell}M^\sm\circledast N^\sm.
\]
\end{lem}

If $M$ is an $\ASe(d)$-module, with character
\[
\ch M=\sum_{\ui\in\Z_{\geq0}^d}r_{\ui}[i_1,\ldots,i_d],
\]
then Lemma ~\ref{sm twisted action} implies that
\[
\ch M^{\sm}=\sum_{\ui\in\Z_{\geq0}^d}r_{\ui}[i_d,\ldots,i_1].
\]

\subsection{Contravariant Forms}\label{SS:contravariantforms}
Let $M$ be in $\Rep\ASe(d)$. A bilinear form $(\cdot,\cdot):M\otimes
M\rightarrow\C$ is called a contravariant form if
\[
(x.v,v')=(v,\tau(x).v')
\]
for all $x\in\ASe(d)$ and $v,v'\in M$.

\begin{lem}\label{L:ASeContraForm} Let $M$ be in $\Rep\ASe(d)$
equipped with a contravariant form $(\cdot,\cdot)$. Then
\[
M_\eta\perp M_\zt^{\mathrm{gen}}\;\;\;\mbox{unless}\;\;\;\eta=\zt.
\]
\end{lem}

\begin{proof} Assume $\eta\neq\zt$, and let $v\in M_\eta$ and $v'\in M^{\mathrm{gen}}_\zt$.
Choose $i$ such that $q(\eta(x_i^2))\neq q(\zt(x_i^2))$, and $N\gg0$ such that
\[
(x_i^2-q(\zt(x_i^2))^N.v'=0.
\]
Then
\begin{align*}
(q(\eta(x_i^2))-q(\zt(x_i^2))^N(v,v')
=&((x_i^2-q(\zt(x_i^2)))^N.v,v')\\
=&(v,\tau((x_i^2-q(\zt(x_i^2)))^N).v')\\
=&(v,(x_i^2-q(\zt(x_i^2)))^N.v')=0
\end{align*}
showing that $(v,v')=0$.
\end{proof}

\subsection{Intertwiners} Define the intertwiner
\begin{eqnarray}\label{E:intertwiner}
\phi_i=s_i(x_i^2-x_{i+1}^2)+(x_i+x_{i+1})-c_ic_{i+1}(x_i-x_{i+1}).
\end{eqnarray}
Given an $\ASe(d)$-supermodule $M$, we understand that
$\phi_i\M_{\zeta}^{\mathrm{gen}}\subseteq\M_{s_i(\zeta)}^{\mathrm{gen}}$.
Moreover, a straightforward calculation gives
\begin{eqnarray}\label{E:intertwinersquared}
\phi_i^2=2x_i^2+2x_{i+1}^2-(x_i^2-x_{i+1}^2)^2.
\end{eqnarray}
The following lemma now directly follows (see also \cite{kl}).

\begin{lem}\label{L:InvertibleIntertwiner} Assume that $Y$ is in $\Rep\ASe(d)$,
and $v\in Y$ satisfies $x_i.v=\sqrt{q(a)}v$ and $x_{i+1}.v=\sqrt{q(b)}v$ for
some $a,b\in\Z$. Then, $\phi_i^2.v\neq 0$ unless $q(a)=q(b+1)$ or
$q(a)=q(b-1)$.
\end{lem}

\section{Standard Modules}\label{S:standardreps}
We construct a family of standard modules which are an analogue of Zelevinsky's
construction for the degenerate affine Hecke algebra.  The key ingredient is to
define certain irreducible supermodules for a parabolic subalgebra of $\ASe
(d)$; the so-called segment representations.  The standard modules are then
obtained by inducing from the outer tensor product of these modules.

\subsection{Segment Representations}\label{subsection irred modules}
We begin by constructing a family of irreducible $\ASe(d)$-supermodules that
are analogues of Zelevinsky's segment representations for the degenerate affine
Hecke algebra. To begin, define the $2^d$-dimensional $\Se(d)$-supermodule
\begin{equation}\label{E:Cldef}
\Cl_{d}=\ind_{S_d}^{\Se(d)}\C\va,
\end{equation}
where $\C\va$ is the trivial representation of $S_d$. That is,
$\Cl_{d}=\Cl(d).\va$, where the cyclic vector $\va$ satisfies
\begin{eqnarray*}
w.\va=\va,\;\;\;w\in S_d.
\end{eqnarray*}
This is often referred to as the \emph{basic spin representation} of $\Se(d)$.

Introduce algebra involutions $\epsilon_i:\Cl(d)\rightarrow\Cl(d)$ by
$\epsilon_i(c_j)=(-1)^{\dt_{ij}}c_j$ for $1\leq i,j\leq d$. The elements $
\epsilon_i $ act on $\Cl_{d}$ by $\epsilon_i.\va=\va$ for $1\leq i\leq d$ and, more generally, $\epsilon_i.s\va=\epsilon_{i}(s)\va$ for $1\leq i\leq d$.
Also, note that the operators $ \epsilon_i $ commute with each other.

For each $a\in\Z$, define the Clifford algebra

\begin{equation}\label{Pha}
\Ph_a= \begin{cases} \C\langle \ph \rangle / (\ph^2-a),
                &\text{if $a \neq 0$}; \\
                \C \langle \ph \rangle / (\ph),
                & \text{if $a=0$}.\end{cases}
\end{equation}
The $\Z_{2}$-grading on $\Ph_{a}$ is given by declaring $p(\ph)=\one$.

Given a pair of integers $a\leq b$ define the \emph{segment}
\begin{equation*}
[a,b]=\{a,a+1,\ldots,b\}.
\end{equation*}
Given a segment $[a,b]$ with $b-a+1=d\in\Z_{\geq0}$, define the
$\Ph_a\otimes {\Se}(d)$-module
\begin{equation}\label{E:segment}
\hat{\Ph}_{[a,b]}=\Ph_a\boxtimes\Cl_{d}.
\end{equation}
 Of course, when $d=0$ the segment
$[a,a-1]=\emptyset$, and $\hat{\Ph}_{\emptyset}=\Ph_a\otimes\C$.

For $i=1, \dotsc ,d$ let $s_{ij}$ denote the transposition $(ij)$, and
\begin{align}\label{E:JMelt}
\L_i=\sum_{j<i}(1-c_jc_i)s_{ij}
\end{align}
be the \emph{$i$th Jucys-Murphy element} (cf. \cite[(13.22)]{kl}).

\begin{prp}\label{segment representation}  Let $[a,b]$ be a segment with $b-a+1=d.$ Then,
\begin{enumerate}
\item[(i)] The vector space
$\hat{\Ph}_{[a,b]}$ is an $\ASe(d)$-module with $s_i.v=(1\otimes s_i).v$,
$c_i.v=(1\otimes c_i).v$ and
\begin{align*}
x_i.v &= \left(a\otimes \epsilon_i+1\otimes \L_{i}-\ph\otimes c_i\right).v \\
&=\left(a\otimes \epsilon_i+\sum_{k<i}1\otimes(1-c_kc_i)s_{ki}-\ph\otimes
c_i\right).v,
\end{align*}
for all $v\in\hat{\Ph}_{[a,b]}$.
\item[(ii)]  The action of $\P_d[x^2]$ on $\hat{\Ph}_{[a,b]}$ is determined by
\[
x_i^2.(\ph^{\dt}\otimes \va)=q(a+i-1)\ph^\dt\otimes\va,\;\;\;\dt\in\{0,1\},
\;\;\;i=1,\ldots,d.
\]
\end{enumerate}
\end{prp}

\begin{proof} (i) The fact that this is an $\ASe(d)$-module is an easy check which
we leave to the reader.

(ii) To check the action of $x_i^2$, observe that
\[
x_i.1\otimes\va=\left(a+i-1-\sum_{j<i}c_jc_i\right).1\otimes\va
+c_i.\ph\otimes\va
\]
and
\[
x_i.\ph\otimes\va=\left(a+i-1-\sum_{j<i}c_jc_i\right).\ph\otimes\va
+ac_i.1\otimes\va.
\]
Now, the result follows using the commutation relations for $\ASe(d)$.
\end{proof}

\begin{rmk}\label{R:Duality} In fact, we need not consider all $a,b\in\Z$. Given any
segment $[a,b]$, consider the module $\hat{\Ph}_{[a,b]}^\sm$ obtained by
twisting the action of $\ASe(d)$ by the automorphism $\sm$ as described in
Section~\ref{SS:duality}.  Note that when $b\neq-1$,
\[
\hat{\Ph}_{[a,b]}^\sm\cong\hat{\Ph}_{[-b-1,-a-1]}.
\]
When $b=-1$,
$\hat{\Ph}_{[a,-1]}^\sm\cong\hat{\Ph}_{[0,-a-1]}^{\oplus2}$. In particular, for
$b\neq0$, $\hat{\Ph}_{[-(b+1),b-1]}^\sm\cong\hat{\Ph}_{[-b,b]}$, and
$\hat{\Ph}_{[-1,-1]}^\sm\cong\hat{\Ph}_{[0,0]}^{\oplus2}$. Therefore, it is
enough to describe the modules
\begin{enumerate}
\item $\hat{\Ph}_{[a,b]}$, $0\leq a\leq b$, and
\item $\hat{\Ph}_{[-a,b]}$, $0<a\leq b$.
\end{enumerate}
\end{rmk}
The following result describes $\hat{\Ph}_{[a,b]}$ at the level of characters.

\begin{prp}\label{character formula} Let $[a,b]$ be a segment with $a,b \geq 0.$
Then,
\begin{enumerate}
\item if $0\leq a\leq b$, then
\begin{equation*}
\ch\hat{\Ph}_{[a,b]}=\begin{cases}[a,\ldots,b],
    &\text{if $a=0$};\\ 2[a,\ldots,b], &\text{if $a \neq 0$};\end{cases}
\end{equation*}

\item if $0<a\leq b$, then
\[
\ch\hat{\Ph}_{[-a,b]}=4[a-1,\ldots,1,0,0,1,\ldots,b]
\]
\end{enumerate}
\end{prp}

\begin{proof} The action of $x_i^2$ commutes with $\Cl(d)$ and
$\hat{\Ph}_{[a,b]}=\Cl(d).(1\otimes\va)+\Cl(d).(\ph\otimes\va)$. Therefore,
applying Proposition~\ref{segment representation}(2), we deduce in both cases
that the $x_i^2$ act by the prescribed eigenvalues. The result now follows from
the dimension formula in Lemma \ref{A(d) irreducibles}.
\end{proof}

Let $\ph\hat{\va}_{[a,b]}=\ph\otimes\va$ and $\hat{\va}_{[a,b]}=1\otimes\va$.
Also, in what follows, we omit the tensor symbols. For example, we write
\[
a\epsilon_i+\L_i-\ph c_i:=a\otimes \epsilon_i+1\otimes \L_{i}-\ph\otimes c_i.
\]

\begin{dfn}\label{X} Let $a\in\Z$ and $\kp_1,\ldots,\kp_d\in\R$ satisfy
$\kp_i^2=q(a+i-1)$ where $d=b-a+1$. Given a subset $S\subseteq\{1,\ldots,d\}$
define the element $X_{S} \in \ASe (d)$ by
\[
X_S=\prod_{i\notin S}(x_i+\kp_i).
\]
Observe that $X_S$ is only defined up to the choices of sign for
$\kp_1,\ldots,\kp_d$.
\end{dfn}

\begin{lem}\label{nonzero} Let $[a,b]$ be a segment with $d=b-a+1$. Assume that either
$-a\notin\{1,\ldots,d\}$ and $S$ is arbitrary, or assume that
$-a\in\{1,\ldots,d\}$ and either $-a+1\in S$ or $-a\in S$. Then
$X_S.\hat{\va}_{[a,b]}\neq0$.
\end{lem}

\begin{proof} Let $\hat{\va}=\hat{\va}_{[a,b]}$.  By
Proposition~\ref{segment representation}(i),
\[
x_k.v=(a\epsilon_k+\L_k-\ph c_k).v.
\]
Let $\{d_1>d_2>\ldots>d_\ell\}=\{1,\ldots,d\}\backslash S$. Since the $x_i$
mutually commute,
\begin{eqnarray*}
X_S.\hat{\va}&=&(x_{d_1}+\kp_{d_1})\cdots(x_{d_\ell}+\kp_{d_\ell}).\hat{\va}\\
    &=&(a\epsilon_{d_1}+\kp_{d_1}+\L_{d_1}-\ph c_{d_1})\cdots
    (a\epsilon_{d_\ell}+\kp_{d_\ell}+\L_{d_\ell}-\ph c_{d_\ell}).\hat{\va}\\
    &=&((a+\kp_{d_1})+\L_{d_1}-\ph c_{d_1})\cdots
    ((a+\kp_{d_\ell})+\L_{d_\ell}-\ph c_{d_\ell}).\hat{\va}.
\end{eqnarray*}
The last equality follows since $\epsilon_k\L_j=\L_j\epsilon_k$ if $k>j$. Now,
\begin{eqnarray}\label{w 1}
X_S.\hat{\va} &=&\bigg(\bigg(a+\kp_{d_1}+\sum_{j<d_1}s_{jd_1}\bigg)+
\bigg(\sum_{j<d_1}s_{jd_1}c_j-\ph
\bigg)c_{d_1}\bigg)\cdots\\\nonumber&&\hspace{1.5in}\cdots
    \bigg(\bigg(a+\kp_{d_\ell}+\sum_{j<d_\ell}s_{jd_\ell}\bigg)
    +\bigg(\sum_{j<d_\ell}s_{jd_\ell}c_j-\ph
\bigg)c_{d_\ell}\bigg).\hat{\va}\\\nonumber
    &=&\prod_{i\notin S}(a+i-1+\kp_i).\hat{\va}+(\bigstar).\hat{\va}
\end{eqnarray}
where $(\bigstar)=p'(c)-\ph p''(c)$, where $p'(c)\in\Cl(d)_\zero$,
$p''(c)\in\Cl(d)_\one$, and $p'(c)$ has no constant term. Therefore, if either
$a\geq 0$, or $-a+1\in S$, $X_S.\va\neq 0$.

Now, assume $-a+1\in\{1,\ldots,d\}$, and $-a+1\notin S$, but $a\in S$. Observe
that $\kp_{-a+1}=\kp_{-a}=0$. Now,
\begin{eqnarray}\label{nonzero 2}
x_{-a}.\hat{\va}=\left(-1-\sum_{j<-a}c_jc_{-a}-\ph
c_{-a}\right).\hat{\va}=-c_{-a}c_{-a+1}x_{-a+1}.\hat{\va}.
\end{eqnarray}
Let $R=S\cup\{-a+1\}$ and $T=R\backslash\{-a\}$. Then,
\[
X_S.\hat{\va}=X_{R}x_{-a+1}.\hat{\va}=c_{-a}c_{-a+1}X_{R}x_{-a}.\hat{\va}
=c_{-a}c_{-a+1}X_T.\hat{\va}\neq0.
\]

Finally, if $d=-a$, then in \eqref{w 1}, $d_1=-a$ and it is clear that
the coefficient of $c_{-a-1}c_{-a}$ is nonzero.
\end{proof}

\begin{lem}\label{A submodule}
If $i\notin S$, then $x_iX_S.\hat{\va}=\kp_iX_S.\hat{\va}$.
\end{lem}

\begin{proof} Since $x_{i}^{2}.\hat{\va} = q(a-i+1)\hat{\va}= \kp_{i}^{2}\hat{\va}$,
\begin{eqnarray*}
x_i(x_i+\kp_i).\hat{\va}=(x_i^2+\kp_ix_i)\hat{\va}=\kp_i(\kp_i+x_i)\hat{\va},
\end{eqnarray*}
so the result follows because the $x_i$ commute.
\end{proof}

\begin{lem}\label{s_i action} If $i,i+1\notin S$ and $i\neq-a$, then
\[
s_iX_S.\hat{\va}=\left(\frac{\kp_{i+1}+\kp_i}{2(a+i)}+
    \frac{\kp_{i+1}-\kp_i}{2(a+i)}c_ic_{i+1}\right)X_S.\hat{\va}.
\]
\end{lem}

\begin{proof} Let $w:=X_S.\hat{\va}$, and recall the intertwining element $\phi_i$.
By character considerations $\phi_i.\hat{\Ph}_{[a,b]}=\{0\}$. In particular,
\begin{eqnarray*}
0&=&\phi_i.w\\
    &=&(s_i(x_i^2-x_{i+1}^2)+(x_i+x_{i+1})-c_ic_{i+1}(x_i-x_{i+1})).w\\
    &=&-2(a+i)s_i.w+((\kp_{i+1}+\kp_i)+(\kp_{i+1}-\kp_i)c_ic_{i+1}).w.
\end{eqnarray*}
Hence, the result.
\end{proof}

We can now describe the irreducible segment representations of $\ASe (d).$
\begin{thm}\label{module decomposition} The following holds:
\begin{enumerate}
\item[(i)] The module $\hat{\Ph}_{[0,d-1]}$
is an irreducible $\ASe(d)$-module of type \texttt{Q}.

\item[(ii)] Assume $0<a\leq b$. The module $\hat{\Ph}_{[a,b]}$, has a submodule
$\hat{\Ph}_{[a,b]}^+=\Cl(d).w$, where $w=X_\emptyset.\hat{\va}$. Moreover, if
$w'=(x_1-\kp_1)X_{\{1\}}.\hat{\va}$, and $\hat{\Ph}_{[a,b]}^-=\Cl(d).w'$, then
\[
\hat{\Ph}_{[a,b]}=\hat{\Ph}_{[a,b]}^+ \oplus \hat{\Ph}_{[a,b]}^-.
\] The
submodules $\hat{\Ph}_{[a,b]}^\pm$ are simple modules of type \texttt{M}.

\item[(iii)] If $0<a\leq b$, the $\hat{\Ph}_{[-a,b]}$ has a submodule
$\hat{\Ph}_{[-a,b]}^+=\Cl(d)w\oplus\Cl(d)\overline{w}$, where
\[
w=-(1+\sqrt{-1}c_ac_{a+1})X_{\{a+1\}}.\hat{\va}\andeqn\overline{w}=s_aw.
\]
Moreover, if
\[
w'=-(1-\sqrt{-1}c_ac_{a+1})X_{\{a+1\}}.\hat{\va},\;\;\;\overline{w}'=s_aw',
\]
and $\hat{\Ph}_{[-a,b]}^-=\Cl(d)w'\oplus\Cl(d)\overline{w}'$, then
\[
\hat{\Ph}_{[-a,b]}=\hat{\Ph}_{[a,b]}^+\oplus \hat{\Ph}_{[-a,b]}^-.
\] The
submodules $\hat{\Ph}_{[-a,b]}^\pm$ are simple of type \texttt{M}.
\end{enumerate}
\end{thm}

\begin{proof} (i) First, we deduce that $\hat{\Ph}_{[0,d-1]}$ is irreducible by
character considerations. It has two \emph{non-homogeneous} submodules:
\[
\Cl(d)(\sqrt{-d}+(c_1+\cdots+c_d)).\hat{\va}_{[0,d-1]}\andeqn
\Cl(d)(\sqrt{-d}-(c_1+\cdots+c_d)).\hat{\va}_{[0,d-1]}.
\]
These vector spaces are clearly stable under the action of $\Se(d)$.  Since $
x_1 $ acts by zero on these vector spaces, the action of $\ASe(d)$ factors
through $\Se(d)$ and thus these vector spaces are $\ASe(d)$-submodules.
Therefore $\hat{\Ph}_{[0,d-1]}$ is of type \texttt{Q} (cf.
Section~\ref{S:Prelim}).

(ii) Let $\hat{\va}=\hat{\va}_{[a,b]}$, $w=X_\emptyset.\hat{\va}$ and
$\hat{\Ph}_{[a,b]}^+=\Cl(d).w$. By Lemma \ref{nonzero}, $w\neq 0$. Now, Lemmas
\ref{A submodule} and \ref{s_i action} together imply that
$\hat{\Ph}_{[a,b]}^+$ is a submodule.

It now remains to show that $\hat{\Ph}_{[a,b]}=\hat{\Ph}_{[a,b]}^+\oplus
\hat{\Ph}_{[a,b]}^-$, where $\hat{\Ph}_{[a,b]}^-$ is as in the statement of the
proposition. To this end, assume that $w'\in \hat{\Ph}_{[a,b]}^+$. That is,
there exists $p(c)\in\Cl(d)$ such that $p(c).w=w'$. Write
\[
p(c)=\sum_{\ep}a_\ep c^\ep,
\]
where the sum is over $\ep=(\ep_1,\ldots,\ep_d)\in\Z_2^d$. Then, for $1\leq
i\leq d$,
\begin{eqnarray*}
(-1)^{\dt_{1i}}w'&=&\frac{1}{\kp_i}x_i.w'
    =\frac{1}{\kp_i}x_i\left(\sum_{\ep}a_\ep c^\ep\right).w
    =\left(\sum_{\ep}(-1)^{\ep_i}a_\ep c^\ep\right).w,
\end{eqnarray*}
where (of course) the $\dt$ on the left of the equal sign is the Kronecker
delta. This forces $ p(c)=r c_1 + s $ for complex numbers $ r $ and $ s$. Since
$ w' $ is even, $ r=0 $ implying that $ w'=s w $ which is impossible.

(iii) We deal with $\hat{\Ph}_{[-a,b]}^+$, the proposed submodule
$\Ph_{[-a,b]}^-$ being similar. Let
$w=-(1+\sqrt{-1}c_ac_{a+1})X_{\{a+1\}}.\hat{\va}$, $\overline{w}=s_a.w$, and
$\hat{\Ph}_{[a,b]}^+=\Cl(d).w+\Cl(d).\overline{w}$. The proof of Lemma
\ref{nonzero} shows that
\[
X_{\{a+1\}}.\hat{\va}=\prod_{\substack{1\leq i\leq d\\i\neq
a+1}}(a+i-1+\kp_i).\hat{\va}+(\bigstar).\hat{\va}
\]
where $(\bigstar)=p'(c)-\ph p''(c)$ where $p'(c)\in\Cl(d)_\zero$,
$p''(c)\in\Cl(d)_\one$, and $p'(c)$ has no constant term. It is also easy to
see that $p'(c)$ and $p''(c)$ have coefficients in $\R$. We conclude from this
that $w\neq 0$. Note that by definition, $c_ac_{a+1}.w=-\sqrt{-1}w$.

Lemma \ref{A submodule} shows that for $i\neq a,a+1$, $x_i.w=\kp_iw$. Moreover,
\[
x_a.w=-(1-\sqrt{-1}c_ac_{a+1})x_aX_{\{a+1\}}.\hat{\va}=0.
\]
Also, $x_a.\hat{\va}=-c_ac_{a+1}x_{a+1}.\hat{\va}$ (see the computation
\eqref{nonzero 2} for details). Thus,
\begin{eqnarray}\label{alternate w}
w=-\sqrt{-1}(1+\sqrt{-1}c_ac_{a+1})X_{\{a\}}.\hat{\va}
\end{eqnarray}
so $x_{a+1}.w=0$. As for $\overline{w}=s_aw$,
$x_i.\overline{w}=\kp_i\overline{w}$ for $i\neq a,a+1$. Using commutation
relations, we compute
\begin{eqnarray}\label{x_a}
x_a\overline{w}=x_as_a.w=(s_ax_{a+1}-1-c_ac_{a+1}).w=-(1+\sqrt{-1})w.
\end{eqnarray}
Similarly,
\begin{eqnarray}\label{x_{a+1}}
x_{a+1}.\overline{w}=(1+\sqrt{-1})w.
\end{eqnarray}

We now turn to the action of the symmetric group. First, for $i\neq a-1,a+1$,
Lemma \ref{s_i action} shows that $s_i.w\in\hat{\Ph}_{[a,b]}^+$. Also by Lemma
\ref{s_i action},
\[
s_{a-1}X_{\{a+1\}}.\hat{\va}=\frac{\kp_{a-1}}{2}(c_{a-1}c_a-1)X_{\{a+1\}}.\hat{\va}.
\]
Thus,
\begin{eqnarray*}
s_{a-1}.w&=&-\frac{\kp_{a-1}}{2}(1+\sqrt{-1}c_{a-1}c_{a+1})
        (c_{a-1}c_a-1)X_{\{a+1\}}.\hat{\va}\\
    &=&-\frac{\kp_{a-1}}{2}(1+c_{a-1}c_a+\sqrt{-1}c_{a-1}c_{a+1}
    -\sqrt{-1}c_ac_{a+1})X_{\{a+1\}}.\hat{\va}\\
    &=&\frac{\kp_{a-1}}{2}(c_{a-1}c_a-1).w.
\end{eqnarray*}
Similarly, by \eqref{alternate w} and Lemma \ref{s_i action},
\[
s_{a+1}.w=\frac{\kp_{a+2}}{2}(1+c_{a+1}c_{a+2}).w.
\]
Now, for $i\neq a-1,a+1$, $s_is_a=s_as_i$. Hence, by Lemma \ref{s_i action}
\begin{eqnarray}\label{s_i.overline{w}}
s_i.\overline{w}=\left(\frac{\kp_{i+1}+\kp_i}{2(a+i)}
    +\frac{\kp_{i+1}-\kp_i}{2(a+i)}c_ic_{i+1}\right).\overline{w}.
\end{eqnarray}
To deduce the action of $s_{a-1}$ and $s_a$ on $\overline{w}$, we proceed as in
the proof of Lemma \ref{s_i action}. Recall again the intertwining elements
$\phi_{a-1}$ and $\phi_{a+1}$. By character considerations, we deduce that
$\phi_{a-1}.\overline{w}=0=\phi_{a+1}.\overline{w}$. Unlike in lemma ~\ref{A
submodule}, in this case the action of $x_a$ (resp. $x_{a+1}$) is given by
\eqref{x_a} (resp. \eqref{x_{a+1}}). Thus,
\begin{eqnarray}\label{s_{a-1}.overline{w}}
s_{a-1}.\overline{w}=\frac{(1+\sqrt{-1})}{2}(1+c_{a-1}c_a).w
    -\frac{\kp_{a-1}}{2}(1-c_{a-1}c_a).\overline{w}
\end{eqnarray}
and
\begin{eqnarray}\label{s_a.overline{w}}
s_{a+1}.\overline{w}=\frac{(1-\sqrt{-1})}{2}(1-c_{a+1}c_{a+2}).w
    +\frac{\kp_{a+2}}{2}(1+c_{a+1}c_{a+2}).\overline{w}.
\end{eqnarray}

It is easy to see that
$\hat{\Ph}_{[-a,b]}=\hat{\Ph}_{[-a,b]}^+ + \hat{\Ph}_{[-a,b]}^-$ since
$\frac{1}{2}(w+w')=X_{\{a\}}.\hat{\va}$ is a cyclic vector for
$\hat{\Ph}_{[-a,b]}$.
As in part (ii), it is easy to see that if $ w' = p(c)w+r(c)s_a w $
where $ p(c) $ and $ r(c) $ are polynomials in the Clifford generators, that
$ p(c) = \lambda_1 + \lambda_2 c_a c_{a+1} $ and
$ r(c) = \lambda_3 + \lambda_4 c_a c_{a+1} $ for some complex numbers
$ \lambda_1, \lambda_2, \lambda_3, \lambda_4$.
Noting that $ c_a c_{a+1} w = -\sqrt{-1} w $ gives that all the coefficients are zero.

Therefore, we are left to show that
$\hat{\Ph}_{[-a,b]}^+$ is simple. Indeed, assume
$V\subseteq\hat{\Ph}_{[-a,b]}^+$ is a submodule. Then,
\[
\ch V=[a-1,\ldots,0,0,\ldots,b].
\]
Let $v=p_1(c).w+p_2(c).\overline{w}\in V$ be a vector satisfying $x_i.v=\kp_iv$
for all $i$, where $p_1(c),p_2(c)\in\Cl(d)$. For $i=1,2$, define $p_i'(c)$ by
the formulae $x_ap_i(c)=p_i'(c)x_a$. Then,
\[
0=x_a.v=-(1+\sqrt{-1})p_2'(c).w
\]
showing that $p_2'(c)=0$ (hence, $p_2(c)=0$). Now, arguing as above with the
vector $s_a.v$ shows that $p_1(c)=0$.
\end{proof}

We can now define the irreducible segment representations which are the key to
defining the standard $\ASe (d)$-modules.
\begin{dfn}\label{segments} Let $a,b \in \Z_{\geq 0}$.
\begin{enumerate}
\item Let $\Ph_{[0,d-1]}=\hat{\Ph}_{[0,d-1]}$,
$\va:=X_{\{1\}}.\hat{\va}$, where $\kp_i=\sqrt{q(i-1)}$.
\item If $0 < a\leq b$, let $\Ph_{[a,b]}=\hat{\Ph}_{[a,b]}^+$ in Proposition
\ref{module decomposition}(ii), with $\kp_i=+\sqrt{q(a+i-1)}$ for all $i$, and
let $\va:=w$.
\item If $0<a\leq b$, let $\Ph_{[-a,b]}=\hat{\Ph}_{[-a,b]}^+$ with
$\kp_i=+\sqrt{q(-a+i-1)}$, $\va:=w$ and $\overline{\va}:=\overline{w}$.
\item If $0\leq a$, let $\Ph_{[a,a-1]}=\Ph_\emptyset=\C$.
\end{enumerate}
\end{dfn}

\subsection{Some Lie Theoretic Notation}\label{SS:LieThy} It is convenient in this
section to introduce some Lie theoretic notation. This section differs from
\cite{kl} in that the notation defined here is associated to the Lie
superalgebra $\q(n)$ (as opposed to the Kac-Moody algebra $\b_\infty$).

Define the sets $P=\Z^n$, $P_{\geq0}=\Z^n_{\geq0}$, and
\begin{eqnarray}
\label{dom wt}P^+&=&\{\,\ld=(\ld_1,\ldots,\ld_n)\in
P\,|\,\ld_i\geq\ld_{i+1}\mbox{ for all
}1\leq i\leq n\,\}\\
\label{dom typ wt}\Pt&=&\{\,\ld\in P^+\,|\,\ld_i+\ld_j\neq0\mbox{ for all }
1\leq i,j\leq n\,\}\\
\label{rat wt}\Pr&=&\{\,\ld\in P^+\,|\,\ld_i=\ld_{i+1}\mbox{ implies }\ld_i=0\,\}\\
\label{poly wt}\Pp&=&\{\,\ld\in\Pr\,|\,\ld_n\geq 0\,\}\\
\label{pos wt}\Ppos&=&\{\ld\in P\,|\,\ld_i\geq0\mbox{ for all }i\,\},
\end{eqnarray}
The weights \eqref{dom wt} are called dominant, and \eqref{dom typ wt} are called dominant typical. A weight $\ld\in P$ is simply \emph{typical} if $\ld_i+\ld_j\neq0$ for all $i,j$. The weights \eqref{rat wt} are called rational, \eqref{poly wt} are polynomial, and the set \ref{pos wt} are simply compositions. For each of the sets $X=P^+,P^{++},\Pr,\Pp,\Ppos$ above, define
\[
X(d)=\{\ld\in X|\ld_1+\cdots+\ld_n=d\}.
\]
Let $R\subset P$ be the root system of type $A_{n-1}$. That is,
$R=\{\af_{ij}\mid 1\leq i\neq j\leq n\}$ where $\af_{ij}$ is the $n$-tuple with 1
in the $i$th coordinate and $-1$ in the $j$th coordinate. The positive roots
are $R^+=\{\af_{ij}\in R \mid i<j\}$, the root lattice $Q$ is the $\Z$-span of
$R$, and $Q^+$ is the $\Z_{\geq 0}$-span of $R^+$. The symmetric group, $S_n$, acts on
$P$ by place permutation. Define the length function
$\ell:S_n\rightarrow\Z_{\geq0}$ in the usual way:
\[
\ell(w)=|\{\af\in R^+\mid w(\af)\in-R^+\}|.
\]
Equivalently, $\ell(w)$ is the number of simple transpositions occurring in a
reduced expression for $w$. Write $w\rightarrow y$ if $y=s_\af w$ for some
$\af\in R^+$ and $\ell(w)<\ell(y)$. Define the \emph{Bruhat} order on $S_n$ by
$w<_by$ if there exists a sequence $w\rightarrow
w_1\rightarrow\cdots\rightarrow y$. Also, for $\ld\in P$, define
\[
S_n[\ld]=\{\,w\in S_n \mid w(\ld)=\ld\,\},\andeqn R[\ld]=\{\,\af_{ij}\in
R|\,s_{ij}(\ld)=\ld\,\},
\]
and define
\[
P^+[\ld]=\{\,\mu\in P\,|\,\mu_i\geq\mu_j\mbox{ if }s_{ij}\in
S_n[\ld]\,\},\andeqn P^-[\ld]=\{\,\mu\in P\,|\,\mu_i\leq\mu_j\mbox{ if
}s_{ij}\in S_n[\ld]\,\}
\]
where $s_{ij}\in S_n$ denotes the transposition $(ij)$.

\subsection{Induced Modules}\label{SS:inducedmodules} Using the irreducible segment representations defined above we now define standard representations. Let $\ld,\mu\in P$ satisfy
$\ld-\mu\in\Ppos(d)$. Define
\[
\widehat{\Ph}(\ld,\mu)
=\hat{\Ph}_{[\mu_1,\ld_1-1]}\boxtimes\cdots\boxtimes\hat{\Ph}_{[\mu_n,\ld_n-1]} \]
and
\[
\Ph(\ld,\mu)=\Ph_{[\mu_1,\ld_1-1]}\circledast\cdots\circledast\Ph_{[\mu_n,\ld_n-1]},
\]
and define \emph{standard (cyclic) modules} for $\ASe(d)$ by
\begin{equation}\label{E:Mhatdef}
\widehat{\M}(\ld,\mu)=\ind_{d_1,\ldots,d_n}^d\widehat{\Ph}(\ld,\mu)
\end{equation}
and
\begin{equation}\label{E:Mdef}
\M(\ld,\mu)=\ind_{d_1,\ldots,d_n}^d\Ph(\ld,\mu).
\end{equation}
We call the standard modules $\widehat{\M}(\ld,\mu)$ and $\M(\ld,\mu)$ \emph{big}
and \emph{little}, respectively.

Both the big and little standard modules are cyclic.  Let
\begin{align}\label{E:hatcyclicvector}
\hat{\va}_{\ld,\mu}=1\otimes(\hat{\va}\otimes\cdots\otimes{\hat{\va}})
    \in\widehat{\M}(\ld,\mu)
\end{align}
be the distinguished cyclic generator of $\widehat{\M}(\ld,\mu)$. Fix the following choice of distinguished cyclic
generator $\va_{\ld,\mu}\in\M(\ld,\mu)$. Let $i_1<\cdots<i_k$ be such that $\mu_{i_j}=0$ for all $j$ and $\gm_0(\mu)=k$. Choose
\[
\va_{\ld,\mu}=\prod_{j=1}^{\lfloor k/2\rfloor} (1-\sqrt{-1}c_{i_{2j-1}}c_{i_{2j}})1\otimes(\va\otimes\cdots\otimes\va).
\]

\begin{lem}\label{L:standard cyclic dim} Let $\lambda, \mu \in P$ so that $\lambda - \mu \in P_{\geq 0}(d).$  Then,
\begin{enumerate}
\item[(i)]  $\dim\widehat{\M}(\ld,\mu)
        =\frac{d!}{d_1!\cdots d_n!}2^{d+n-\gm_0(\mu)}$
\item[(ii)] $\dim\M(\ld,\mu)
        =\frac{d!}{d_1!\cdots d_n!}2^{d-\lfloor\frac{\gm_0(\mu)}{2}\rfloor}$
\item[(iii)] $\widehat{M}(\ld,\mu)\cong\M(\ld,\mu)^{\oplus 2^{n-\lfloor\frac{\gm_0(\mu)+1}{2}\rfloor}}$.
\end{enumerate}
\end{lem}

\begin{proof}(i) The dimension of $\widehat{\M}(\ld,\mu)$ follows from the definition.

(ii) Use Proposition \ref{module decomposition}.

(iii) Since induction commutes with direct sums we have that  $\widehat{\M}(\ld,\mu)$ is a direct sum of copies $\M(\ld,\mu)$.  A count using (i) and (ii) yields (iii).
\end{proof}

We end this section by recording certain data about the weight spaces and generalized weight spaces of $\M (\lambda, \mu)$ which will be useful later.  Define the weight $\zt_{\ld,\mu}:\P_d[x^2]\rightarrow\C$ by
$f.\va_{\ld,\mu}=\zt_{\ld,\mu}(f)\va_{\ld,\mu}$ for all $f\in\P_d[x]$. As in $\S$\ref{SS:LieThy}, the symmetric group, $S_d$, acts on an integral weight $\zt:\P_d[x^2]\rightarrow\C$ by $w(\zt)(x_i^2)=\zt(x_{w(i)}^2)$. Let
\[
S_d[\zt]=\{\,w\in S_d\,|\,w(\zt)=\zt\,\}.
\]
Define $\ell(w)$ to be the length of $w$ (i.e.\ the number of simple transpositions
occurring in a reduced expression of $w$) and recall the definition of the Bruhat order given in section~\ref{SS:LieThy}.

\begin{lem}\label{L:weights of M} Given $\ld,\mu\in P$ with $\ld-\mu\in\Ppos(d)$,
\begin{enumerate}
\item[(i)] $P(\M(\ld,\mu))=\{\,w(\zt_{\ld,\mu})\,|\,w\in D_{\ld-\mu}\,\}$,
\item[(ii)] For any $\zt\in P(\M(\ld,\mu))$,
\[
\dim\M(\ld,\mu)_\zt^{\mathrm{gen}}=2^{d-\lfloor\frac{\gm_0(\mu)}{2}\rfloor}
    |\{\,w\in D_{\ld-\mu}\,|\,w(\zt)=\zt\,\}|.
\]

In particular,
\[
\dim\M(\ld,\mu)_{\zt_{\ld,\mu}}^{\mathrm{gen}}=
    2^{d-\lfloor\frac{\gm_0(\mu)}{2}\rfloor}\big|D_{\ld-\mu}\cap
    S_d[\zt_{\ld,\mu}]\big|.
\]
\end{enumerate}
\end{lem}

\begin{proof} (i) This follows directly upon applying the Mackey Theorem to the character map.

(ii) Given $f\in\P_d[x^2]$ and $w\in S_d$, we have the relation
\[
fw=w\cdot w^{-1}(f)+\sum_{u<_{b}w}uC_uf_u
\]
where the sum is over $u<_{b}w$ in the Bruhat order, $C_u\in\Cl(d)$,
$f_u\in\P_d[x]$ and $\deg f_u<\deg f$, see \cite[Lemma 14.2.1]{kl}. Therefore, if $f\in\P_d[x^2]$, $C\in\Cl(d)$ and $w\in
D_{\ld-\mu}$,
\begin{align}\label{E:lowerTriangular}
f(wC.\va_{\ld,\mu})=w(\zt_{\ld,\mu})(f)wC.\va_{\ld,\mu}+\sum_{u<_{b}w}uC_uf_u.\va_{\ld,\mu}
\end{align}
where the sum is over $u\in D_{\ld-\mu}$. In particular, $wC.\va_{\ld-\mu}\in\M(\ld,\mu)_{\zt_{\ld,\mu}}^{\mathrm{gen}}$  only if $w\in D_{\ld-\mu}\cap S_d[\zt_{\ld,\mu}]$. Conversely, if $w\in D_{\ld-\mu}\cap S_d[\zt_{\ld,\mu}]$, it is straightforward to see that all $u$ occurring on the right hand side of \eqref{E:lowerTriangular} also belong to $D_{\ld-\mu}\cap S_d[\zt_{\ld,\mu}]$.  This gives the result.
\end{proof}

\subsection{Unique Simple Quotients}\label{unique simple quotient}
In general, the standard cyclic module $\M(\ld,\mu)$ may not have a unique
simple head. However, in this subsection, we determine sufficient conditions
for this to hold. Throughout this section, keep in mind that $q(a)=q(-a-1)$ for all $a\in\Z$.
We follow closely the strategy in \cite{su2}.
We begin with some preparatory lemmas.

\begin{lem}\label{L:x weights} Let $M$ be an $\ASe(d)$-module,
and $\zt$ a weight of $M$, then there exists $v\in\M(\ld,\mu)_{\zt}$ such that
\[
x_i.v=\sqrt{q(\zt(x_i^2))}\;v
\]
for all $i=1,\ldots,d$.
\end{lem}

\begin{proof}
Choose $0\neq v_0\in M_\zt$. Recall the definition \ref{X}. We adapt this to
our current situation by setting $\kp_i=\sqrt{q(\zt(x_i^2))}$ and
$S=\{i \mid x_iv=-\kp_iv \}$. Then, $v_1:=X_S.v_0\in\M_\zt$ is nonzero and
$x_i.v_1=\pm\kp_iv_1$ for all $i$. Now, set
\[
v=\left(\prod_{i\in S}c_i\right)v_1.
\]
Then, $v$ is nonzero and has the desired properties.
\end{proof}

Therefore, we may define the non-zero subspace
\[
M_{\sqrt{\zt}}= \left\{\,m\in M_\zt \mid x_i.m=\sqrt{q(\zt(x_i^2))}\;m\mbox{ for
}i=1,\ldots,d\,\right\}.
\]

We will use the following key lemma repeatedly in this section.

\begin{lem}
\label{techlemma} Let $ Y $ be in $\Rep\ASe(d)$ and $v \in Y_{\sqrt{\zt}}$ for
some weight $\zt$. Assume that for some $1\leq i<d-1$, $x_i.v=\sqrt{q(a)}$,
$x_{i+1}=\sqrt{q(b)}$ where $a,b\in\Z$ and either $q(a)\neq0$ or $q(b)\neq 0$.
Further, if $q(a)=q(b\pm1)$, assume that
\begin{align}\label{E:techlemma}
s_{i+1}.v =(\kp_1+\kp_2c_{i+1}c_{i+2}).v
\end{align}
for some constants $\kp_1,\kp_2\in\C$, not both 0. Then,
$v\in\ASe(d).\phi_i.v$.
\end{lem}

\begin{proof}
First, if $q(a)=q(b)\neq0$, then using \eqref{E:intertwiner} and Lemma 14.8.1 of \cite{kl} we deduce that
\[
\phi_i.v=2q(a)v\neq0,
\]
so the result is trivial. If $q(a)\neq q(b\pm1)$, then using
\eqref{E:intertwinersquared} we deduce that
\[
\phi_i^2.v=(2q(a)-2q(b)-(q(a)-q(b))^2)v\neq0
\]
and again the result is trivial.

Now, let $\kp_3=q(a)-q(b)\neq0$, $\kp_4=\sqrt{q(a)}-\sqrt{q(b)}\neq0$ and
$\kp_5=\sqrt{q(a)}+\sqrt{q(b)}>0$. Then, appealing again to
\eqref{E:intertwiner} we have that
\[
\phi_{i} v = (\kp_3 s_{i} - \kp_4 c_{i} c_{i+1} + \kp_5)v \]

Let $ \mathbf{c'}$ and $ \mathbf{c''} $ be two elements of the Clifford
algebra. Consider an expression of the form
\begin{align*}
(1+ \mathbf{c'} s_{i+1}- \mathbf{c''} s_{i}s_{i+1})\phi_i v =&(\kp_3 s_{i} -
\kp_4 c_{i}c_{i+1} + \kp_5 + \kp_3 \mathbf{c'}
    s_{i+1}s_{i}\\
&- \kp_4 \mathbf{c'} c_{i}c_{i+2}s_{i+1}+
    \kp_5 \mathbf{c'} s_{i+1}- \kp_3 \mathbf{c''} s_{i+1}s_{i}s_{i+1}\\
&+ \kp_4 \mathbf{c''} c_{i+1}c_{i+2}s_{i}s_{i+1}-\kp_5 \mathbf{c''}
s_{i}s_{i+1}) v.
\end{align*}
By \eqref{E:techlemma}, this equals
\begin{align*}
(\kp_3& s_{i}-\kp_4 c_{i}c_{i+1}+\kp_5+\kp_3 \mathbf{c'} s_{i+1}s_{i}
- \kp_1 \kp_4 \mathbf{c'} c_{i}c_{i+2}\\
&-\kp_2\kp_4 \mathbf{c'} c_{i}c_{i+1}+ \kp_1 \kp_5 \mathbf{c'} + \kp_2
\kp_5 \mathbf{c'} c_{i+1}c_{i+2} - \kp_1\kp_3 \mathbf{c''} s_{i+1}s_{i}\\
&-\kp_2\kp_3 \mathbf{c''} c_{i}c_{i+1} s_{i+1}s_{i} + \kp_1 \kp_4 \mathbf{c''}
c_{i+1}c_{i+2}s_{i} - \kp_2 \kp_4 \mathbf{c''}
c_{i}c_{i+1} s_{i}\\
&-\kp_1\kp_5 \mathbf{c''} s_{i} -\kp_2\kp_5 \mathbf{c''} c_{i}c_{i+2}s_{i})v.
\end{align*}

The coefficient of $ s_i v $ is
$$ \kp_3 + \kp_1 \kp_4 \mathbf{c''} c_{i+1}c_{i+2}
    - \kp_2 \kp_4 \mathbf{c''} c_{i}c_{i+1}
    - \kp_1 \kp_5 \mathbf{c''}
    - \kp_2 \kp_5 \mathbf{c''} c_{i}c_{i+2}. $$
The coefficient of $ s_{i+1}s_{i}v $ is
$$ \kp_3 \mathbf{c'} - \kp_1 \kp_3 \mathbf{c''}
    - \kp_2 \kp_3 \mathbf{c''} c_{i}c_{i+1}. $$
In order to make both of these coefficients zero, set $ \mathbf{c'} =
\mathbf{c''}(\kp_1+\kp_2c_{i}c_{i+1}) $ and
$$ \mathbf{c''} = \gamma(\kp_1 \kp_5 +\kp_1 \kp_4 c_{i+1}c_{i+2}
    -\kp_2 \kp_4 c_{i}c_{i+1} -\kp_2 \kp_5 c_{i}c_{i+2}), $$
where $$ \gamma = \frac{-\kp_3}{(\kp_1^2+\kp_2^2)(\kp_4^2+\kp_5^2)}.
$$

The coefficient of $ v $ is
\begin{align*}
-\kp_4 c_{i}c_{i+1}&+\kp_5-\kp_1\kp_4 \mathbf{c'} c_{i}c_{i+2}
    - \kp_2 \kp_4 \mathbf{c'} c_{i}c_{i+1} + \kp_1 \kp_5 \mathbf{c'}
    + \kp_2 \kp_5 \mathbf{c'} c_{i+1} c_{i+2}\\
=& -\kp_4 c_{i}c_{i+1} + \kp_5 -\kp_1 \kp_4 \mathbf{c''}(\kp_1c_{i}c_{i+2} +
\kp_2 c_{i+1}c_{i+2})
-\kp_2\kp_4 \mathbf{c''}(\kp_1 c_{i}c_{i+1} - \kp_2)\\
    &+ \kp_1\kp_5 \mathbf{c''}(\kp_1+\kp_2c_{i}c_{i+1})
    + \kp_2 \kp_4 \mathbf{c''}(\kp_1 c_{i+1}c_{i+2}
    - \kp_2 c_{i}c_{i+2}).
\end{align*}

This is equal to
\begin{align*}
\kp_5 - \kp_4& c_{i}c_{i+1} + (-\kp_1\kp_2\kp_4 + \kp_1 \kp_2 \kp_5)
\mathbf{c''} c_{i}c_{i+1}
+ (-\kp_1^2\kp_4 - \kp_2^2\kp_5)\mathbf{c''} c_{i}c_{i+2}\\
&+(-\kp_1 \kp_2 \kp_4 + \kp_1 \kp_2 \kp_5)\mathbf{c''} c_{i+1}c_{i+2} +
(\kp_2^2 \kp_4 + \kp_1^2 \kp_5) \mathbf{c''}\\
= &\kp_5 - \kp_4 c_{i}c_{i+1}+(\kp_1 \kp_2 \kp_5 - \kp_1 \kp_2 \kp_4)
\gamma(-\kp_1 \kp_5 c_{i}c_{r+p}
    - \kp_1 \kp_4 c_{i}c_{i+2} - \kp_2 \kp_4
    - \kp_2 \kp_5 c_{i+1}c_{i+2})\\
&+(-\kp_1^2 \kp_4 - \kp_2^2 \kp_5) \gamma(-\kp_1 \kp_5 c_{i}c_{i+2}
    + \kp_1 \kp_4 c_{i}c_{i+1} - \kp_2 \kp_5
    + \kp_2 \kp_4 c_{i+1}c_{i+2})\\
&+(-\kp_1 \kp_2 \kp_4 + \kp_1 \kp_2 \kp_5) \gamma(-\kp_1 \kp_5 c_{i+1}c_{i+2}
    - \kp_2 \kp_4c_{i}c_{i+2} + \kp_1 \kp_4
    + \kp_2 \kp_5 c_{i}c_{i+1})\\
&+(\kp_2^2 \kp_4 + \kp_1^2 \kp_5) \gamma(-\kp_1 \kp_4 c_{i+1}c_{i+2}
    + \kp_2 \kp_4 c_{i}c_{i+1} - \kp_1 \kp_5
    + \kp_2 \kp_5 c_{i}c_{i+2})\\
=& \kp_5 + \dt_1 c_{i}c_{i+1} + \dt_2 c_{i+1}c_{i+2} + \dt_3 c_{i}c_{i+2}
\end{align*}
for some constants $ \dt_1, \dt_2, \dt_3\in\R. $

Thus,
\begin{align*}
(\kp_5 -\dt_1 c_{i}c_{i+1}-\dt_2 c_{i+1} c_{i+2}
    -\dt_3 c_{i}c_{i+2})&(1 + \mathbf{c'} s_{i+1}
    - \mathbf{c''} s_{i}s_{i+1})\phi_i v\\
    =& (\kp_5+\dt_1^2+\dt_2^2+\dt_3^2) v.
\end{align*}
Since $ \dt_1^2, \dt_2^2, \dt_3^2 \in \R_{\geq 0} $ and $ \kp_5 > 0, $ the
result follows.
\end{proof}

\begin{prp}\label{dominant wt space} Assume that $\ld\in\Pt$,
$\mu\in P^+[\ld]$, and $\ld-\mu\in\Ppos(d)$. Then,
\[
\M(\ld,\mu)_{\sqrt{\zt_{\ld,\mu}}}=\C\va_{\ld,\mu}.
\]
\end{prp}

We begin by proving a special case of the Proposition. Suppose $ n $ divides $
d $, and $d/n=b-a$ for some $a,b\in\Z$, $b>0$. Let $ \lambda = (b, \ldots, b) $
and $ \mu = (a, \ldots, a) $ be weights of $ \q(n). $  Set $ \M_{a,b,n} =
\M(\lambda, \mu), $ and set $\va_{a,b,n}=\va_{\ld,\mu}$. Let
$$ \zt_{a,b,n} = (a,a+1, \ldots, b-1, \ldots, a, a+1, \ldots, b-1) $$
be a weight for $ \ASe(d) $ where the sequence $ a, a+1, \ldots, b-1 $ appears
$ n $ times.

The first goal is to compute the weight space $(\M_{a,b,n})_{\sqrt{\zt_{a,b,n}}}$.

Set $n=d$ in the definition above so that $b=a+1$. The resulting module is the
Kato module $K(a, \ldots, a)=K_a$, where all the $ x_i^2 $ act by $q(a)$ on the
vector $\va_{a,b,n}$.

The following  is \cite[Lemma 16.3.2, Theorem 16.3.3]{kl}.

\begin{lem}
\label{katolemma}
\begin{enumerate}
\item If $a\neq-1$ or $0$, the weight space of $ K(a, \ldots, a) $
corresponding to $ (a, \ldots, a) $ with respect to the operators $ x_1^2, \ldots, x_n^2 $
has dimension $ 2^n. $
If $a=-1$ or $0$,
then the weight space of $K(a,\ldots,a)$ corresponding to $(a,\ldots,a)$ with respect to the
operators $ x_1, \ldots, x_n $ has
dimension $2^{\lfloor\frac{n+1}{2}\rfloor}$.
\item The module $ K(a, \ldots, a) $ is equal to its generalized weight space for
the weight $ (a, \ldots, a). $
\item The module $ K(a, \ldots, a) $ is simple of type \texttt{Q} if $ a=0 $ and $ d $
is odd, and is of type \texttt{M} otherwise.
\end{enumerate}
\end{lem}

Set $m=d/n$. In the set of weights of $ \M_{a,b,n}, $ there exists a unique
anti-dominant weight $ \zt_{a,b,n}^{\circ} $ that is given by
$$ \zt_{a,b,n}^{\circ}
    = (\underbrace{a, \ldots, a,}_n \underbrace{a+1,
    \ldots, a+1,}_n \ldots, \underbrace{b-1,
    \ldots, b-1}_n). $$

Take an element $ \tau\in D_{\lambda-\mu} $ such that $ \tau(\zt_{a,b,n}) =
\zt_{a,b,n}^{\circ}. $ If $a\geq 0$, it is given by $ \tau = \om^1 \cdots
\om^{m-1}, $ where $ \om^p=\rho_{n-1}^p \rho_{n-2}^p \cdots \rho_1^p $,
$$
\rho_k^p = \xi_{k(p+1)-(k-1)}^p \cdots \xi_{(k(p+1)-1)}^p \xi_{k(p+1)}^p,
$$
and, for $1 \leq r \leq d-1, $ and $1 \leq p \leq d-r$, $\xi_r^p = s_{r+p-1}
\cdots s_{r+1}s_r$.

If $b\leq0$, then $\tau=\sm(\om^1\cdots\om^{m-1})$, where $\sm$ is the
automorphism of $\ASe(d)$. Finally, if $a<0$ and $b>0$,
$\tau=\sm_{(-a+1)n}(\om^2\cdots\om^{-a})\om^{-a+1}\cdots\om^{m-1}$, where
$\sm_{-a}$ is the automorphism of $\ASe(-a)\subseteq\ASe(d)$ embedded on the
left.

\begin{lem}\label{cyclicvectorlemma} The vector $ \phi_{\tau} \va_{a,b,n} $ is a cyclic
vector of $ \M_{a,b,n}. $
\end{lem}

\begin{proof}
This follows from iterated applications of lemma ~\ref{techlemma}.
\end{proof}

The proof of the following lemma is similar to \cite[Lemma A.7]{su2},
substituting Lemmas ~\ref{katolemma} and ~\ref{L:weights of M} appropriately
into Suzuki's argument.

\begin{lem}
\label{antidominantlemma} $(\M_{a,b,n})_{\sqrt{\zt_{a,b,n}^{\circ}}} \subseteq
\phi_{\tau} \Cl(d) \va_{a,b,n}. $
\end{lem}

\begin{proof}
By an argument similar to the proof of \cite[Lemma A.7]{su2}, we deduce that
\[
(\M_{a,b,n})_{\zt_{a,b,n}^{\circ}} \cong (K_a)_{a^{(n)}} \circledast
(K_{a+1})_{(a+1)^{(n)}} \circledast \cdots \circledast
(K_{b-1})_{(b-1)^{(n)}}
\]
if $a\geq 0$, and
\[
(\M_{a,b,n})_{\zt_{a,b,n}^{\circ}} \cong (K_{-a-1})_{(-a-1)^{(n)}}
\circledast \cdots
(K_1)_{1^{(n)}}\circledast(K_0)_{0^{(2n)}}\circledast(K_1)_{1^{(n)}}\cdots
\circledast (K_{b-1})_{(b-1)^{(n)}}
\]
if $a<0$. Here, $ (K_j)_{j^{(n)}} $ is the weight space $ K(j, \ldots,
j)_{(j, \ldots, j)} $ of a Kato module. Since
\[
(\M_{a,b,n})_{\sqrt{\zt_{a,b,n}^{\circ}}} \subseteq
(\M_{a,b,n})_{\zt_{a,b,n}^{\circ}},
\]
we deduce that if $a\geq 0$
\[
(\M_{a,b,n})_{\sqrt{\zt_{a,b,n}^{\circ}}} = (K_a)_{\sqrt{a^{(n)}}} \circledast
(K_{a+1})_{\sqrt{(a+1)^{(n)}}} \circledast \cdots
\circledast(K_{b-1})_{\sqrt{(b-1)^{(n)}}} \subseteq \Cl(d) \phi_{\tau}
\va_{a,b,n}.
\]
Similarly, if $a<0$, $(\M_{a,b,n})_{\sqrt{\zt_{a,b,n}^{\circ}}}\subseteq\Cl(d)
\phi_{\tau} \va_{a,b,n}$.
\end{proof}

\begin{prp}
\label{mainprop1} For the special standard module defined above, $
(\M_{a,b,n})_{\sqrt{\zt_{a,b,n}}} \subseteq \Cl(d) \va_{a,b,n}. $
\end{prp}

\begin{proof}
For $ i = 1, \ldots, d, $ let $ i = jm+r $ where $ 0 \leq j <n $ and $ 0 < r <
m. $ Take any $v \in ({\M}_{a,b,n})_{\sqrt{\zt_{a,b,n}}}$.
Lemma~\ref{antidominantlemma} implies that $ \phi_{\tau} v =  \phi_{\tau} z\va
$ for some $ z \in \Cl(d). $ Put $ v_0 = v- z 1. $ Then $ \phi_{\tau} v_0 = 0$.
Note that since $ r \neq m, $ $ \phi_i v_0 = 0$ since $s_i(\zt_{a,b,n})$ is not
a weight of ${\M}_{a,b,n}$.

If $r\neq -a$, we can solve for $s_iv_0$ in the equation $\phi_i.v_0=0$ to get
\[
s_i.v_0=\left(\frac{\kp_r-\kp_{r-1}}{-2(a+r)}
    +\frac{\kp_r+\kp_{r-1}}{-2(a+r)}c_ic_{i+1}\right)v_0
\]
where $\kp_r=\sqrt{q(a+r-1)}$.

Similarly, if $ r \neq -a,$
\[
s_i.{\va}_{a,b,n}=\left(\frac{\kp_r-\kp_{r-1}}{-2(a+r)}
    +\frac{\kp_r+\kp_{r-1}}{-2(a+r)}c_ic_{i+1}\right){\va}_{a,b,n}.
\]

If $r=-a$, then routine calculations from earlier gives that
\[
c_i c_{i+1} {\va}_{a,b,n} = - \sqrt{-1} {\va}_{a,b,n}.
\]

Hence there exists an $\ASe(d)$-homomorphism $ \psi:{\M}_{a,b,n}
\rightarrow {\M}_{a,b,n} $ such that $ \psi({\va}_{a,b,n})=v_0 $ if $ a \geq 0 $ or $ b \leq 0. $
If $ a < 0 < b, $ then
there is an $\ASe(d)$-homomorphism $ \psi:{\M}_{a,b,n}
\rightarrow {\M}_{a,b,n} $ such that $ \psi({\va}_{a,b,n})=\prod_{0 \leq j <n} (1+ \sqrt{-1} c_{jm-a} c_{jm-a+1} )v_0 $

Thus by lemma ~\ref{antidominantlemma}, the kernel of $ \psi $ is equal to $
\M_{a,b,n}. $ Therefore $ v_0 =0. $ Thus $ v \in \Cl(d) {\va}_{a,b,n}. $
\end{proof}

We now reduce the general case to the special case above. To this end, fix
$\ld\in\Pt$, $\mu\in P^+[\ld]$, and $\ld-\mu\in\Ppos(d)$.  Set
$d_i=\ld_i-\mu_i$, and let $a_i=d_1+\cdots+d_{i-1}+1$, $b_i=d_1+\cdots+d_i$.
Observe that
\begin{align}\label{E:Step2Formulae}
\zt_{\ld,\mu}(x^2_{a_i})=\mu_i
\andeqn
\zt_{\ld,\mu}(x^2_{b_i})=\ld_i-1.
\end{align}
Furthermore, observe that if $a_i\leq c\leq b_i$,
\begin{align}\label{E:Step2Formulae2}
\zt_{\ld,\mu}(x^2_{c})=\zt_{\ld,\mu}(x^2_{b_i})-(b_i-c)
\andeqn
\zt_{\ld,\mu}(x^2_{c})=\zt_{\ld,\mu}(x^2_{a_i})+(c-a_i).
\end{align}

Since $\ld\in\Pt$ and $\mu\in P^+[\ld]$, we can find integers $ 0 = n'_0 < n'_1
< \cdots < n'_r = n $, and $ 0 = n_0 < n_1
< \cdots < n_s = n $ such that
\[
R[\ld] = R \cap \sum_{i \neq n'_0, \ldots, n'_r} \mathbb{Z} \af_i\andeqn
R[\ld]\cap R[\mu] = R \cap \sum_{i \neq n_0, \ldots, n_s} \mathbb{Z} \af_i.
\]
Let
$$ I'_p = \{\, a_{n'_{p-1}+1}, a_{n'_{p-1}+1}+1, \ldots, b_{n'_p}-1\,\} \;\;\;
(p=1, \ldots, r),\;\;\;
I' = I'_1 \cup \ldots \cup I'_r, $$
and
$$ I_p = \{\, a_{n_{p-1}+1}, a_{n_{p-1}+1}+1, \ldots, b_{n_p}-1\,\} \;\;\;
(p=1, \ldots, s),\;\;\;
I = I_1 \cup \ldots \cup I_s. $$
Then, $S_{\ld-\mu}\subseteq S_I\subseteq S_{I'}$ and
\[
S_{I'}/S_{\ld-\mu}\cong D_{\ld-\mu}\cap S_{I'}\andeqn S_I/S_{\ld-\mu}\cong D_{\ld-\mu}\cap S_I,\;\;\;\mbox{(cf. $\S$\ref{SS:Mackey}).}
\]

\begin{lem}\cite[Lemma A.9]{su2}
There is a containment of sets $ D_{\lambda - \mu}\cap S_d[\zt_{\lambda, \mu}]
\subset D_{\lambda - \mu}\cap S_I.  $
\end{lem}

Let $ v \in \M(\lambda, \mu)_{\sqrt{\zt_{\lambda, \mu}}}. $ For each $ p \in
\lbrace 1, \ldots, s \rbrace, $ we can write $ v = \sum_j x_j^{(p)} z_j^{(p)}
v_j $ where $ v_j \in \Phi(\lambda, \mu), $ $ \lbrace x_j^{(p)} \rbrace_j $ are
linearly independent elements of $ \C[D_{\lambda - \mu} \cap S_{I - I_p}] $ and
$z_j^{(p)} \in \C[D_{\lambda - \mu} \cap S_{I_p}]$. Let $\P_d[x^2]_{I_p}=\C[x_i^2|i\in I_p]$.

\begin{lem}\cite[Lemma A.10]{su2}
For $ f \in \P_d[x^2]_{I_p}, $ $ f z_k^{(p)} v_j =
\zt_{\lambda, \mu}(f) z_k^{(p)} v_j. $
\end{lem}

\begin{proof}
Observe
\[
0=(f-\zt_{\ld,\mu}(f))v
=\sum_jx_j^{(p)}(f-\zt_{\ld,\mu}(f))z_j^{(p)}\va_{\ld,\mu}.
\]
Since $S_{I_p}\subset S_d$ is closed with respect to the Bruhat order we have
$f z_j^{(p)}\va_{\ld,\mu}\in\C[D_{\ld-\mu}\cap S_{I_p}]$. Since
$\{x_j^{(p)}\}_j$ are linearly independent, each
$(f-\zt_{\ld,\mu}(f))z_j^{(p)}\va_{\ld,\mu}$ must be 0.
\end{proof}

\noindent\emph{Proof of Proposition \ref{dominant wt space}.}
Let $ \ASe(I_p) $ be the subalgebra corresponding to $ I_p. $  Note that
$ \ASe(I_p) \cong \ASe(|I_p|). $ First note that $
\ASe(I_p) v_j \cong \M_{a,b,n_p-n_{p-1}} $ for some $a,b$. Thus by Proposition
~\ref{mainprop1}, $z_k^{(p)} v_j \in \C\va_{\lambda, \mu}$. Thus, $ v \in
\C[D_{\lambda - \mu} \cap S_{I - I_p}] $ for any $ p. $  It now follows that $
v \in \mathbb{C} \va_{\lambda, \mu}. $\QED

\begin{thm}\label{thm:unique irred quotient} Assume that $\ld\in\Pt$, $\mu\in P^+[\ld]$, and
$\ld-\mu\in\Ppos(d)$. Then $\M(\ld,\mu)$ has a unique simple quotient module,
denoted $\L(\ld,\mu)$.
\end{thm}

\begin{proof}
Assume $N$ is a submodule of $\M(\ld,\mu)$. If
$N_{\zt_{\ld,\mu}}^{\mathrm{gen}}\neq0$, then $N_{\sqrt{\zt_{\ld,\mu}}}\neq0$.
By the previous lemma, $N\cap\Cl(d)\va_{\ld,\mu}\neq\{0\}$, so $\va_{\ld,\mu}\in N$ because $\Cl(d)\va_{\ld,\mu}$ is an irreducible $\ASe(\ld-\mu)$-module. Hence, $N=\M(\ld,\mu)$. It follows that
\[
N\subseteq\bigoplus_{\eta\neq\zt_{\ld,\mu}}\M(\ld,\mu)^{\mathrm{gen}}_\eta.
\]
The sum of all proper submodules satisfies this property. Therefore, $\M(\ld,\mu)$
has a unique maximal proper submodule and a unique simple quotient.
\end{proof}

Let $\mathcal{R}(\ld,\mu)$ denote the unique maximal submodule, and define
$\L(\ld,\mu)=\M(\ld,\mu)/\mathcal{R}(\ld,\mu)$.

\section{Classification of Calibrated Representations}\label{S:Calibrated}
A representation $ M $ of the AHCA is called
\emph{calibrated} if the polynomial subalgebra $\P_d[x]\subseteq\ASe(d)$ acts
semisimply. The main combinatorial object associated to such a representation
is the shifted skew shape.  Calibrated representations of the affine Hecke
algebra were studied and classified in \cite{ram}.  The main combinatorial
object in that case were pairs of skew shapes and content functions.  That
construction along with \cite[Conjecture 52]{lec} motivated the construction
given here.  A proof of a slightly modified version of that conjecture is given
here.  Leclerc defined a calibrated representation to be one in which
$\P_d[x^2]$ acts semisimply.  For example, the module $ \Phi_{[-1,0]} $ is
calibrated in the sense of \cite{lec} but $ x_1, x_2 $ do not act diagonally in
any basis.

\subsection{Construction of Calibrated Representations}\label{SS:Calibrated} Let $ \lambda = (\lambda_1, \ldots, \lambda_r) $ and $ \mu = (\mu_1, \ldots, \mu_r) $ be two partitions with $ \lambda_1 > \cdots > \lambda_r >0 $ and
$ \mu_1 \geq \cdots \geq \mu_r $ such that  $ \mu_i = \mu_{i+1} $ implies $ \mu_i = 0$ and $ \lambda_i \geq \mu_i $ for all $ i. $
To such data, associate a shifted skew shape of boxes where row $ i $ has $ \lambda_i - \mu_i $ boxes and the leftmost box occurs in position $ i. $  Figure \ref{ex1}
illustrates a skew shape for $ \lambda = (5,2,1) $ and $ \mu = (3,1,0). $

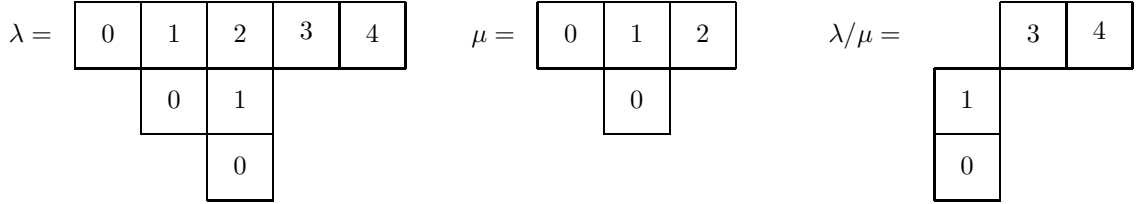
\begin{figure}[ht]
\begin{picture}(420,85)

\put(0,70){$\ld=$}

\put(35,70){$0$}\put(60,70){$1$}\put(85,70){2}\put(110,71){3}
    \put(135,70){4}
\put(60,45){0}\put(85,45){1}
\put(85,20){0}

\put(25,85){\line(1,0){125}}
    \put(25,85){\line(0,-1){25}}\put(50,85){\line(0,-1){50}}
    \put(75,85){\line(0,-1){75}}\put(100,85){\line(0,-1){75}}
    \put(125,85){\line(0,-1){25}}\put(150,85){\line(0,-1){25}}
\put(25,60){\line(1,0){125}}
\put(50,35){\line(1,0){50}}
\put(75,10){\line(1,0){25}}

\put(175,70){$\mu=$}

\put(210,70){$0$}\put(235,70){$1$}\put(260,70){2}
\put(235,45){0}

\put(200,85){\line(1,0){75}}
    \put(200,85){\line(0,-1){25}}\put(225,85){\line(0,-1){50}}
    \put(250,85){\line(0,-1){50}}\put(275,85){\line(0,-1){25}}
\put(200,60){\line(1,0){75}}
\put(225,35){\line(1,0){25}}

\put(310,70){$\ld/\mu=$}

\put(385,70){3}
    \put(410,71){4}
\put(360,45){1}
\put(360,20){0}

\put(375,85){\line(1,0){50}}
    \put(375,85){\line(0,-1){75}}\put(400,85){\line(0,-1){25}}
    \put(425,85){\line(0,-1){25}}
\put(350,60){\line(1,0){75}}
    \put(350,60){\line(0,-1){50}}
\put(350,35){\line(1,0){25}}
\put(350,10){\line(1,0){25}}

\end{picture}
\caption{Skew Shape filled with contents}\label{ex1}
\end{figure}

A standard filling of a skew shape $ \lambda / \mu $ with a total of $ d $
boxes is an insertion of the set $ \{ 1, \ldots, d \}$ into the boxes of the
skew shape such that each box gets exactly one element, each element is used
exactly once and the rows are increasing from left to right and the columns are
increasing from top to bottom.  In a shifted shape, $\ld$, all the boxes will
lie above one main diagonal running from northwest to southeast. Each box in
this main diagonal will be assigned content 0.  The contents of the other boxes
will be constant along the diagonals where the contents of the diagonal
northeast of its immediate neighbor will be one more than the contents of its
immediate neighbor. In a shifted skew shape, $\ld/\mu$, the contents are
defined as in figure \ref{ex1}.

Given a standard tableaux $ L $ for a shifted skew shape $ \lambda / \mu, $ let
$ c(L_i) $ be the contents of the box labeled by $ i. $  Thus $ L $ gives rise
to a $ d$-tuple $ c(L) = (c(L_1), \ldots, c(L_d)) $ called the content reading
of $ \lambda / \mu $ with respect to $ L. $

Let $ \lambda / \mu $ be a shifted skew shape such that $ \lambda / \mu $ has $
d $ boxes. Set $ \kp_{i,L} = \sqrt{q(c(L_{i}))} $ and
$$ \mathcal{Y}_{i,L} = \sqrt{1-\frac{1}{({\kp_{i+1,L}-\kp_{i,L}})^2}
        -\frac{1}{({\kp_{i+1,L}+\kp_{i,L}})^2}}.
        $$

Now to a skew shape $ \lambda/\mu, $ associate a vector space $
\widehat{H}^{\lambda / \mu} = \oplus_L Cl(d) v_L $ where $ L $ ranges over all
standard tableaux of shape $ \lambda/\mu $ and $ d $ is the number of boxes in the shifted skew shape.
Define $ x_i v_L = \kp_{i,L} v_L. $ Define
$$ s_i v_L = \frac{1}{\kp_{i+1,L}-\kp_{i,L}} v_L
+  \frac{1}{\kp_{i+1,L}+\kp_{i,L}} c_i c_{i+1} v_L + \mathcal{Y}_{i,L} v_{s_i L} $$
where $ v_{s_i L} = 0 $ if $ s_i L $ is not a standard tableaux.

\begin{prp}
The action of the $ x_i $ and $s_{i}$ given above endow $ \widehat{H}^{\lambda / \mu} $ with
the structure of a $ \ASe(d)$-module.
\end{prp}

\begin{proof}
We have
\begin{align*}
s_i^2 v_L &= \frac{1}{\kp_{i+1,L}-\kp_{i,L}} s_i v_L - \frac{1}{\kp_{i+1,L}+\kp_{i,L}} c_i c_{i+1} s_i v_L + \mathcal{Y}_{i,L} s_i v_{s_i L}\\
&= \frac{1}{\kp_{i+1,L}-\kp_{i,L}}\left(\frac{1}{\kp_{i+1,L}-\kp_{i,L}} v_L + \frac{1}{\kp_{i+1,L}+\kp_{i,L}} c_i c_{i+1} v_L + \mathcal{Y}_{i,L} v_{s_i L}\right)\\ &-\frac{c_i c_{i+1}}{\kp_{i+1,L}+\kp_{i,L}}\left(\frac{1}{\kp_{i+1,L}-\kp_{i,L}} v_L
+ \frac{1}{\kp_{i+1,L}+\kp_{i,L}} c_i c_{i+1} v_L + \mathcal{Y}_{i,L} v_{s_i L}\right)\\
&+ \mathcal{Y}_{i,L}\left(\frac{1}{\kp_{i,L}-\kp_{i+1,L}} v_{s_i L} + \frac{1}{\kp_{i+1,L}+\kp_{i,L}} c_i c_{i+1} v_{s_i L} + \mathcal{Y}_{i,L} v_{L}\right)\\
&= \left(\frac{1}{(\kp_{i+1,L}-\kp_{i,L})^2}+\frac{1}{(\kp_{i+1,L}+\kp_{i,L})^2} + \mathcal{Y}_{i,L} \mathcal{Y}_{i,L}\right) v_L = v_L.
\end{align*}
Note that if $ v_{s_i L} = 0, $ then $ \frac{1}{(\kp_{i+1,L}-\kp_{i,L})^2} + \frac{1}{(\kp_{i+1,L}+\kp_{i,L})^2}=1. $

Next,
$$ s_i x_i v_L = \frac{\kp_{i,L}}{\kp_{i+1,L}-\kp_{i,L}} v_L + \frac{\kp_{i,L}}{\kp_{i+1,L}+\kp_{i,L}} c_i c_{i+1} v_L + \mathcal{Y}_{i,L} v_{s_i L}.  $$
On the other hand,
$$ x_{i+1} s_i v_L - v_L + c_i c_{i+1} v_L =
\frac{\kp_{i+1,L}}{\kp_{i+1,L}-\kp_{i,L}} v_L - \frac{\kp_{i+1,L}}{\kp_{i+1,L}+\kp_{i,L}} c_i c_{i+1} v_L +\mathcal{Y}_{i,L} v_{s_i L} - v_L + c_i c_{i+1} v_L. $$
Thus it is easily seen that
$$ s_i x_i v_L = x_{i+1} s_i v_L - v_L + c_i c_{i+1} v_L. $$

We now check the braid relations.  To this end, fix $j\in\mathbb{N}$ and set
$\kp_i=\sqrt{j+i}$ for $i\geq0$.

\begin{figure}[ht]
\begin{picture}(340,40)

\put(115,14){$L\;\;=$}
    \put(160,14){$i$}\put(178,14){$i+1$}\put(203,14){$i+2$}

\put(150,5){\line(1,0){75}}
    \put(150,5){\line(0,1){25}}\put(175,5){\line(0,1){25}}
    \put(200,5){\line(0,1){25}}\put(225,5){\line(0,1){25}}
\put(150,30){\line(1,0){75}}
\end{picture}
\caption{Case 1}\label{F:Case 1}
\end{figure}

Case 1: Let $ L $ be the standard tableaux given in Figure \ref{F:Case 1}. A
calculation gives {\small
\begin{align*}
s_i s_{i+1} s_i v_L = s_{i+1} s_i s_{i+1} v_L
=&\left(\frac{1}{(\kp_3-\kp_2)^2(\kp_2-\kp_1)}
    -\frac{1}{(\kp_2+\kp_3)^2(\kp_1+\kp_2)}\right)v_L\\
&+\left(\frac{1}{(\kp_3^2-\kp_2^2)(\kp_2+\kp_1)}
        +\frac{1}{(\kp_3^2-\kp_2^2)^2(\kp_2-\kp_1)}\right)c_ic_{i+1}v_L\\
&+\left(\frac{1}{(\kp_3^2-\kp_2^2)(\kp_2-\kp_1)}
    +\frac{1}{(\kp_3^2-\kp_2^2)^2(\kp_2+\kp_1)}\right)c_{i+1}c_{i+2}v_L\\
&+\left(\frac{1}{(\kp_3-\kp_2)^2(\kp_2+\kp_1)}
    -\frac{1}{(\kp_2+\kp_3)^2(\kp_2-\kp_1)}\right)c_ic_{i+2}v_L.
\end{align*}
}

\begin{figure}[ht]
\begin{picture}(340,70)

\put(50,27){$L_1=$}
    \put(78,14){$i+2$}\put(85,39){$i$}\put(103,39){$i+1$}

\put(75,5){\line(1,0){25}}
    \put(75,5){\line(0,1){50}}\put(100,5){\line(0,1){50}}
\put(75,30){\line(1,0){50}}
    \put(100,30){\line(0,1){25}}\put(125,30){\line(0,1){25}}
\put(75,55){\line(1,0){50}}

\put(200,27){$L_2=$}
   \put(228,14){$i+1$}\put(235,39){$i$}\put(253,39){$i+2$}

\put(225,5){\line(1,0){25}}
    \put(225,5){\line(0,1){50}}\put(250,5){\line(0,1){50}}
\put(225,30){\line(1,0){50}}
    \put(250,30){\line(0,1){25}}\put(275,30){\line(0,1){25}}
\put(225,55){\line(1,0){50}}

\end{picture}
\caption{Case 2}\label{F:Case 2}
\end{figure}

Case 2: Let $ L_1 $ and $ L_2 $ be the standard tableaux given in Figure
\ref{F:Case 2}. A calculation gives {\small
\begin{align*}
s_i s_{i+1} s_i v_{L_1} = s_{i+1} s_i s_{i+1}v_{L_1} =&\left(\frac{1}{(\kp_3-\kp_2)^2(\kp_1-\kp_3)}
    +\frac{1}{(\kp_2+\kp_3)^2(\kp_1+\kp_3)}\right)v_{L_1}\\
    &+\left(\frac{1}{(\kp_3^2-\kp_2^2)(\kp_1-\kp_3)}
    -\frac{1}{(\kp_3^2-\kp_2^2)^2(\kp_1+\kp_3)}\right)c_ic_{i+1}v_{L_1}\\
&+\left(\frac{1}{(\kp_2^2-\kp_3^2)(\kp_1+\kp_3)}
    +\frac{1}{(\kp_3^2-\kp_2^2)^2(\kp_1-\kp_3)}\right)c_{i+1}c_{i+2}v_{L_1}\\
&+\left(\frac{1}{(\kp_3-\kp_2)^2(\kp_1+\kp_3)}
    -\frac{1}{(\kp_2+\kp_3)^2(\kp_1-\kp_3)}\right)c_ic_{i+2}v_{L_1}\\
&+\left(\frac{\mathcal{Y}_{i+1,L_1}}{(\kp_3-\kp_2)(\kp_1-\kp_2)}\right)v_{L_2}
    +\left(\frac{\mathcal{Y}_{i+1,L_1}}{(\kp_3-\kp_2)(\kp_1+\kp_2)}\right)
    c_ic_{i+1}v_{L_2}\\
&+\left(\frac{\mathcal{Y}_{i+1,L_1}}{(\kp_1-\kp_2)(\kp_2+\kp_3)}\right)
    c_{i+1}c_{i+2}v_{L_2}
    +\left(\frac{\mathcal{Y}_{i+1,L_1}}{(\kp_2+\kp_3)(\kp_1+\kp_2)}\right)
    c_ic_{i+2}v_{L_2}.
\end{align*}
\begin{align*}
s_i s_{i+1} s_i v_{L_2} = s_{i+1} s_i s_{i+1}v_{L_2} =&
\left(\frac{1}{(\kp_1-\kp_2)^2(\kp_3-\kp_1)}
    +\frac{1}{(\kp_1+\kp_2)^2(\kp_1+\kp_3)}\right)v_{L_2}\\
&+\left(\frac{1}{(\kp_1^2-\kp_2^2)(\kp_3-\kp_1)}
    -\frac{1}{(\kp_1^2-\kp_2^2)^2(\kp_1+\kp_3)}\right)c_ic_{i+1}v_{L_2}\\
&+\left(\frac{-1}{(\kp_1^2-\kp_2^2)(\kp_1+\kp_3)}
    +\frac{1}{(\kp_1^2-\kp_2^2)^2(\kp_3-\kp_1)}\right)c_{i+1}c_{i+2}v_{L_2}\\
&+\left(\frac{1}{(\kp_1-\kp_2)^2(\kp_1+\kp_3)}
    +\frac{1}{(\kp_1+\kp_2)^2(\kp_3-\kp_1)}\right)c_ic_{i+2}v_{L_2}\\
&+\left(\frac{\mathcal{Y}_{i+1,L_2}}{(\kp_3-\kp_2)(\kp_1-\kp_2)}\right)v_{L_1}
    +\left(\frac{\mathcal{Y}_{i+1,L_2}}{(\kp_1-\kp_2)(\kp_2+\kp_3)}\right)
    c_ic_{i+1}v_{L_1}\\
&+\left(\frac{\mathcal{Y}_{i+1,L_2}}{(\kp_3-\kp_2)(\kp_1+\kp_2)}\right)
    c_{i+1}c_{i+2}v_{L_1}
+\left(\frac{\mathcal{Y}_{i+1,L_1}}{(\kp_2+\kp_3)(\kp_1+\kp_2)}\right)c_ic_{i+2}v_{L_1}.
\end{align*}
}
\begin{figure}[ht]
\begin{picture}(340,70)

\put(50,27){$L_1=$}
    \put(80,14){$i$}\put(103,39){$i+1$}\put(103,14){$i+2$}

\put(75,5){\line(1,0){50}}
    \put(75,5){\line(0,1){25}}\put(100,5){\line(0,1){50}}\put(125,5){\line(0,1){50}}
\put(75,30){\line(1,0){50}}
\put(100,55){\line(1,0){25}}

\put(200,27){$L_2=$}
   \put(228,14){$i+1$}\put(252,14){$i+2$}\put(258,39){$i$}

\put(225,5){\line(1,0){50}}
    \put(225,5){\line(0,1){25}}\put(250,5){\line(0,1){50}}\put(275,5){\line(0,1){50}}
\put(225,30){\line(1,0){50}}
\put(250,55){\line(1,0){25}}

\end{picture}
\caption{Case 3}\label{F:Case 3}
\end{figure}

Case 3: Let $L_1$ and $L_2$ be as in figure \ref{F:Case 3}. Then, a calculation analogous to case 2 shows that $s_is_{i+1}s_iv_{L_1}=s_{i+1}s_is_{i+2}v_{L_1}$ and $s_is_{i+1}s_iv_{L_2}=s_{i+1}s_is_{i+2}v_{L_2}$.

\begin{figure}[ht]
\begin{picture}(420,60)

\put(0,27){$L_1=$}
    \put(37,14){$i$}\put(53,14){$i+1$}\put(78,39){$i+2$}

\put(25,5){\line(1,0){50}}
    \put(25,5){\line(0,1){25}}\put(50,5){\line(0,1){25}}
    \put(75,5){\line(0,1){50}}
\put(25,30){\line(1,0){75}}
    \put(100,30){\line(0,1){25}}
\put(75,55){\line(1,0){25}}

\put(150,27){$L_2=$}
    \put(187,14){$i$}\put(203,14){$i+2$}\put(228,39){$i+1$}

\put(175,5){\line(1,0){50}}
    \put(175,5){\line(0,1){25}}\put(200,5){\line(0,1){25}}
    \put(225,5){\line(0,1){50}}
\put(175,30){\line(1,0){75}}
    \put(250,30){\line(0,1){25}}
\put(225,55){\line(1,0){25}}

\put(300,27){$L_3=$}
    \put(328,14){$i+1$}\put(353,14){$i+2$}\put(385,39){$i$}

\put(325,5){\line(1,0){50}}
    \put(325,5){\line(0,1){25}}\put(350,5){\line(0,1){25}}
    \put(375,5){\line(0,1){50}}
\put(325,30){\line(1,0){75}}
    \put(400,30){\line(0,1){25}}
\put(375,55){\line(1,0){25}}

\end{picture}
\caption{Case 4}\label{F:Case 4}
\end{figure}

Case 4: Let $ L_1, L_2, $ and $ L_3 $ be the standard tableaux given in Figure
\ref{F:Case 3}. {\small
\begin{align*}
s_i s_{i+1} s_i v_{L_1} = s_{i+1} s_i s_{i+1}v_{L_1} =&
\left(\frac{1}{(\kp_1-\kp_0)^2(\kp_3-\kp_1)}
    +\frac{1}{(\kp_0+\kp_1)^2(\kp_1+\kp_3)}\right)v_{L_1}\\
&+\left(\frac{1}{(\kp_1^2-\kp_0^2)(\kp_3-\kp_1)}
    -\frac{1}{(\kp_1^2-\kp_0^2)(\kp_1+\kp_3)}\right)c_ic_{i+1}v_{L_1}\\
&+\left(\frac{1}{(\kp_0^2-\kp_1^2)(\kp_1+\kp_3)}
    -\frac{1}{(\kp_0^2-\kp_1^2)(\kp_3-\kp_1)}\right)c_{i+1}c_{i+2}v_{L_1}\\
&+\left(\frac{1}{(\kp_1-\kp_0)^2(\kp_1+\kp_3)}
    -\frac{1}{(\kp_0+\kp_1)^2(\kp_3-\kp_1)}\right)c_ic_{i+2}v_{L_1}\\
&+\left(\frac{\mathcal{Y}_{i+1,L_1}}{(\kp_1-\kp_0)(\kp_3-\kp_0)}\right)v_{L_2}
    +\left(\frac{\mathcal{Y}_{i+1,L_1}}{(\kp_1-\kp_0)(\kp_0+\kp_3)}\right)
    c_ic_{i+1}v_{L_2}\\
&+\left(\frac{\mathcal{Y}_{i+1,L_1}}{(\kp_0+\kp_1)(\kp_3-\kp_0)}\right)
    c_{i+1}c_{i+2}v_{L_2}
    +\left(\frac{\mathcal{Y}_{i+1,L_1}}{(\kp_0+\kp_1)(\kp_0+\kp_3)}\right) c_ic_{i+2}v_{L_2}\\
&+\left(\frac{\mathcal{Y}_{i+1,L_1}\mathcal{Y}_{i,L_2}}{\kp_1-\kp_0}\right)v_{L_3}
    +\left(\frac{\mathcal{Y}_{i+1,L_1}\mathcal{Y}_{i,L_2}}{\kp_0+\kp_1}\right) c_{i+1}c_{i+2}v_{L_3}.
\end{align*}
\begin{align*}
s_i s_{i+1} s_i v_{L_2} = s_{i+1} s_i s_{i+1}v_{L_2} =&
\left(\frac{1}{(\kp_3-\kp_0)^2(\kp_1-\kp_3)}
    +\frac{1}{(\kp_0+\kp_3)^2(\kp_1+\kp_3)}+\frac{\mathcal{Y}_{i,L_2} \mathcal{Y}_{i,L_3}}{\kp_1-\kp_0}\right)v_{L_2}\\
&+\left(\frac{1}{(\kp_3^2-\kp_0^2)(\kp_1-\kp_3)}
    -\frac{1}{(\kp_3^2-\kp_0^2)(\kp_1+\kp_3)}\right)c_ic_{i+1}v_{L_2}\\
&+\left(\frac{-1}{(\kp_3^2-\kp_0^2)(\kp_1-\kp_3)}
    -\frac{1}{(\kp_3^2-\kp_0^2)(\kp_1-\kp_3)}\right)c_{i+1}c_{i+2}v_{L_2}\\
&+\left(\frac{1}{(\kp_3-\kp_0)^2(\kp_1+\kp_3)}
    -\frac{1}{(\kp_0+\kp_3)^2(\kp_1-\kp_3)}
    +\frac{\mathcal{Y}_{i,L_2}\mathcal{Y}_{i,L_3}}{\kp_0+\kp_1}\right)
    c_ic_{i+2}v_{L_2}\\
&+\left(\frac{\mathcal{Y}_{i,L_2}}{(\kp_3-\kp_0)(\kp_1-\kp_3)}
    +\frac{\mathcal{Y}_{i,L_2}}{(\kp_1-\kp_0)(\kp_0-\kp_3)}\right)v_{L_3}\\
&+\left(\frac{-\mathcal{Y}_{i,L_2}}{(\kp_3+\kp_0)(\kp_1+\kp_3)}
    +\frac{\mathcal{Y}_{i,L_2}}{(\kp_1-\kp_0)(\kp_0+\kp_3)}\right)
    c_ic_{i+1}v_{L_3}\\
&+\left(\frac{\mathcal{Y}_{i,L_2}}{(\kp_3+\kp_0)(\kp_1-\kp_3)}
    -\frac{\mathcal{Y}_{i,L_2}}{(\kp_1+\kp_0)(\kp_0+\kp_3)}\right)
    c_{i+1}c_{i+2}v_{L_3}\\
&+\left(\frac{\mathcal{Y}_{i,L_2}}{(\kp_3-\kp_0)(\kp_1+\kp_3)}
    +\frac{\mathcal{Y}_{i,L_2}}{(\kp_1+\kp_0)(\kp_0-\kp_3)}\right)
    c_ic_{i+2}v_{L_3}\\
&+\left(\frac{\mathcal{Y}_{i+1,L_2}}{(\kp_1-\kp_0)(\kp_3-\kp_0)}\right)v_{L_1}
    +\left(\frac{\mathcal{Y}_{i+1,L_2}}{(\kp_1+\kp_0)(\kp_3-\kp_0)}\right)
    c_ic_{i+1}v_{L_1}\\
&+\left(\frac{\mathcal{Y}_{i+1,L_2}}{(\kp_3+\kp_0)(\kp_1-\kp_0)}\right)
    c_{i+1}c_{i+2}v_{L_1}
    +\left(\frac{\mathcal{Y}_{i+1,L_2}}{(\kp_3+\kp_0)(\kp_1+\kp_0)}\right)
    c_ic_{i+2}v_{L_1}.
\end{align*}
\begin{align*}
s_i s_{i+1} s_i v_{L_3} = s_{i+1} s_i s_{i+1}v_{L_3} =&
\left(\frac{1}{(\kp_3-\kp_0)^2(\kp_1-\kp_0)}
    +\frac{1}{(\kp_0+\kp_3)^2(\kp_0+\kp_1)}+\frac{\mathcal{Y}_{i,L_2} \mathcal{Y}_{i,L_3}}{\kp_1-\kp_3}\right) v_{L_3} \\
&+\left(\frac{1}{(\kp_0^2-\kp_3^2)(\kp_1-\kp_0)}
    -\frac{1}{(\kp_0^2-\kp_3^2)(\kp_0+\kp_1)}\right)c_ic_{i+1}v_{L_3}\\
&+\left(\frac{-1}{(\kp_0^2-\kp_3^2)(\kp_0+\kp_1)}
    +\frac{1}{(\kp_0^2-\kp_3^2)(\kp_1-\kp_0)}\right)c_{i+1}c_{i+2}v_{L_3}\\
&+\left(\frac{1}{(\kp_3-\kp_0)^2(\kp_0+\kp_1)}
    +\frac{1}{(\kp_0+\kp_3)^2(\kp_1-\kp_0)}
    +\frac{\mathcal{Y}_{i,L_2}\mathcal{Y}_{i,L_3}}{\kp_1+\kp_3}\right)
    c_ic_{i+2}v_{L_3}\\
&+\left(\frac{\mathcal{Y}_{i,L_3}}{(\kp_0-\kp_3)(\kp_1-\kp_0)}
    +\frac{\mathcal{Y}_{i,L_3}}{(\kp_1-\kp_3)(\kp_3-\kp_0)}\right)v_{L_2}\\
&+\left(\frac{-\mathcal{Y}_{i,L_3}}{(\kp_3+\kp_0)(\kp_0+\kp_1)}
    +\frac{\mathcal{Y}_{i,L_3}}{(\kp_1-\kp_3)(\kp_0+\kp_3)}\right)
    c_ic_{i+1}v_{L_2}\\
&+\left(\frac{\mathcal{Y}_{i,L_3}}{(\kp_3+\kp_0)(\kp_1-\kp_0)}
    -\frac{\mathcal{Y}_{i,L_3}}{(\kp_1+\kp_3)(\kp_0+\kp_3)}\right)
    c_{i+1}c_{i+2}v_{L_2}\\
&+\left(\frac{\mathcal{Y}_{i,L_3}}{(\kp_0-\kp_3)(\kp_0+\kp_1)}
    +\frac{\mathcal{Y}_{i,L_3}}{(\kp_1+\kp_3)(\kp_3-\kp_0)}\right)
    c_ic_{i+2}v_{L_2}\\
&+\left(\frac{\mathcal{Y}_{i,L_3}\mathcal{Y}_{i+1,L_2}}{\kp_1-\kp_0}\right)v_{L_1}
    +\left(\frac{\mathcal{Y}_{i,L_3}\mathcal{Y}_{i+1,L_2}}{\kp_1+\kp_0}\right)
    c_ic_{i+1}v_{L_1}.
\end{align*}
}

\begin{figure}[ht]
\begin{picture}(420,185)

\put(0,148){$L_1=$}
    \put(37,114){$i$}\put(53,139){$i+1$}\put(78,164){$i+2$}

\put(25,105){\line(1,0){25}}
    \put(25,105){\line(0,1){25}}\put(50,105){\line(0,1){50}}
\put(25,130){\line(1,0){50}}
    \put(75,130){\line(0,1){50}}
\put(50,155){\line(1,0){50}}
    \put(100,155){\line(0,1){25}}
\put(75,180){\line(1,0){25}}

\put(150,148){$L_2=$}
    \put(178,114){$i+1$}\put(212,139){$i$}\put(228,164){$i+2$}

\put(175,105){\line(1,0){25}}
    \put(175,105){\line(0,1){25}}\put(200,105){\line(0,1){50}}
\put(175,130){\line(1,0){50}}
    \put(225,130){\line(0,1){50}}
\put(200,155){\line(1,0){50}}
    \put(250,155){\line(0,1){25}}
\put(225,180){\line(1,0){25}}

\put(300,148){$L_3=$}
    \put(337,114){$i$}\put(353,139){$i+2$}\put(378,164){$i+1$}

\put(325,105){\line(1,0){25}}
    \put(325,105){\line(0,1){25}}\put(350,105){\line(0,1){50}}
\put(325,130){\line(1,0){50}}
    \put(375,130){\line(0,1){50}}
\put(350,155){\line(1,0){50}}
    \put(400,155){\line(0,1){25}}
\put(375,180){\line(1,0){25}}

\put(0,48){$L_4=$}
    \put(28,14){$i+2$}\put(62,39){$i$}\put(78,64){$i+1$}

\put(25,5){\line(1,0){25}}
    \put(25,5){\line(0,1){25}}\put(50,5){\line(0,1){50}}
\put(25,30){\line(1,0){50}}
    \put(75,30){\line(0,1){50}}
\put(50,55){\line(1,0){50}}
    \put(100,55){\line(0,1){25}}
\put(75,80){\line(1,0){25}}

\put(150,48){$L_5=$}
    \put(178,14){$i+2$}\put(203,39){$i+1$}\put(237,64){$i$}

\put(175,5){\line(1,0){25}}
    \put(175,5){\line(0,1){25}}\put(200,5){\line(0,1){50}}
\put(175,30){\line(1,0){50}}
    \put(225,30){\line(0,1){50}}
\put(200,55){\line(1,0){50}}
    \put(250,55){\line(0,1){25}}
\put(225,80){\line(1,0){25}}

\put(300,48){$L_6=$}
    \put(328,14){$i+1$}\put(353,39){$i+2$}\put(387,64){$i$}

\put(325,5){\line(1,0){25}}
    \put(325,5){\line(0,1){25}}\put(350,5){\line(0,1){50}}
\put(325,30){\line(1,0){50}}
    \put(375,30){\line(0,1){50}}
\put(350,55){\line(1,0){50}}
    \put(400,55){\line(0,1){25}}
\put(375,80){\line(1,0){25}}

\end{picture}
\caption{Case 5}\label{F:Case 5}
\end{figure}

Case 5: Let $ L_1, L_2, L_3, L_4, L_5, $ and $ L_6 $ be given as in Figure
\ref{F:Case 5}.

{\small
\begin{align*}
s_i s_{i+1} s_i v_{L_1} = s_{i+1} s_i s_{i+1}v_{L_1} =&
\left(\frac{1}{(\kp_2-\kp_0)^2(\kp_4-\kp_2)}
    +\frac{1}{(\kp_2+\kp_0)^2(\kp_4+\kp_2)}+\frac{\mathcal{Y}_{i,L_1} \mathcal{Y}_{i, L_2}}{\kp_4-\kp_0}\right)v_{L_1}\\
&+\left(\frac{1}{(\kp_2^2-\kp_0^2)(\kp_4-\kp_2)}
    -\frac{1}{(\kp_2^2-\kp_0^2)(\kp_4+\kp_2)}\right)c_ic_{i+1}v_{L_1}\\
&+\left(\frac{-1}{(\kp_2^2-\kp_0^2)(\kp_4+\kp_2)}
    +\frac{1}{(\kp_2^2-\kp_0^2)(\kp_4-\kp_2)}\right)c_{i+1}c_{i+2}v_{L_1}\\
&+\left(\frac{1}{(\kp_2-\kp_0)^2(\kp_4+\kp_2)}
    +\frac{1}{(\kp_2+\kp_0)^2(\kp_4-\kp_2)}+\frac{\mathcal{Y}_{i,L_1} \mathcal{Y}_{i, L_2}}{\kp_4+\kp_0}\right)c_ic_{i+2}v_{L_1}\\
&+\left(\frac{\mathcal{Y}_{i,L_1}}{(\kp_2-\kp_0)(\kp_4-\kp_2)}
    +\frac{\mathcal{Y}_{i,L_1}}{(\kp_4-\kp_0)(\kp_0-\kp_2)}\right)v_{L_2}\\
&+\left(\frac{-\mathcal{Y}_{i,L_1}}{(\kp_2+\kp_0)(\kp_4+\kp_2)}
    +\frac{\mathcal{Y}_{i,L_1}}{(\kp_4-\kp_0)(\kp_0+\kp_2)}\right)c_ic_{i+1}v_{L_2}\\
&+\left(\frac{\mathcal{Y}_{i,L_1}}{(\kp_2+\kp_0)(\kp_4-\kp_2)}
    -\frac{\mathcal{Y}_{i,L_1}}{(\kp_4+\kp_0)(\kp_0+\kp_2)}\right)
    c_{i+1}c_{i+2}v_{L_2}\\
&+\left(\frac{\mathcal{Y}_{i,L_1}}{(\kp_2-\kp_0)(\kp_4+\kp_2)}
    +\frac{\mathcal{Y}_{i,L_1}}{(\kp_4+\kp_0)(\kp_0-\kp_2)}\right)c_{i}c_{i+2}v_{L_2}\\
&+\left(\frac{\mathcal{Y}_{i+1,L_1}}{(\kp_2-\kp_0)(\kp_4-\kp_0)}\right)v_{L_3}+
    \left(\frac{\mathcal{Y}_{i+1,L_1}}{(\kp_2-\kp_0)(\kp_4+\kp_0)}\right)
    c_ic_{i+1}v_{L_3}\\
&+\left(\frac{\mathcal{Y}_{i+1,L_1}}{(\kp_2+\kp_0)(\kp_4-\kp_0)}\right)
    c_{i+1}c_{i+2}v_{L_3}
    +\left(\frac{\mathcal{Y}_{i+1,L_1}}{(\kp_2+\kp_0)(\kp_4+\kp_0)}\right) c_ic_{i+2}v_{L_3}\\
&+\left(\frac{\mathcal{Y}_{i+1,L_1})\mathcal{Y}_{i,L_3}}{\kp_2-\kp_0}\right)v_{L_6}
    +\left(\frac{\mathcal{Y}_{i+1,L_1}\mathcal{Y}_{i,L_3}}{\kp_2+\kp_0}\right) c_{i+1}c_{i+2}v_{L_6}\\
&+\left(\frac{(\mathcal{Y}_{i,L_1})(\mathcal{Y}_{i+1,L_2})}{\kp_4-\kp_2}\right)v_{L_4}
    +\left(\frac{\mathcal{Y}_{i,L_1}\mathcal{Y}_{i+1,L_2}}{\kp_4+\kp_2}\right)
    c_ic_{i+1}v_{L_4}+(\mathcal{Y}_{i,L_1}\mathcal{Y}_{i+1,L_2}\mathcal{Y}_{i,L_4})v_{L_5}.
\end{align*}
\begin{align*}
s_i s_{i+1} s_i v_{L_2} = s_{i+1} s_i s_{i+1}v_{L_2} =& \left(\frac{1}{(\kp_2-\kp_0)^2(\kp_4-\kp_0)}
    +\frac{1}{(\kp_2+\kp_0)^2(\kp_4+\kp_0)}+\frac{\mathcal{Y}_{i,L_1} \mathcal{Y}_{i,L_2}}{\kp_4-\kp_2}\right)v_{L_2}\\
&+\left(\frac{1}{(\kp_0^2-\kp_2^2)(\kp_4-\kp_0)}
    -\frac{1}{(\kp_0^2-\kp_2^2)(\kp_4+\kp_0)}\right)c_ic_{i+1}v_{L_2}\\
&+\left(\frac{-1}{(\kp_0^2-\kp_2^2)(\kp_4+\kp_0)}
    +\frac{1}{(\kp_0^2-\kp_2^2)(\kp_4-\kp_0)}\right)c_{i+1}c_{i+2}v_{L_2}\\
&+\left(\frac{1}{(\kp_0-\kp_2)^2(\kp_4+\kp_0)}
    +\frac{1}{(\kp_2+\kp_0)^2(\kp_4-\kp_0)}+\frac{\mathcal{Y}_{i,L_1} \mathcal{Y}_{i, L_2}}{\kp_4+\kp_2}\right)c_ic_{i+2}v_{L_2}\\
&+\left(\frac{\mathcal{Y}_{i,L_2}}{(\kp_0-\kp_2)(\kp_4-\kp_0)}
    +\frac{\mathcal{Y}_{i,L_2}}{(\kp_4-\kp_2)(\kp_2-\kp_0)}\right)v_{L_1}\\
&+\left(\frac{-\mathcal{Y}_{i,L_2}}{(\kp_2+\kp_0)(\kp_4+\kp_0)}
    +\frac{\mathcal{Y}_{i,L_2}}{(\kp_4-\kp_2)(\kp_0+\kp_2)}\right)c_ic_{i+1}v_{L_1}\\
&+\left(\frac{\mathcal{Y}_{i,L_2}}{(\kp_2+\kp_0)(\kp_4-\kp_0)}
    -\frac{\mathcal{Y}_{i,L_2}}{(\kp_4+\kp_2)(\kp_0+\kp_2)}\right)
    c_{i+1}c_{i+2}v_{L_1}\\
&+\left(\frac{\mathcal{Y}_{i,L_2}}{(\kp_0-\kp_2)(\kp_0+\kp_4)}
    +\frac{\mathcal{Y}_{i,L_2}}{(\kp_4+\kp_2)(\kp_0-\kp_2)}\right)c_{i}c_{i+2}v_{L_1}\\
&+\left(\frac{\mathcal{Y}_{i+1,L_2}}{(\kp_0-\kp_2)(\kp_4-\kp_2)}\right)v_{L_4}
    +\left(\frac{\mathcal{Y}_{i+1,L_2}}{(\kp_0-\kp_2)(\kp_4+\kp_2)}\right) c_ic_{i+1}v_{L_4}\\
&+\left(\frac{\mathcal{Y}_{i+1,L_2}}{(\kp_2+\kp_0)(\kp_4-\kp_2)}\right)
    c_{i+1}c_{i+2}v_{L_4}
    +\left(\frac{\mathcal{Y}_{i+1,L_2}}{(\kp_2+\kp_0)(\kp_4+\kp_2)}\right) c_ic_{i+2}v_{L_4}+\\
&+\left(\frac{\mathcal{Y}_{i+1,L_1}\mathcal{Y}_{i,L_4}}{\kp_0-\kp_2}\right)v_{L_5}
    +\left(\frac{\mathcal{Y}_{i+1,L_1}\mathcal{Y}_{i,L_4}}{\kp_2+\kp_0}\right) c_{i+1}c_{i+2}v_{L_5}\\
&+\left(\frac{\mathcal{Y}_{i,L_2}\mathcal{Y}_{i+1,L_1}}{\kp_4-\kp_0}\right)v_{L_3}
    +\left(\frac{\mathcal{Y}_{i,L_2}\mathcal{Y}_{i+1,L_1}}{\kp_4+\kp_0}\right) c_ic_{i+1}v_{L_3}
    +\left(\mathcal{Y}_{i,L_2} \mathcal{Y}_{i+1, L_1} \mathcal{Y}_{i, L_3}\right)v_{L_6}.
\end{align*}
\begin{align*}
s_i s_{i+1} s_i v_{L_3} = s_{i+1} s_i s_{i+1}v_{L_3} =& \left(\frac{1}{(\kp_4-\kp_0)^2(\kp_2-\kp_4)}
    +\frac{1}{(\kp_4+\kp_0)^2(\kp_4+\kp_2)}+\frac{\mathcal{Y}_{i,L_3} \mathcal{Y}_{i, L_6}}{\kp_2-\kp_0}\right)v_{L_3}\\
&+\left(\frac{1}{(\kp_4^2-\kp_0^2)(\kp_2-\kp_4)}
    -\frac{1}{(\kp_4^2-\kp_0^2)(\kp_4+\kp_2)}\right)c_ic_{i+1}v_{L_3}\\
&+\left(\frac{-1}{(\kp_4^2-\kp_0^2)(\kp_4+\kp_2)}
    +\frac{1}{(\kp_4^2-\kp_0^2)(\kp_2-\kp_4)}\right)c_{i+1}c_{i+2}v_{L_3}\\
&+\left(\frac{1}{(\kp_4-\kp_0)^2(\kp_4+\kp_2)}
    +\frac{1}{(\kp_4+\kp_0)^2(\kp_2-\kp_4)}+\frac{\mathcal{Y}_{i,L_3} \mathcal{Y}_{i, L_6}}{\kp_0+\kp_2}\right)c_i c_{i+2}v_{L_3}\\
&+\left(\frac{\mathcal{Y}_{i,L_3}}{(\kp_4-\kp_0)(\kp_2-\kp_4)}
    +\frac{\mathcal{Y}_{i,L_3}}{(\kp_2-\kp_0)(\kp_0-\kp_4)}\right)v_{L_6}\\
&+\left(\frac{-\mathcal{Y}_{i,L_3}}{(\kp_4+\kp_0)(\kp_4+\kp_2)}
    +\frac{\mathcal{Y}_{i,L_3}}{(\kp_2-\kp_0)(\kp_0+\kp_4)}\right) c_i c_{i+1}v_{L_6}\\
&+\left(\frac{\mathcal{Y}_{i,L_3}}{(\kp_4+\kp_0)(\kp_2-\kp_4)}
    -\frac{\mathcal{Y}_{i,L_3}}{(\kp_0+\kp_2)(\kp_0+\kp_4)}\right) c_{i+1} c_{i+2}v_{L_6}\\
&+\left(\frac{\mathcal{Y}_{i,L_3}}{(\kp_4-\kp_0)(\kp_2+\kp_4)}
    +\frac{\mathcal{Y}_{i,L_3}}{(\kp_0+\kp_2)(\kp_0-\kp_4)}\right)c_{i}c_{i+2}v_{L_6}\\
&+\left(\frac{\mathcal{Y}_{i+1,L_3}}{(\kp_4-\kp_0)(\kp_2-\kp_0)}\right)v_{L_1}
    +\left(\frac{\mathcal{Y}_{i+1,L_3}}{(\kp_4-\kp_0)(\kp_0+\kp_2)}\right) c_i c_{i+1}v_{L_1}\\
&+\left(\frac{\mathcal{Y}_{i+1,L_3}}{(\kp_4+\kp_0)(\kp_2-\kp_0)}\right)
    c_{i+1}c_{i+2}v_{L_1}+
    \left(\frac{\mathcal{Y}_{i+1,L_3}}{(\kp_4+\kp_0)(\kp_0+\kp_2)}\right) c_i c_{i+2}v_{L_1}\\
&+\left(\frac{\mathcal{Y}_{i+1,L_3}\mathcal{Y}_{i,L_1}}{\kp_4-\kp_0}\right)v_{L_2}
    +\left(\frac{\mathcal{Y}_{i+1,L_3}\mathcal{Y}_{i,L_1}}{\kp_4+\kp_0}\right) c_{i+1}c_{i+2}v_{L_2}\\
&\left(\frac{\mathcal{Y}_{i,L_3}\mathcal{Y}_{i+1,L_6}}{\kp_2-\kp_4}\right)v_{L_5}
    +\left(\frac{\mathcal{Y}_{i,L_3}\mathcal{Y}_{i+1,L_6}}{\kp_4+\kp_2}\right) c_ic_{i+1}v_{L_5}
    +(\mathcal{Y}_{i,L_3}\mathcal{Y}_{i+1,L_6}\mathcal{Y}_{i,L_3})v_{L_4}.
\end{align*}
\begin{align*}
s_i s_{i+1} s_i v_{L_4} = s_{i+1} s_i s_{i+1}v_{L_4} =& \left(\frac{1}{(\kp_4-\kp_2)^2(\kp_0-\kp_4)}+\frac{1}{(\kp_4+\kp_2)^2(\kp_4+\kp_0)}
    +\frac{\mathcal{Y}_{i,L_4}\mathcal{Y}_{i,L_5}}{\kp_0-\kp_2}\right)v_{L_4}\\
&+\left(\frac{1}{(\kp_4^2-\kp_2^2)(\kp_0-\kp_4)}
    -\frac{1}{(\kp_4^2-\kp_2^2)(\kp_4+\kp_0)}\right)c_ic_{i+1}v_{L_4}\\
&+\left(\frac{-1}{(\kp_4^2-\kp_2^2)(\kp_4+\kp_0)}
    +\frac{1}{(\kp_4^2-\kp_2^2)(\kp_0-\kp_4)}\right)c_{i+1}c_{i+2}v_{L_4}\\
&+\left(\frac{1}{(\kp_4-\kp_2)^2(\kp_4+\kp_0)}
    +\frac{1}{(\kp_4+\kp_2)^2(\kp_0-\kp_4)}
    +\frac{\mathcal{Y}_{i,L_4}\mathcal{Y}_{i,L_5}}{\kp_0+\kp_2}\right)
    c_ic_{i+2}v_{L_4}\\
&+\left(\frac{\mathcal{Y}_{i,L_4}}{(\kp_4-\kp_2)(\kp_0-\kp_4)}
    +\frac{\mathcal{Y}_{i,L_4}}{(\kp_0-\kp_2)(\kp_2-\kp_4)}\right)v_{L_5}\\
&+\left(\frac{-\mathcal{Y}_{i,L_4}}{(\kp_4+\kp_2)(\kp_4+\kp_0)}
    +\frac{\mathcal{Y}_{i,L_4}}{(\kp_0-\kp_2)(\kp_2+\kp_4)}\right)c_ic_{i+1}v_{L_5}\\
&+\left(\frac{\mathcal{Y}_{i,L_4}}{(\kp_4+\kp_2)(\kp_0-\kp_4)}
    -\frac{\mathcal{Y}_{i,L_4}}{(\kp_0+\kp_2)(\kp_2+\kp_4)}\right) c_{i+1} c_{i+2}v_{L_5}\\
&\left(\frac{\mathcal{Y}_{i,L_4}}{(\kp_4-\kp_2)(\kp_0+\kp_4)}
    +\frac{\mathcal{Y}_{i,L_4}}{(\kp_0+\kp_2)(\kp_2-\kp_4)}\right) c_{i} c_{i+2}v_{L_5}\\
&+\left(\frac{\mathcal{Y}_{i+1,L_4}}{(\kp_4-\kp_2)(\kp_0-\kp_2)}\right)v_{L_2}
    +\left(\frac{\mathcal{Y}_{i+1,L_4}}{(\kp_4-\kp_2)(\kp_0+\kp_2)}\right) c_i c_{i+1}v_{L_2}\\
&+\left(\frac{\mathcal{Y}_{i+1,L_4}}{(\kp_4+\kp_2)(\kp_0-\kp_2)}\right)
    c_{i+1}c_{i+2}v_{L_2}
    +\left(\frac{\mathcal{Y}_{i+1,L_4}}{(\kp_4+\kp_2)(\kp_0+\kp_2)}\right) c_i c_{i+2}v_{L_2}\\
&+\left(\frac{\mathcal{Y}_{i+1,L_4}\mathcal{Y}_{i,L_2}}{\kp_4-\kp_2}\right)v_{L_1}
    +\left(\frac{\mathcal{Y}_{i+1,L_4}\mathcal{Y}_{i,L_2}}{\kp_4+\kp_2}\right) c_{i+1} c_{i+2}v_{L_1}\\
&+\left(\frac{\mathcal{Y}_{i,L_4}\mathcal{Y}_{i+1,L_5}}{\kp_0-\kp_4}\right)v_{L_6}
    +\left(\frac{\mathcal{Y}_{i,L_4}\mathcal{Y}_{i+1,L_5}}{\kp_4+\kp_0}\right) c_i c_{i+1}v_{L_6}
    +(\mathcal{Y}_{i,L_4} \mathcal{Y}_{i+1, L_5} \mathcal{Y}_{i, L_4})v_{L_3}.
\end{align*}
\begin{align*}
s_i s_{i+1} s_i v_{L_5} = s_{i+1} s_i s_{i+1} v_{L_5} =& \left(\frac{1}{(\kp_4-\kp_2)^2(\kp_0-\kp_2)}
    +\frac{1}{(\kp_4+\kp_2)^2(\kp_2+\kp_0)}+\frac{\mathcal{Y}_{i,L_5} \mathcal{Y}_{i, L_4}}{\kp_0-\kp_4}\right)v_{L_5} \\
&+\left(\frac{1}{(\kp_2^2-\kp_4^2)(\kp_0-\kp_2)}
    -\frac{1}{(\kp_2^2-\kp_4^2)(\kp_2+\kp_0)}\right)c_ic_{i+1}v_{L_5}\\
&+\left(\frac{-1}{(\kp_2^2-\kp_4^2)(\kp_2+\kp_0)}
    +\frac{1}{(\kp_2^2-\kp_4^2)(\kp_0-\kp_2)}\right)c_{i+1}c_{i+2}v_{L_5}\\
&+\left(\frac{1}{(\kp_4-\kp_2)^2(\kp_2+\kp_0)}
    +\frac{1}{(\kp_4+\kp_2)^2(\kp_0-\kp_2)}+\frac{\mathcal{Y}_{i,L_5} \mathcal{Y}_{i, L_4}}{\kp_0+\kp_4}\right)c_ic_{i+2}v_{L_5}\\
&+\left(\frac{\mathcal{Y}_{i,L_5}}{(\kp_2-\kp_4)(\kp_0-\kp_2)}
    +\frac{\mathcal{Y}_{i,L_5}}{(\kp_0-\kp_4)(\kp_4-\kp_2)}\right)v_{L_4}\\
&+\left(\frac{-\mathcal{Y}_{i,L_5}}{(\kp_4+\kp_2)(\kp_2+\kp_0)}
    +\frac{\mathcal{Y}_{i,L_5}}{(\kp_0-\kp_4)(\kp_2+\kp_4)}\right)c_ic_{i+1}v_{L_4}\\
&+\left(\frac{\mathcal{Y}_{i,L_5}}{(\kp_4+\kp_2)(\kp_0-\kp_2)}
    -\frac{\mathcal{Y}_{i,L_5}}{(\kp_0+\kp_4)(\kp_2+\kp_4)}\right)
    c_{i+1}c_{i+2}v_{L_4}\\
&+\left(\frac{\mathcal{Y}_{i,L_5}}{(\kp_2-\kp_4)(\kp_0+\kp_2)}
    +\frac{\mathcal{Y}_{i,L_5}}{(\kp_0+\kp_4)(\kp_4-\kp_2)}\right) c_{i} c_{i+2}v_{L_4}\\
&+\left(\frac{\mathcal{Y}_{i+1,L_5}}{(\kp_2-\kp_4)(\kp_0-\kp_4)}\right)v_{L_6}
    +\left(\frac{\mathcal{Y}_{i+1,L_5}}{(\kp_4-\kp_2)(\kp_0+\kp_2)}\right)
    c_ic_{i+1}v_{L_6}\\
&+\left(\frac{\mathcal{Y}_{i+1,L_5}}{(\kp_4+\kp_2)(\kp_0-\kp_4)}\right)
    c_{i+1}c_{i+2}v_{L_6}
    +\left(\frac{\mathcal{Y}_{i+1,L_5}}{(\kp_4+\kp_2)(\kp_0+\kp_4)}\right) c_i c_{i+2}v_{L_6}\\
&+\left(\frac{(\mathcal{Y}_{i+1,L_5})(\mathcal{Y}_{i,L_6})}{\kp_2-\kp_4}\right)v_{L_3}
    +\left(\frac{(\mathcal{Y}_{i+1,L_5})(\mathcal{Y}_{i,L_6})}{\kp_4+\kp_2}\right) c_{i+1} c_{i+2}v_{L_3}\\
&+\left(\frac{(\mathcal{Y}_{i,L_5})(\mathcal{Y}_{i+1,L_4})}{\kp_0-\kp_2}\right)v_{L_2}
    +\left(\frac{(\mathcal{Y}_{i,L_5})(\mathcal{Y}_{i+1,L_4})}{\kp_2+\kp_0}\right) c_i c_{i+1}v_{L_6}
    +\left(\mathcal{Y}_{i,L_5}\mathcal{Y}_{i+1,L_4}\mathcal{Y}_{i,L_2}\right)v_{L_1}.
\end{align*}
\begin{align*}
s_i s_{i+1} s_i v_{L_6} = s_{i+1} s_i s_{i+1} v_{L_6} =& \left(\frac{1}{(\kp_0-\kp_4)^2(\kp_2-\kp_0)}
    +\frac{1}{(\kp_0+\kp_4)^2(\kp_2+\kp_0)}+\frac{\mathcal{Y}_{i,L_6} \mathcal{Y}_{i, L_3}}{\kp_2-\kp_4}\right) v_{L_6} \\
&+\left(\frac{1}{(\kp_0^2-\kp_4^2)(\kp_2-\kp_0)}
    -\frac{1}{(\kp_0^2-\kp_4^2)(\kp_2+\kp_0)}\right)c_ic_{i+1}v_{L_6}\\
&+\left(\frac{-1}{(\kp_0^2-\kp_4^2)(\kp_2+\kp_0)}
    +\frac{1}{(\kp_0^2-\kp_4^2)(\kp_2-\kp_0)}\right)c_{i+1}c_{i+2}v_{L_6}\\
&+\left(\frac{1}{(\kp_0-\kp_4)^2(\kp_2+\kp_0)}
    +\frac{1}{(\kp_4+\kp_0)^2(\kp_2-\kp_0)}
    +\frac{\mathcal{Y}_{i,L_6}\mathcal{Y}_{i,L_3}}{\kp_2+\kp_4}\right)
    c_ic_{i+2}v_{L_6}\\
&+\left(\frac{\mathcal{Y}_{i,L_6}}{(\kp_0-\kp_4)(\kp_2-\kp_0)}
    +\frac{\mathcal{Y}_{i,L_6}}{(\kp_2-\kp_4)(\kp_4-\kp_0)}\right)v_{L_3}\\
&+\left(\frac{-\mathcal{Y}_{i,L_6}}{(\kp_4+\kp_0)(\kp_2+\kp_0)}
    +\frac{\mathcal{Y}_{i,L_6}}{(\kp_2-\kp_4)(\kp_0+\kp_4)}\right)c_ic_{i+1}v_{L_3}\\
&+\left(\frac{\mathcal{Y}_{i,L_6}}{(\kp_0+\kp_4)(\kp_2-\kp_0)}
    -\frac{\mathcal{Y}_{i,L_6}}{(\kp_2+\kp_4)(\kp_0+\kp_4)}\right)
    c_{i+1}c_{i+2}v_{L_3}\\
&+\left(\frac{\mathcal{Y}_{i,L_6}}{(\kp_0-\kp_4)(\kp_0+\kp_2)}
    +\frac{\mathcal{Y}_{i,L_6}}{(\kp_2+\kp_4)(\kp_4-\kp_0)}\right) c_{i} c_{i+2}v_{L_3}\\
&+\left(\frac{\mathcal{Y}_{i+1,L_6}}{(\kp_0-\kp_4)(\kp_2-\kp_4)}\right)v_{L_5}
    +\left(\frac{\mathcal{Y}_{i+1,L_6}}{(\kp_0-\kp_4)(\kp_2+\kp_4)}\right) c_i c_{i+1} v_{L_5}\\
&+\left(\frac{\mathcal{Y}_{i+1,L_6}}{(\kp_0+\kp_4)(\kp_2-\kp_4)}\right)
    c_{i+1}c_{i+2}v_{L_5}
    +\left(\frac{\mathcal{Y}_{i+1,L_6}}{(\kp_0+\kp_4)(\kp_2+\kp_4)}\right)
    c_ic_{i+2}v_{L_5}\\
&+\left(\frac{\mathcal{Y}_{i+1,L_6}\mathcal{Y}_{i,L_5}}{\kp_0-\kp_4}\right)v_{L_4}
    +\left(\frac{\mathcal{Y}_{i+1,L_6}\mathcal{Y}_{i,L_5}}{\kp_0+\kp_4}\right) c_{i+1} c_{i+2}v_{L_4}\\
&\left(\frac{\mathcal{Y}_{i,L_6}\mathcal{Y}_{i+1,L_3}}{\kp_2-\kp_0}\right)v_{L_1}
    +\left(\frac{(\mathcal{Y}_{i,L_6})(\mathcal{Y}_{i+1,L_3})}{\kp_2+\kp_0}\right) c_i c_{i+1}v_{L_1}
    +\left(\mathcal{Y}_{i,L_6}\mathcal{Y}_{i+1,L_3}\mathcal{Y}_{i,L_6}\right)v_{L_2}.
\end{align*}
}
\begin{figure}[ht]
\begin{picture}(340,60)

\put(125,27){$L=$}
    \put(162,39){$i$}\put(178,39){$i+1$}\put(178,14){$i+2$}

\put(175,5){\line(1,0){25}}
    \put(175,5){\line(0,1){50}}\put(200,5){\line(0,1){50}}
\put(150,30){\line(1,0){50}}
    \put(150,30){\line(0,1){25}}
\put(150,55){\line(1,0){50}}

\end{picture}
\caption{Case 6}\label{F:Case 6}
\end{figure}

Case 6: Let $ L $ be as in Figure \ref{F:Case 6}.
Then
$$ s_i s_{i+1} s_i v_L = s_{i+1} s_i s_{i+1} v_L
=\frac{1}{\sqrt{2}}(-c_i c_{i+1} v_L + c_i c_{i+2} v_L). $$

\end{proof}

Now define an $ \A(d)$-module $ H^{\lambda / \mu} $ to be $ \sum_{w\in S_n}
\phi_w\L(c(L)) $ where  $ L $ is a fixed standard filling of the shifted skew shape
$ \lambda/ \mu $ and $ \mathcal{L}(c(L))=\L(c(L_1))\circledast\cdots\circledast\L(c(L_d)) $ is an irreducible $ \A(d)$
submodule of $ Cl(d) v_L $ introduced in section \ref{SS:characters}.

\begin{prp}
The  $ \mathcal{A}(d)$-module $ H^{\lambda / \mu} $ is a $ \ASe(d)$-module.
\end{prp}

\begin{proof}
Let $ c v_L \in  H^{\lambda / \mu}. $  Then $ (\phi_i - s_i(x_i^2-x_{i+1}^2)) c
v_L \in H^{\lambda / \mu}. $ Note that  $ \phi_i c v_L = {}^{s_i}c \phi_i v_L =
k {}^{s_i}c v_{s_i L} $ where $ {}^{s_i}c=s_ics_i$ denotes the Clifford element
twisted by $ s_i. $ This element is in $ H^{\lambda / \mu} $ because the
twisting of the Clifford element $ c $ by $ s_i $ is compatible with the
permutation of the zero eigenvalues of the $ x_j's $ by $ s_i. $ Thus $
s_i(x_i^2-x_{i+1}^2) c v_L = k' s_i c v_L \in  H^{\lambda / \mu}. $  Since by
construction $ (x_i^2-x_{i+1}^2) v_L \neq 0 $ by construction, $ s_i c v_L \in
H^{\lambda / \mu}. $
\end{proof}

\begin{thm}
For each shifted skew shape $ \lambda/\mu, $ $ H^{\lambda/\mu} $ is an
irreducible $ \ASe(d)$-module.  Every irreducible, calibrated $ \ASe(d)$-module
is isomorphic to exactly one such $ H^{\lambda/\mu}. $
\end{thm}

\begin{proof}
First to show that $ H^{(\lambda, \mu)} $ is irreducible. Let $ L $ be a
standard tableaux of shape $ \lambda / \mu. $ Let $ N $ be a non-zero submodule
of $ H^{\lambda / \mu} $ and let $ v = \Sigma_Q C_Q v_Q \in N $ be non-zero
where $ C_Q \in Cl(d). $ Let $ L $ be a standard tableaux such that $
\mathcal{Y}_L \neq 0. $ If $ P \neq L $ then there exists an $ i $ such that $
x_i v_P \neq x_i v_L. $  Suppose $ \mathcal{Y}_P \neq 0. $ Then $
\frac{x_i-\kp_{i,P}}{\kp_{i,L}-\kp_{i,P}} v $ no longer has a $ v_P $ term but
still has a $ v_L $ term.  This element is also in $ N. $ Iterating this
process it is clear that $ v_L \in N. $ The set of tableaux is identified with
an interval of $ S_n $ under the Bruhat order.  The minimal element is the
column reading $ C. $ Thus there exists a chain $ C < s_{i_1} C < \cdots <
s_{i_p} \cdots s_{i_1} C = L. $ Therefore $ \tau_{i_1} \cdots \tau_{i_p} v_L =
\kappa v_C $ for some non-zero complex number $ \kappa. $  This implies $ v_C
\in N. $ Now let $ Q $ be an arbitrary standard tableaux of $ \lambda / \mu. $
There is a chain $ C < s_{j_1} C < \cdots < s_{j_p} \cdots s_{j_1} C = Q. $
Then $ \tau_{j_p} \cdots \tau_{j_1} v_C = \kappa' v_Q $ for some non-zero
complex number $ \kappa'. $  Thus $ v_Q \in N $ so $ N = H^{\lambda/ \mu}. $

It is clear by looking at the eigenvalues that if $ \lambda / \mu \neq \lambda'
/ \mu', $ then $ H^{\lambda / \mu} \neq H^{\lambda' / \mu'}. $

Next to show that the weight of a calibrated module $ M $ is obtained by
reading the contents of a shifted skew shape via a standard filling. That is,
if $ (t_1, \ldots, t_d) $ be such a weight, then it is necessary to show that
it is equal to $ (c(L_1), \ldots, c(L_d)) $ for some standard tableaux $ L. $
It will be shown that if $ t_i = t_j $ for some $ i<j, $ then there exists $
k,l $ such that $ i<k<l<j $ such that $ t_k = t_i \pm 1 $ and $ t_l = t_i \mp 1
$ unless $ t_i =t_j= 0 $ in which case there is a $ k $ with $ i < k < j $ such
that $ t_k =1. $

Let $ j>i $ be such that $ t_j=t_i $ and $ j-i $ is minimal, let $m_t\in M$ be
anonzero vector of weight $t=(t_1,\ldots,t_d)$, and let
$\varrho_i=\sqrt{q(t_i)}$.  The proof will be by induction on $ j-i. $

\noindent\textbf{Case 1:} Suppose $ j-i=1. $

First the case that $ t_i = 0. $  If $ t_i =0, $ then $ t_{i+1}=0 $ by
assumption and then $ x_i s_i m_t = -m_t -c_i c_{i+1} m_t. $  It is clear that
$ -m_t - c_i c_{i+1} m_t \neq 0. $ Otherwise, $ m_t = -c_i c_{i+1} m_t $ which
implies after multiplying both sides by $ c_i c_{i+1} $ that $ m_t = -m_t $
giving $ m_t =0. $ Thus $ x_i^2 s_i m_t = 0 $ but $ x_i s_i m_t \neq 0. $
Similarly, $ x_{i+1}^2 s_i m_t = 0, $ but $ x_{i+1} s_i m_t \neq 0. $ Clearly $
(x_j - \varrho_j) s_i m_t = 0 $ for $ j \neq i, i+1. $ Thus if $ t_i = 0, $
then $ s_i m_t \in M^{\text{gen}}_t, $ but not in $ M_t $contradicting the
assumption that $ M $ is calibrated.

Now assume $ t_i \neq 0. $  Then, $ s_i m_t - \frac{1}{2t_i} c_i c_{i+1} m_t
\in M^{\text{gen}}_t $ but not in $ M_t. $ To see this, calculate:
$$ x_i(s_i m_t - \frac{1}{2\varrho_i}c_i c_{i+1} m_t) = t_i s_i m_t - \frac{1}{2} c_i c_{i+1} m_t - m_t. $$
This implies $ (x_i - \varrho_i)(s_i m_t - \frac{1}{2\varrho_i} c_i c_{i+1}
m_t) = -m_t \neq 0 $ and $ (x_i - \varrho_i)^2(s_i m_t - \frac{1}{2\varrho_i}
c_i c_{i+1} m_t) = 0. $ Similarly, $ (x_{i+1} - \varrho_{i+1})(s_i m_t -
\frac{1}{2\varrho_i} c_i c_{i+1} m_t) = m_t \neq 0 $ and $ (x_{i+1} -
\varrho_{i+1})^2(s_i m_t - \frac{1}{2\varrho_i} c_i c_{i+1} m_t) = 0. $ If $ j
\neq i, i+1, $ then $ (x_j -\varrho_j)(s_i m_t - \frac{1}{2\varrho_i} c_i
c_{i+1} m_t) =  0. $ Thus $ s_i m_t - \frac{1}{2\varrho_i} c_i c_{i+1} m_t \in
M^{\text{gen}}_t $ but not in $ M_t $ verifying case 1.

\noindent\textbf{Case 2:} Suppose $ j-i =2. $

Since $ m_t $ is a weight vector, the vector
$$ m_{s_i t}=\phi_i m_t =(\varrho_i - \varrho_{i+1}) s_i m_t
-(\varrho_i-\varrho_{i+1})c_i c_{i+1} m_t + (\varrho_i + \varrho_{i+1}) m_t $$
is a weight vector of weight $ t' = s_i t. $ Then $ t_{i+1}' = t_{i+2}'. $  By
case 1, this is impossible so $ m_{s_i t}=0. $ Note that $ \varrho_i +
\varrho_{i+1} \neq 0. $  If it did, then $ m_{s_i t} = 0 $ which would imply $
c_i c_{i+1} m_t = 0 $ which would imply $ m_t = 0. $ Thus, $ s_i m_t =
\frac{m_t}{\varrho_{i+1}-\varrho_i} + \frac{c_i c_{i+1}
m_t}{\varrho_{i+1}+\varrho_i}. $ Since $ s_i^2 m_t = m_t, $ it follows that $
m_t = (\frac{2(\varrho_i+\varrho_{i+1})}{(\varrho_i-\varrho_{i+1})^2}) m_t. $
This implies $ 2(\varrho_i + \varrho_{i+1}) = (\varrho_i-\varrho_{i+1})^2. $
The solutions of this equation are
$$ \varrho_{i+1} \in \lbrace \pm \sqrt{(t_i+1)(t_i+2)},\pm \sqrt{(t_i-1)(t_i)} \rbrace. $$
Since it is assumed that the positive square root is taken, there are only two
subcases to investigate. For the first subcase, assume $
\varrho_{i+1}=\sqrt{q(t_i+1)}. $ A routine calculation gives
$$ s_i s_{i+1} s_i m_t = \frac{-s_i m_t}{(\varrho_i-\varrho_{i+1})^2}
+ \frac{c_i c_{i+2} s_i m_t}{\varrho_{i+1}-\varrho_i} + \frac{c_{i+1}c_{i+2}
s_i m_t}{\varrho_i-\varrho_{i+1}} - \frac{c_i c_{i+1} s_i
m_t}{(\varrho_i+\varrho_{i+1})^2}. $$ From this it follows that the coefficient
of $ m_t $ is $
\frac{1}{(\varrho_i-\varrho_{i+1})^3}+\frac{1}{(\varrho_i+\varrho_{i+1})^3}. $
Similarly, from
$$ s_{i+1} s_{i} s_{i+1} m_t = \frac{-s_{i+1} m_t}{(\varrho_i-\varrho_{i+1})^2}
+ \frac{c_i c_{i+2} s_{i+1} m_t}{\varrho_i-\varrho_{i+1}}
+ \frac{c_{i}c_{i+1} s_{i+1} m_t}{\varrho_{i+1}-\varrho_i}
- \frac{c_{i+1} c_{i+2} s_{i+1} m_t}{(\varrho_i+\varrho_{i+1})^2} $$
it follows that the coefficient of $ m_t $ is $
\frac{-1}{(\varrho_i-\varrho_{i+1})^3}+\frac{-1}{(\varrho_i+\varrho_{i+1})^3}.
$ Therefore $ (\varrho_i-\varrho_{i+1})^3 + (\varrho_i+\varrho_{i+1})^3 =0. $
Recalling that $ \varrho_{i+1}=\sqrt{q(t_i+1)} $ in this subcase, it is clear
that $ t_i = t_{i+2}=0 $ and $ t_{i+1} = 1. $ The other subcase is similar.

Now for the induction step.  Assume $j-i>2$.  If $ t_{j-1} \neq t_j \pm 1, $
then the vector $ \phi_{j-1} m_t $ is a non-zero weight vector of weight $ t' =
s_{j-1} t $ by \cite[Lemma 14.8.1]{kl}. Since $ t_i' = t_i = t_j = t_{j-1}', $
the induction hypothesis may be applied to conclude that there exists $ k $ and
$ l $ with $ i < k < l < j-1 $ such that $ t_k' = t_j \pm 1 $ and $ t_l' = t_j
\mp 1. $  (In the case $ t_i = t_j = 0, $ then there exists $ t_k' =1. $) This
implies $ t_k  = t_j \pm 1 $ and $ t_l = t_j \mp 1. $  (In the case $ t_i = t_j
=0, $ there exists $ t_k = 1. $) Similarly, if $ t_{i+1} \neq t_i \pm 1, $
consider $ \phi_i m_t $ and proceed by induction. Otherwise, $ t_{i+1} = t_i
\pm 1 $ and $ t_{j-1} = t_i \pm 1. $  Since $ i $ and $ j $ are chosen such
that $ t_i = t_j $ and $ j-i $ is minimal, $ t_{i+1} \neq t_{j-1}. $  This then
gives the conclusion.  (If $ t_i = t_j = 0, $ then $ t_{i+1} = 1 $ or $ t_{j-1}
= 1. $

Suppose $ M $ is an irreducible, calibrated $ \ASe(d)$-module such that $ m_t $
is a weight vector with weight $ t = (t_1,\ldots,t_d) $ such that $ t_{i+1} =
t_i \pm 1. $  Then $ \phi_i m_t = 0.  $  This follows exactly as in step 5 of
\cite[Theorem 4.1]{ram}.

Finally, let $ m_t $ be a non-zero weight vector of an irreducible, calibrated
module $ M. $  By the above, $ t = (c(L_1), \ldots, c(L_d)) $ for $ L $ some
standard tableaux of shifted skew shape $ \lambda / \mu. $  The rest of the
proof follows as in step 6 of \cite[Theorem 4.1]{ram}. Choose a word $ w =
s_{i_p} \cdots s_{i_1} $ such that $ w $ applied to the column reading tableaux
of $ \lambda / \mu $ gives the tableaux $ L. $ Then $ m_C = \phi_{i_1} \cdots
\phi_{i_p} m_t $ is non-zero. Now to any other standard tableaux $ Q $ of $
\lambda / \mu $ there is a non-zero weight vector obtained by applying a
sequence of intertwiners to $ m_C. $  By the above, $ \phi_i m_Q = 0 $ if $ s_i
Q $ is not standard.  Thus the span of vectors $ \lbrace m_Q \rbrace $ over all
the standard tableaux of shape $ \lambda /\mu $ is a submodule of $ M. $ Since
$ M $ is irreducible, this span must be the entire module. Thus there is an
isomorphism $ M \cong H^{\lambda / \mu} $ defined by sending $ \phi_w m_C $ to
$ \phi_w v_C. $
\end{proof}

\begin{cor}\label{C:CalibratedSimples} Let $\ld/\mu$ be a shifted skew shape.
Then, $\L(\ld,\mu)\cong H^{\ld/\mu}$.
\end{cor}

\begin{proof} Let $T$ be the standard tableaux obtained by filling in the numbers
$1,\ldots,d$ along rows from top to bottom and left to right. Note that if
$s_i\in S_{\ld-\mu}$, then $v_{s_iT}=0$ because $s_iT$ is not standard. By
Frobenius reciprocity, it follows that there exists a surjective
$\ASe(d)$-homomorphism $f:\M(\ld,\mu)\rightarrow H^{\ld/\mu}$ given by
$f(\va_{\ld-\mu})=v_T$.
\end{proof}

Furthermore, by construction we have the following result.  Note that this agrees with Leclerc's conjectural formula for the calibrated simple modules of $\ASe (d)$ \cite[Proposition 51]{lec}.
\begin{cor}\label{C:characters}  Let $\lambda/\mu$ be a shifted skew shape.  Then,
\[
\ch \L (\lambda, \mu)= \sum_{L} \left[c(L_1),\ldots,c(L_d) \right],
\]
where the sum is over all standard fillings of the shape $\lambda / \mu$.
\end{cor}

\section{The Lie Superalgebras $\gl(n|n)$ and $\q(n)$}\label{S:Lie algebras}

\subsection{The Algebras}\label{SS:qndfn} Let $I=\{-n,\ldots,-1,1,\ldots,n\}$, and
$I^+=\{1,\ldots,n\}$. Let $V=\C^{n|n}$ be the $2n$-dimensional vector
superspace with standard basis $\{v_i\}_{i\in I}$. The standard basis for the
superalgebra $\End(V)$ is the set of matrix units $\{E_{ij}\}_{i,j\in I}$, and
the $\Z_2$-grading for $\End(V)$ and $V$ are given by
\[
p(v_k)=\zero,\;\;\;p(v_{-k})=\one,\andeqn p(E_{ij})=p(v_i)+p(v_j)
\]
for $k\in I^+$ and $i,j\in I$.

Let $C=\sum_{i,j\in I^+}(E_{-i,j}-E_{i,-j})$, and let $Q(V)\subset\End(V)$ be the
supercentralizer of $C$. Then, $Q(V)$ has basis given by elements
\[
e_{ij}=E_{ij}+E_{-i,-j},\andeqn f_{ij}=E_{-i,j}+E_{i,-j}\;\;\;i,j\in I^+.
\]
When $Q(V)$ and $\End(V)$ are viewed as Lie superalgebras relative to the
superbracket:
\[
[x,y]=xy-(-1)^{p(x)p(y)}yx,
\]
for homogeneous $x,y\in\End(V)$, we denote them $\q(n)$ and $\gl(n|n)$
respectively.

We end this section by introducing important elements of $\gl(n|n)$ that will
be needed later. Set
\begin{eqnarray}\label{bar-e/f}
\e_{ij}=E_{ij}-E_{-i,-j},\andeqn \f_{ij}=E_{-i,j}-E_{i,-j},\;\;\;i,j\in I^+.
\end{eqnarray}

\subsection{Root Data, Category $\mathcal{O}$, and Verma Modules}\label{SS:RootData}
Fix the triangular decomposition
\[
\q(n)=\n^-\oplus\h\oplus\n^+,
\]
 where
$\n^+_{\bar{0}}$ (resp. $\n^-_{\bar{0}}$) is the subalgebra spanned by the
$e_{ij}$ for $1\leq i<j\leq n$ (resp. $i>j$), $\h_\zero$ is spanned by the
$e_{ii}$, $1\leq i\leq n$, $\n^+_\one$ (resp. $\n^-_\one$) is the subalgebra
spanned by the $f_{ij}$ for $1\leq i<j\leq n$ (resp. $i>j$) and $\h_\one$ is
spanned by the $f_{ii}$, $1\leq i\leq n$.  Let $\b^+=\h\oplus\n^+$ and let $\b^-=\h\oplus\n^-$.

The isomorphism
$\q(n)_\zero\rightarrow\gl(n)$, $e_{ij}\mapsto E_{ij}$, identifies $\h_\zero$
with the standard torus for $\gl(n)$. Let $\ep_i\in\h_\zero^*$ denote the $i$th
coordinate function. For $i\neq j$, define $\af_{ij}=\ep_i-\ep_j$, and fix the
choice of simple roots $\Dt=\{\af_i=\af_{i,i+1}|1\leq i<n\}$. The corresponding
root system is $R=\{\af_{ij}|1\leq i\neq j\leq n\}$, and the positive roots are
$R^+=\{\af_{ij}|1\leq i<j\leq n\}$. The root lattice is
$Q=\sum_{i=1}^{n-1}\Z\af_i$ and weight lattice $P=\sum_{i=1}^n\Z\ep_i$. We can,
and will, identify $P=\Z^n$, and $Q=\{\ld\in P|\ld_1+\cdots+\ld_n=0\}$. Define
the sets of weights $P^+$, $\Pt$, $\Pr$, $\Pp$ and $\Ppos$ as in
$\S$\ref{SS:LieThy}. We call these sets dominant, dominant-typical, rational,
polynomial, and positive, respectively. Finally, let $\Prt = \Pr \cap \Pt$, and $\Ppt =\Pp \cap \Pt$.

To begin, let $\mathcal{O}:=\mathcal{O}(\q(n))$ denote the category of all
finitely generated $\q(n)$-supermodules $M$ that are locally finite dimensional
over $\b$ and satisfy
\[
M=\bigoplus_{\ld\in P}M_\ld
\]
where $M_\ld=\{\,v\in M \mid h.v=\ld(h)v\mbox{ for all }h\in\h_\zero\,\}$ is the
$\ld$-weight space of $M$.

We now define two classes of \emph{Verma modules}.  To this end, given
$\ld\in P$, let $\C_\ld$ be the 1-dimensional $\h_\zero$-module associated to
the weight $\ld$.  Let $\te_\ld:\h_\one\rightarrow\C$
be given by $\te_\ld(k)=\ld([k,k])$ for all $k\in\h_\one$. Let  $\h_\one'=\ker\te$. Let
$\overline{\U(\h)}=\U(\h)/\mathfrak{i}$, where $\mathfrak{i}$ is the left ideal
of $U(\h)$ generated by $\{\,h-\ld(h) \mid h\in\h_\zero\,\}\cup\h_\one'$. Recall, $\gm_0(\ld)=|\{\,i \mid \ld_i=0\,\}|$. Since
$\overline{\U(\h)} $ is isomorphic to a Clifford algebra of rank $
n-\gm_0(\ld), $ we can define the $\overline{\U(\h)}$-modules $ C(\ld) $ and $
E(\ld) $ where $ C(\ld) $ is the regular representation of the resulting
Clifford algebra and $ E(\ld) $ is its unique irreducible quotient. Both $ C(\ld) $ and $ E(\ld) $ become modules for $\U(\h)$ via
inflation through the canonical projection $\U (\h)\to \overline{\U(\h)}$. Note that as a $ \U(\h)$-module, $ C(\ld) \cong
\ind_{\U(\h_\zero+\h_\one')}^{\U(\h)}\C_\ld. $ Extend $C(\ld)$ and $E(\ld)$ to
representations of $\U(\b^+)$ by inflation, and define the \emph{Big Verma}
$\widehat{M}(\ld)$ and \emph{Little Verma} $M(\ld)$ by
\[
\widehat{M}(\ld)=\ind_{\U(\b^+)}^{\U(\q(n))}C(\ld)\andeqn
M(\ld)=\ind_{\U(\b^+)}^{\U(\q(n))}E(\ld).
\]
The following lemma is obtained from the standard decomposition of the Clifford
algebra into irreducible modules:

\begin{lem}\label{L:little verma in big verma}
We have $\widehat{M}(\ld) \cong M(\ld)^{\oplus 2^{\lfloor
\frac{n-\gm_0(\ld)}{2} \rfloor}}. $
\end{lem}

It is known that $M(\ld)$ has a unique irreducible quotient $L(\ld)$ (see, for
example, \cite{g}). Moreover, it is known $L(\ld)$ is finite dimensional if,
and only if, $\ld\in\Pr$ (see \cite{p}).

The following lemma seems standard, but we cannot find it stated in the
literature. See \cite[Corollary 7.1, 11.6]{g} for related statements.  If $M$ is a $\U (\fq)$-module, then recall that a vector $m \in M$ is called \emph{primitive} if $\n^{+}v=0$.

\begin{lem}\label{L:InjHom} Let $\ld\in P$,
and assume that for some $\af\in R^+$, there exists $r>0$ such that
$s_\af\ld=\ld-r\af$. Then, there exists an injective homomorphism
\[
M(s_\af\ld)\rightarrow M(\ld).
\]
\end{lem}

\begin{proof}
Let $\af=\af_{ij}$, and let $v_\ld\in M(\ld)_\ld$ be an odd primitive
vector. Then, direct calculation verifies that
\[
v_{\ld-r\af}:=(e_{ji}^{r-1}(rf_{ji}-e_{ji}(f_{ii}-f_{jj})).v_\ld
\]
is a primitive vector of weight $\ld-r\af$ (see, for example \cite[Corollary 7.1]{g}). This implies that there is an
injective $\U(\b^+)$-homomorphism
\[
E(s_\af\ld)\to\U(\h).v_{\ld-r\af}.
\]
Indeed, clearly every vector in $\U(\h).v_{\ld-r\af}$ has weight $\ld-r\af$. Moreover, if $N\in\U(\n^+)$ and $H\in\U(\h)$, then $[N,H]\in\U(\n^+)$, so
\[
N.(H.v_{\ld-r\af})=(HN+[N,H]).v_{\ld-r\af}=0.
\]
The result follows because, by our choice of primitive vector, a standard argument using the filtration of $\U (\fq(n) )$ by total degree and a calculation in $U(\fq (2)$ shows that $\U(b^-).v_{\ld-r\af}$ is a free $\U(\n^-)$-module.
\end{proof}

\subsection{The Shapovalov Form}\label{SS:ShapovalovForm} The Shapovalov map for $\q(n)$ was constructed in
\cite{g}. We review this construction briefly.

Let $\mathcal{D}$ be the category of $Q^-=-Q^+$-graded $\q(n)$-modules with
degree 0 with respect to this grading. We regard the big
and little Verma's as objects in this category by declaring $\deg
M(\ld)_{\ld-\nu}=-\nu$ for all $\nu\in Q^+$. Let $\mathcal{C}$ be the category
of left $\h$-modules.

Let $\Psi_0:\mathcal{D}\rightarrow\mathcal{C}$ be the functor $\Psi_0(N)=N_0$
(i.e.\ the degree 0 component). The functor $\Psi_0$ has a left adjoint
$\ind:\mathcal{C}\rightarrow\mathcal{D}$ given by $\ind A=\ind_{\b^+}^{\q(n)}
A$, where we regard the $\h$-module $A$ as a $\b^+$-module by inflation. The
functor $\Psi_0$ also has an exact right adjoint $\coind$ (see \cite[Proposition
4.3]{g}).

As in \cite{g}, let $\Theta(A):\ind A\rightarrow\coind A$ be the morphism
corresponding to the identity map $\mathrm{id}_A:A\rightarrow A$. This induces
a morphism of functors $\Theta:\ind\rightarrow\coind$. The main property we
will use is

\begin{thm}\cite[Proposition 4.4]{g} We have $\ker\Theta(A)$ is the maximal
graded submodule of $\ind A$ which avoids $A$.
\end{thm}

Define the Shapovalov map $S:=\Theta(\U(\h)):\ind(\U(\h)) \rightarrow
\coind(\U(\h))$. Given an object $A$ in $\mathcal{C}$, proposition 4.3 of
\cite{g} shows there is a canonical isomorphism $\ind A\cong
\ind\U(\h)\otimes_{\U(\h)}A$ and $\coind A\cong\coind\U(\h)\otimes_{\U(\h)}A$.
In this way, we may identify $\Theta(A)$ with
$\Theta(\U(\h))\otimes_{\U(\h)}\mathrm{id}_A$. It follows that the map
$\Theta(A)$ is completely determined by the Shapovalov map.

In order to describe $S$ in more detail, we introduce some auxiliary data. Let
$\varsigma:\U(\q(n))\rightarrow\U(\q(n))$ be the antiautomorphism defined by
$\varsigma(x)=-x$ for all $x\in\q(n)$ and extended to $\U(\q(n))$ by the rule
$\varsigma(xy)=(-1)^{p(x)p(y)}\varsigma(y)\varsigma(x)$ for $x,y\in\U(\q(n))$.
Also, define the Harish-Chandra projection $HC:\U(\q(n))\rightarrow\U(\h)$
along the decomposition
\[
\U(\q(n))=\U(\h)\oplus(\U(\q(n))\n^++\n^-\U(\q(n))).
\]

Now, we may naturally identify $\ind\U(\h)\cong\U(\b^-)$ as $ (\b^-,\h) $-bimodules. The
$Q^-$-grading on $\U(\b^-)$ is given by
\begin{eqnarray}\label{E:Q grading of bminus}
\U(\b^-)_{-\nu}=\{\,x\in\U(\b^-) \mid [h,x]=-\nu(h)x\mbox{ for all
}h\in\h_\zero\,\}
\end{eqnarray}
for all $\nu\in Q^+$.

To describe $\coind\U(\h)$, let $\mathcal{D}_+$ be the category of $Q^+$ graded
submodules and $\ind_+$ be the left adjoint to the functor
$\Psi_0^+:\mathcal{C}\rightarrow\mathcal{D}_+$. We may naturally identify
$\ind_+\U(\h)\cong\U(\b^+)$ as $ (\b^+,\h) $-bimodules  and $\U(\b^+)$ has a $Q^+$-grading analogous to
\eqref{E:Q grading of bminus}. Now, let $\U(\h)^\varsigma$ be the
$(\h, \h)$-bimodule obtained by twisting the action of $\h$ with $\varsigma$. That is,
$h.x=(-1)^{p(h)p(x)} \varsigma(h)x$ and $ x.h = (-1)^{p(h)p(x)} x \varsigma(h)$
for all $x\in\U(\h)^\varsigma$ and $h\in\h$. Then, there is
a natural identification of $\coind\U(\h)$ with the graded dual of $\U(\b^+)$ as $ (\U\g, \U\h) $ -bimodules:
\[
\coind\U(\h)\cong\U(\b^+)^{\#} :=\bigoplus_{\nu\in
Q^+}\Hom_{\mathcal{C}}(\U(b^+)_\nu,\U(\h)^\varsigma),
\]
see \cite[Proposition 4.3(iii)]{g}. Observe that $\U(\b^+)^{\#}$ has a $Q^-$
grading given by
$\U(\b^+)^{\#}_{-\nu}=\Hom_{\mathcal{C}}(\U(b^+)_\nu,\U(\h)^\varsigma)$.

Using these identifications, we realize the Shapovalov map via the formula:
\[
S(x)(y)=(-1)^{p(x)p(y)}HC(\varsigma(y)x),
\]
for $x\in\U(\q(n))$ and $y\in\U(\q(n))$, \cite[$\S$4.2.4, Claim 3]{g}.

The Shapovalov map is homogeneous of degree 0. Therefore, $S=\sum_{\nu\in
Q^+}S_\nu$, where $S_\nu:\U(\b^-)_{-\nu}\rightarrow\U(\b^+)^{\#}_{-\nu}$ is
given by restriction.

For our purposes, it is more convenient to introduce a bilinear form
\[
(\cdot,\cdot)_S:\U(\q(n))\otimes\U(\q(n))\rightarrow\U(\h)
\]
with the property that $\Rad(\cdot,\cdot)_S=\ker S$. To do this we introduce the (non-super)
\emph{transpose} antiautomorphism $\tau:\U(\q(n))\rightarrow\U(\q(n))$ given by
$\tau(x)=x^t$ if $x\in\q(n)$ and extend to $\U(\q(n))$ by
$\tau(xy)=\tau(y)\tau(x)$. Note that this is the ``naive'' antiautomorphism introduced in \cite{g}.
Define $(\cdot,\cdot)_S$ by
\[
(u,v)_S=(-1)^{p(u)p(v)}S(v)(\varsigma\tau(u))=HC(\tau(u)v)
\]
for all $u,v\in\U(\q(n))$.
%This is the proof that the form is symmetric, but it isn't correct.
%Now,
%\[
%HC(\tau(u)v)=S(\tau(u)v)(1)=(1,\tau(u)v)_S
%\]
%and
%\[
%HC(\tau(u)v)=HC(\tau(\tau(v)u))=S(1)(\varsigma\tau(\tau(v)u))=(\tau(v)u,1)_S.
%\]
%The rightmost formulae above are well defined if we interpret
%$(1,x)_S=(x,1)_S=0$ if $x\in\U(\q(n))\n^+$.

%Now,
%\[
%(v,u)_S=HC(\tau(v)u)=S(\tau(v)u)(1)=(1,\tau(v)u)_S
%\]
%so $(u,v)_S=(v,u)_S$, and this form is symmetric (in the non-super sense!).

\begin{prp}
The radical of the form may be identified as: $\Rad(\cdot,\cdot)_S=\ker S$.
\end{prp}

\begin{proof} Assume $u\in\ker S$ and $v\in\U(\b^-)$. Then, $\tau(v)\in\U(\b^+)$
and
\[
(\tau\varsigma(v),u)_S
=(-1)^{p(u)p(v)}S(u)(\varsigma\tau\tau\varsigma(v))=(-1)^{p(u)p(v)}S(u)(v)=0,
\]
showing that $u\in\Rad(\cdot,\cdot)_S$.

Conversely, assume $u\in\Rad(\cdot,\cdot)_S$ and $v\in\U(\b^+)$. Then,
$\tau\varsigma(v)\in\U(\b^-)$ and
\[
0=(\tau\varsigma(v),u)_S
=(-1)^{p(u)p(v)}S(u)(\varsigma\tau\tau\varsigma(v))=(-1)^{p(u)p(v)}S(u)(v).
\]
Hence, $u\in\Ker S$.
\end{proof}

\begin{rmk} We have already defined $\tau$ to be an antiautomorphism of the
AHCA. We will show the compatibility of the two
anti-automorphisms in Proposition \ref{P:when tau's collide}.
\end{rmk}

\section{A Lie-Theoretic construction of $\ASe(d)$}\label{S:LieTheoreticConstr}
Let $X$ be a $\fq(n)$-supermodule. In this section we construct a homomorphism of
superalgebras
\[
\ASe(d)\rightarrow\End_{\fq(n)}(X\otimes V^{\otimes d})
\]
along the lines of Arakawa and Suzuki, \cite{as}. The main
difficulty is the lack of an even invariant bilinear form, and
consequently, a lack of a suitable Casimir element in
$q(n)^{\otimes2}$. However, we find inspiration for a suitable
substitute in Olshanski's work in the quantum setting \cite{o}.

\subsection{Lie Bialgebra structures on $\fq(n)$}
We begin by reviewing the construction of a Manin triple for $\fq(n)$
from \cite{o} (see also \cite{d1}). A Manin triple $(\p,\p_1,\p_2)$
consists of a Lie superalgebra $\p$, a nondegenerate even invariant
bilinear symmetric form $B$ and two subalgebras $\p_1$ and $\p_2$
which are $B$-isotropic transversal subspaces of $\p$. Then, $B$
defines a nondegenerate pairing between $\p_1$ and $\p_2$.

Define a cobracket $\Dt:\p_1\rightarrow \p_1^{\otimes2}$ by
dualizing the bracket $\p_2^{\otimes 2}\rightarrow\p_2$:
\[
B^{\otimes2}(\Dt(X),Y_1\otimes Y_2)=B(X,[Y_1,Y_2]),\;\;\;(X\in\p_1).
\]
Then, the pair $(\p_1,\Dt)$ is called a Lie (super)bialgebra.

Choose a basis $\{X_\af\}$ for $\p_1$ and a basis $\{Y_\af\}$ for
$\p_2$ such that $B(X_\af,Y_\bt)=\dt_{\af\bt}$, and set $s=\sum_\af
X_\af\otimes Y_\af$. Then, it turns out that $s$ satisfies the
classical Yang-Baxter equation
\[
[s^{12},s^{13}]+[s^{12},s^{23}]+[s^{13},s^{23}]=0
\]
and $\Dt(X)=[1\otimes X+X\otimes 1,s]$, for $X\in\p_1$.

\subsection{The Super Casimir}  Note that when $\p=\g$ is a simple Lie algebra,
$\p_1=\mathfrak{b}_+$, $\p_2=\mathfrak{b}_-$ are the positive and
negative Borel subalgebras and $B$ is the trace form, $s$ becomes
the classical $r$-matrix, which we will denote $r^{12}$. We can
repeat this construction with the roles of $\p_1$ and $\p_2$
reversed and obtain another classical $r$-matrix which we denote
$r^{21}$. Then, the Casimir is simply $\Om=r^{12}+r^{21}$, see
\cite{as} $\S1.2$.

In \cite{o}, Olshanski constructs such an element $s$ for $\p=\gl(n|n)$,
$\p_1=\fq(n)$ and some fixed choice of $\p_2$ analogous to a positive Borel. We
will review this construction to obtain an element which we will call $s_+$,
then replace $\p_2$ with an analogue of a negative Borel to obtain another
element called $s_-$. Then, we show that the element $\Om=s_++s_-$ performs the
role of the Casimir in our setting.

\begin{dfn}Let $\p=\gl(n|n)$, $B(x,y)=\str(xy)$ (where
$\str(E_{ij})=\dt_{ij}\mathrm{sgn}(i)$ for $i,j\in I$), and
$\p_1=\q(n)$.

\begin{enumerate}
\item Let
\[
\p_2^+=\sum_{i\in I^+}\C(E_{ii}-E_{-i,-i})+\sum_{\substack{i,j\in
I,\\
i<j}}\C E_{ij}.
\]
Then the corresponding element $s_+$ is given by
\[
s_+=\frac12\sum_{i\in
I^+}e_{ii}\otimes\e_{ii}+\sum_{\substack{i,j\in
I^+\\i>j}}e_{ij}\otimes E_{ji}-\sum_{\substack{i,j\in
I^+\\i<j}}e_{ij}\otimes E_{-j,-i}-\sum_{i,j\in I^+}f_{ij}\otimes
E_{-j,i}.
\]
\item Let
\[
\p_2^-=\sum_{i\in I^+}\C(E_{ii}-E_{-i,-i})+\sum_{\substack{i,j\in
I,\\
i>j}}\C E_{ij}.
\]
Then, the corresponding element $s_-$ is given by
\[
s_-=\frac12\sum_{i\in
I^+}e_{ii}\otimes\e_{ii}-\sum_{\substack{i,j\in
I^+\\i>j}}e_{ij}\otimes E_{-j,-i}+\sum_{\substack{i,j\in
I^+\\i<j}}e_{ij}\otimes E_{j,i}+\sum_{i,j\in I^+}f_{ij}\otimes
E_{j,-i}.
\]
\end{enumerate}
\end{dfn}

We now define our substitute Casimir:
\begin{eqnarray}\label{casimir}
\Om=s_++s_-=\sum_{i,j\in I^+}e_{ij}\otimes\e_{ji}-\sum_{i,j\in
I^+}f_{ij}\otimes\f_{ji}\in Q(V)\otimes\End(V),
\end{eqnarray}
where $\e_{ij}$ and $\f_{ij}$ are given in \eqref{bar-e/f}.

\subsection{Classical Sergeev Duality}\label{SS:Sergeev Duality} We now need to recall Sergeev's duality between $\Se (d)$ and $\fq (n)$.  Recall the matrix
$C=\sum_{i\in I^+}\f_{ii}$ from the previous section, and define the
superpermutation operator
\[
S=\sum_{i,j\in I}\mathrm{sgn}(j)E_{ij}\otimes
E_{ji}\in\End(V)^{\otimes2},
\]
where $\mathrm{sgn}(j)$ is the sign of $j$. Let
$\pi_i:\End(V)\rightarrow\End(V)^{\otimes d}$ be given by $\pi_i(x)=1^{\otimes
i-1}\otimes x\otimes 1^{\otimes d-i}$ for all $x\in\End(V)$ and $i=1,\ldots,
d$; similarly, define $\pi_{ij}:\End(V)^{\otimes 2}\rightarrow\End(V)^{\otimes
d}$ by $\pi_{ij}(x\otimes y)=1^{\otimes i-1}\otimes x\otimes 1^{\otimes
j-i-1}\otimes y\otimes 1^{\otimes d-j}$. Set $C_i=\pi_i(C)$ and, for $1\leq
i<j\leq d$, set $S_{ij}=\pi_{ij}(S)$. Then,

\begin{thm}\label{Sergeev Duality Theorem}\cite[Theorem 3]{s} The map which sends
$c_i\mapsto C_i$ and $s_i\mapsto S_{i,i+1}$ is an isomorphism of superalgebras
\[
\Se(d)\rightarrow\End_{\fq(n)}(V^{\otimes d}).
\]
\end{thm}

\subsection{$\ASe(d)$-action}\label{SS:action} Let $M$ be a $\q(n)$-supermodule.
In this section we construct an action of $\ASe(d)$ on $M\otimes
V^{\otimes d}$ that commutes with the action of $\q(n)$. To this
end, extend the map $\pi_i$ from $\S$\ref{SS:Sergeev Duality} to a map
$\pi_i:\End(V)\rightarrow\End(V)^{\otimes d+1}$ so that
$\pi_i(x)=1^{\otimes i}\otimes x\otimes 1^{\otimes d-i}$ for $x\in
\End(V)$ and $i=0,\ldots,d$ (i.e.\ add a 0th tensor place); similarly, extend $\pi_{ij}$.

Define $C_i$ and $S_{ij}$ as in $\S$\ref{SS:Sergeev Duality}. Define
\[
\Om_{ij}=\pi_{ij}(\Om)\;\;\;0\leq i<j\leq d
\]
and set $X_i=\Om_{0i}+\sum_{1\leq j<i}(1-C_jC_i)S_{ji}$.

\begin{thm}\label{Affine Sergeev Action} Let $M$ be a $\q(n)$-supermodule.
Then, the map which sends $c_i\mapsto C_i$, $s_i\mapsto S_{i,i+1}$
and $x_i\mapsto X_i$ defines a homomorphism
\[
\ASe(d)\rightarrow\End_{\fq(n)}(M\otimes V^{\otimes d}).
\]
\end{thm}

\begin{proof} It is clear from Theorem \ref{Sergeev Duality Theorem} that the $C_i$
and $S_{i,i+1}$ form a copy of the Sergeev algebra $\Se(d)$ inside
$\End_{\q(n)}(M\otimes V^{\otimes d})$ via the obvious embedding
$\End_{\q(n)}(V^{\otimes d})\hookrightarrow\End_{\q(n)}(M\otimes
V^{\otimes d})$, $A\mapsto\mathrm{id}_M\otimes A$. Moreover, for
$i=1,\ldots,d$, $X_i\in\End(M\otimes V^{\otimes d})$, since
$X_i\in Q(n)\otimes\End(V)^{\otimes d}$. Therefore it is enough to
check the following properties:
\begin{enumerate}
\item[(a)] The $X_i$ satisfy the mixed relations \eqref{c&x} and
\eqref{s&x},
\item[(b)] $X_iX_j-X_jX_i=0$, and
\item[(c)] the $X_i$ commute with the action of $\q(n)$ on $M\otimes
V^{\otimes d}$.
\end{enumerate}
First, we check that $\Om(1\otimes C)=-(1\otimes C)\Om$. To do this, a
calculation shows that $C\bar{e}_{ji}=-\bar{e}_{ji}$ and
$C\bar{f}_{ji}=\bar{f}_{ji}C$. Hence,
\[
(1\otimes C)(e_{ij}\otimes\bar{e}_{ji})=-(e_{ij}\otimes\bar{e}_{ji})(1\otimes C)
\]
and
\begin{eqnarray*}
(1\otimes C)(f_{ij}\otimes\bar{f}_{ji})&=&(-1)^{p(f_{ij})p(C)}
(f_{ij}\otimes C\bar{f}_{ji})\\
&=&(-1)^{p(\bar{f}_{ji})p(C)}(f_{ij}\otimes \bar{f}_{ji}C)\\
&=&(-1)^{p(\bar{f}_{ji})p(C)+p(1)p(\bar{f}_{ji})}(f_{ij}\otimes
\bar{f}_{ji})(1\otimes C),
\end{eqnarray*}
so the result follows since $p(1)=\zero$. Next, it is easy to see that
$S_i\Om_{0i}S_i=\Om_{i+1}$ using \eqref{tensor product rule-algebra}.
Therefore, (a) follows from the definition of $X_i$. It is now easy
to show that, for $i<j$, (b) is equivalent to
\[
\Om_{0i}\Om_{0j}-\Om_{0j}\Om_{0i}=(\Om_{0j}-\Om_{0i})S_{ij}+
(\Om_{0j}+\Om_{0i})C_iC_jS_{ij}.
\]
This equality is then a direct calculation. Finally, to verify (c),
it is enough to show that for any $X\in\q(n)$,
\[
[1\otimes X+X\otimes1,\Om]=0.
\]
This is another routine calculation using \eqref{tensor product
rule-algebra}.
\end{proof}

Now, recall the ``naive'' antiautomorphism $\tau:\U(\q(n))\rightarrow\U(\q(n))$. This extends to an antiautomorphism of $\U(\gl(n|n))$. Extend $\tau$ to an antiautomorphism of $\U(\gl(n|n))^{\otimes 2}$ by $\tau(x\otimes y)=(-1)^{p(x)}\tau(x)\otimes\tau(y)$. By induction, extend $\tau$ to an antiautomorphism of $\U(\gl(n|n))^{\otimes k}$ by $\tau(x_1\otimes\cdots\otimes x_k)=(-1)^{p(x_1)}\tau(x_1)\otimes\tau(x_2\otimes\cdots\otimes x_k)$.  A direct check verifies the following result.

\begin{prp}\label{P:when tau's collide} We have that $\tau(C_i)=-C_i$, $\tau(S_{i,i+1})=S_{i,i+1}$ and $\tau(X_i)=X_i$ for all admissible $i$'s. In particular, the antiautomorphism $\tau^{\otimes d+1}:\U(\gl(n|n))^{\otimes d+1}\rightarrow\U(\gl(n|n))^{\otimes d+1}$ coincides with the antiautomorphism $\tau:\ASe(d)\rightarrow\ASe(d)$.
\end{prp}

\subsection{The Functor $F_\ld$}\label{SS:Flambda} In the previous section, we showed that there is a
homomorphism from $\ASe(d)$ to $\End_{\q(n)}(M\otimes V^{\otimes d})$. Since
the action of $\ASe(d)$ on $M\otimes V^{\otimes d}$ commutes with the action of
$\q(n)$, it preserves both primitive vectors and weight spaces. By \emph{primitive vector} we mean an element of $M\otimes V^{\otimes d}$ which is annihilated by the subalgebra $\n^{+}$ given by the triangular decomposition of $\fq (n)$ as in Section~\ref{SS:RootData}.  Therefore, given a weight $\ld\in P(M\otimes V^{\otimes d})$ we have an action of
$\ASe(d)$ on

\begin{equation}\label{E:Flambdadef}
F_\ld M :=\left\{m \in  M\otimes V^{\otimes d} \mid \n^+.m=0 \text{ and } m \in \left(M\otimes V^{\otimes d} \right)_{\lambda} \right\}
\end{equation}

In the case when $\ld\in\Pt$ we can provide alternative descriptions of the
functor $F_{\lambda}$. First we recall the following key result of Penkov
\cite{p}.  Given a weight $\lambda \in P$, we write $\chi_{\lambda}$ for the
central character defined by the simple $\fq (n)$-module of highest weight
$\lambda$. Then, there is a block decomposition
\begin{equation}\label{E:blockdecomp}
\mathcal{O}(\q(n))=\bigoplus_{\chi_{_{\lambda}}}\mathcal{O}(\q(n))^{[\ld]}
\end{equation}
where the sum is over all central characters $\chi_{\lambda}$ and
$\mathcal{O}(\q(n))^{[\ld]}=\mathcal{O}(\q(n))^{[\chi_\ld]}$ denotes the block
determined by the central character $\chi_\ld$.  Given $N$ in $\mathcal{O}(\q(n))$, let $N^{[\chi_{\gamma}]}=N^{[\gamma]}$ denote the projection of $N$ onto the direct summand which lies in $\mathcal{O}(\q(n))^{[\chi_\gamma]}$

The question then becomes to
describe when $\chi_{\lambda}=\chi_{\mu}$ for $\lambda, \mu \in P$. This is
answered in the case when $\lambda$ is typical by the following result of
Penkov \cite{p}. Recall that the symmetric group acts on $P$ by
permuation of coordinates.
\begin{prp}\label{P:penkov}  Let $\lambda \in \Pt $ be a typical weight
and let $\mu \in P$.  Then $\chi_{\lambda}=\chi_{\mu}$ if and only if $\mu
=w(\lambda)$ for some $w \in S_n.$
\end{prp}

For short we call a weight $\lambda \in P$ \emph{atypical} if it is not typical.  By the description of the blocks $\mathcal{O}(\q(n))^{[\ld]}$, if $L(\mu)$ is an object of $\mathcal{O}(\q(n))^{[\ld]}$ then $\lambda$ is typical if and only if $\mu$ is typical (c.f.\ \cite[Proposition 1.1]{ps2} and the remarks which follow it).  We then have the following preparatory lemma.

\begin{lem}  Let $\lambda,\gamma \in P$.  Then the following statements hold:
\begin{enumerate}
\item [(i)] Assume $\gamma$ is atypical and $\lambda$ is typical.  If $N$ is an object of $\mathcal{O}^{[\gamma]}$, then
$N_{\lambda}=(\fnminus N)_{\lambda}.$
\item [(ii)] Assume $\lambda, \gamma \in P^{++}$ are typical and dominant and $\lambda \neq  \gamma$.  If $N$ is an object of $\mathcal{O}^{[\gamma]}$, then $N_{\lambda}=(\fnminus N)_{\lambda}$.
\end{enumerate}
\end{lem}

\begin{proof} By \cite[Lemma 4.5]{b}, every object $\mathcal{O}(\q(n))$ has a finite Jordan-H\"{o}lder series.   The proof of (i) is by induction on the length of a composition series
of $N.$  The base case is when $N$ has length one (ie.\ $N \cong L(\nu)$ is a
simple module).  This case immediately follows from the fact that in order for $N_{\lambda}$ to be nontrivial it must be that $\lambda < \nu$.  But then it follows from the assumption that $\nu$ is atypical (since $L(\nu)$ is an object of
$\mathcal{O}^{[\gamma]}$) while $\lambda$ is typical.  Now consider a composition
series
\[
0=N_{0} \subset N_{1} \subset \dotsb \subset  N_{t} =N.
\]  Let $v \in N_{\lambda}$ so that $v +N_{t-1} \in N_{t}/N_{t-1}$ is nonzero.
Since $N_{t}/N_{t-1}$ is a simple module in $\mathcal{O}^{[\gamma]},$ by the
base case there exists a $w \in N_{t}=N$ and $y \in \fnminus$ so
that $yw + N_{t-1}=v + N_{t-1}.$  Thus, $v-yw \in N_{t-1}$ and is of weight
$\lambda.$ By the inductive assumption, there exists $w' \in N_{t-1} \subset N$
and $y' \in \fnminus$ such that $y'w' = v-yw.$  That is, $v=yw+y'w' \in
\fnminus N.$  This proves the desired result.

Now, (ii) follows by a similar argument by induction on the length
of a composition series.   If $N$ is simple and $N_{\lambda}\neq 0$, then
$\lambda$ is not the highest weight of $N$ (as $\gamma$ is the unique dominant
highest weight among the simple modules in $\mathcal{O}^{[\gamma]}$ by Proposition~\ref{P:penkov}).  From
this it immediately follows that  $N_{\lambda}=(\fnminus N)_{\lambda}$.   Now
proceed by induction as in the previous paragraph.  \end{proof}

\begin{lem}\label{L:TypicalFlambda} Let $\lambda \in P^{++}$ be typical
and dominant, and let $M \in \mathcal{O}.$  Then
\[
F_{\lambda}\left(M \right) \cong  \left( (M\otimes V^{\otimes
d})^{[\lambda]}\right)_\ld \cong \left[ M\otimes V^{\otimes d}/\n_-(M\otimes
V^{\otimes d})\right]_\ld
\] as $\ASe (d)$-modules.
\end{lem}

\begin{proof}  It should first be remarked that since the action of $\ASe (d)$ commutes with the action of $\q (n)$, the action of $\ASe (d)$ on $M \otimes V^{\otimes d}$ induces an action on each of the vector spaces given in the theorem.

Now, by Proposition~\ref{P:penkov} and the assumption that $\lambda$ is
dominant, it follows that for any module $N \in \mathcal{O}^{[\lambda]},$  $N_{\nu} \neq 0$ only if $\nu \leq \lambda$ in the dominance order.  Thus
any vector of weight $\lambda$ in $M \otimes V^{\otimes d}$ is necessarily a
primitive vector. On the other hand, if there is a primitive vector of weight $\lambda$ in $M \otimes V^{\otimes d},$ then it must lie in the image of a nonzero homomorphism $M(\lambda) \to M \otimes V^{\otimes d}$.  But as $M(\lambda)$ is an object in $\mathcal{O}^{[\lambda]}$, it follows that the primitive vector lies in $\left( (M\otimes V^{\otimes
d})^{[\lambda]}\right)_\ld $.  Thus, there exists a canonical projection map
\[
F_{\lambda}\left(M \right) \to  \left( (M\otimes V^{\otimes
d})^{[\lambda]}\right)_\ld
\] and this map is necessarily a vector space isomorphism.  The fact that it is a $\ASe (d)$-module homomorphism follows from the fact that the action of $\ASe (d)$ on both vector spaces is induced by the action of $\ASe (d)$ on $M \otimes V^{\otimes d}.$

Now consider the block decomposition
\begin{equation*}
M\otimes V^{\otimes d}= \oplus_{\chi_{\gamma}} (M\otimes V^{\otimes
d})^{[\chi_{\gamma}]},
\end{equation*}
where the direct sum runs over dominant $\gamma \in \fh_{\zero}^{*}$ so that
different $\chi_{\gamma}$ are different central characters of $U(\fg )$.
This then induces the vector space direct sum decomposition
\[
(M\otimes V^{\otimes d})/\fnminus (M\otimes V^{\otimes d}) =
\oplus_{\chi_{\gamma}} (M\otimes V^{\otimes d})^{[\chi_{\gamma}]}/\fnminus
(M\otimes V^{\otimes d})^{[\chi_{\gamma}]},
\]
where $(M\otimes V^{\otimes d})^{[\chi_{\gamma}]}$ denotes the direct summand
of $M\otimes V^{\otimes d}$ which lies in the block $\mathcal{O}^{[\gamma]}.$

By the previous lemma, if $\gamma$ is atypical or if $\gamma$ is typical
and $\gamma \neq \lambda$, then
\[
\left[ (M\otimes V^{\otimes d})^{[\chi_{\gamma}]})/ \fnminus (M\otimes
V^{\otimes d})^{[\chi_{\gamma}]})\right]_{\lambda}=0.
\]
Therefore,
\begin{equation}\label{E:functorvariations}
\left[ (M\otimes V^{\otimes d})/\fnminus (M\otimes V^{\otimes d})\right]_{\lambda} =  \left[(M\otimes
V^{\otimes d})^{[\chi_{\lambda}]}/\fnminus (M\otimes V^{\otimes
d})^{[\chi_{\lambda}]} \right]_{\lambda}.
\end{equation}
Finally, if $N$ is an object of $\mathcal{O}^{[\lambda]},$ then $N_{\mu} \neq 0$
only if $\mu \leq \lambda$ in the dominance order.  Thus weight considerations imply
$\left[ \fnminus (M\otimes V^{\otimes d})^{[\chi_{\lambda}]})\right]_{\lambda}=0$
which, in turn, implies that canonical projection
\[
 \left( (M\otimes V^{\otimes
d})^{[\lambda]}\right)_\ld \to \left[ M\otimes V^{\otimes d}/\n_-(M\otimes
V^{\otimes d})\right]_\ld
\] is a vector space isomorphism.  That is its a $\ASe (d)$-module homomorphism follows from the fact that in both cases the action is induced from the $\ASe (d)$ action on $M \otimes V^{\otimes d}$.
\end{proof}

\begin{cor}\label{C:Flambdaexactness}  If $\lambda \in \Pt$ is dominant and
typical, then the functor $F_{\lambda}:\mathcal{O} \to \ASe (d)$-mod is exact.
\end{cor}

\begin{proof}  This follows immediately from the first alternative description of $F_{\lambda}$ in the above theorem as it
is the composition of the exact functors $- \otimes V^{\otimes d}$, projection onto the direct summand lying in the block $\mathcal{O}^{[\lambda]}$,  and projection onto the $\lambda$ weight space.
\end{proof}

In what follows when $\lambda$ is dominant and typical we use whichever description of $F_{\lambda}$ given in lemma ~\ref{L:TypicalFlambda} is most convenient.

\subsection{Image of the Functor}\label{SS:functorimage} We can now describe the
image of Verma modules under the functor.

\begin{lem}\label{L:description}  Let $M(\mu)$ be a Verma module in $\mathcal{O}$
and let $\lambda \in P^{++}$ be a dominant and typical weight.  The natural
inclusion
\[
E(\mu)\otimes(V^{\otimes d})_{\ld-\mu}\hookrightarrow(M(\mu)\otimes V^{\otimes
d})_\ld
\]
induces an isomorphism of $\Se(d)$-modules $E(\mu)\otimes(V^{\otimes
d})_{\ld-\mu}\cong F_\ld(M(\mu))$. In particular, $F_\ld(M(\mu))=0$ unless
$\ld-\mu\in \Ppos(d)$.
\end{lem}

\begin{proof}  This is proved exactly as in \cite[Lemma 3.3.2]{as},
except now the highest weight space of $M(\mu)$ is $E(\mu)$.  Namely, by the
tensor identity and the PBW theorem,
\begin{equation}\label{E:tensoridentity}
M(\mu) \otimes V^{\otimes d} \cong U(\fg ) \otimes_{U(\fb)} \left(E(\mu)
\otimes V^{\otimes d} \right) \cong U(\fnminus) \otimes E(\mu) \otimes
V^{\otimes d},
\end{equation}
where the first isomorphism is as $\fg$-modules and the second is as
$\fh_{\bar{0}}$-modules. Thus the canonical projection map induces the
isomorphism of $\fh_{\bar{0}}$-modules given by
\begin{equation*}
1 \otimes E(\mu) \otimes V^{\otimes d} \cong M(\mu) \otimes V^{\otimes d} /
\fnminus \left(M(\mu) \otimes V^{\otimes d} \right).
\end{equation*}
Taking $\lambda$ weight spaces on both sides yields the vector space isomorphism
\[
1 \otimes E(\mu) \otimes \left( V^{\otimes d}\right)_{\lambda-\mu}\cong \left[
M(\mu) \otimes V^{\otimes d} / \fnminus \left(M(\mu) \otimes V^{\otimes d}
\right)\right]_{\lambda}.
\]  Now, the composition of the natural inclusion $E(\mu)\otimes(V^{\otimes d})_{\ld-\mu}\hookrightarrow(M(\mu)\otimes V^{\otimes
d})_\ld$ with \eqref{E:tensoridentity}, and the isomorphism above implies that
\[
 E(\mu) \otimes \left( V^{\otimes d}\right)_{\lambda-\mu} \cong 1 \otimes E(\mu) \otimes \left( V^{\otimes d}\right)_{\lambda-\mu} \cong \left[ M(\mu) \otimes V^{\otimes d} /
\fnminus \left(M(\mu) \otimes V^{\otimes d} \right)\right]_{\lambda} = F_{\lambda}\left(M(\mu) \right),
\]

That it is an isomorphism of $\Se(d)$-modules follows from the fact that in
each case the action of $\Se(d)$ is via the action induced from the action of
$\Se(d)$ on $M(\mu) \otimes V^{\otimes d}.$
\end{proof}

\begin{cor}\label{C:StandardDim} Let  $\lambda \in P^{++}$ be a dominant
and typical weight and let $\mu \in P$ with $\ld-\mu\in \Ppos(d)$.  Set $d_{i}=\lambda_{i}-\mu_{i}$ for $i=1, \dotsc ,n$.
\begin{enumerate}
\item [(i)]  Let $M(\mu)$ be the little Verma module of highest weight $\mu$.  Then,
\[
\dim F_{\lambda}(M(\mu)) = 2^{d+\lfloor(n-\gm_0(\mu)+1)/2 \rfloor}\frac{d!}{d_{1}!
\dotsb d_{n}!}.
\]
\item [(ii)] Let $\widehat{M}(\mu)$ be the big Verma module of highest weight $\mu$. Then,
\[
\dim F_\ld(\widehat{M}(\mu))=2^{d+n-\gm_0(\mu)}\frac{d!}{d_1!\cdots d_n!}.
\]
\end{enumerate}
\end{cor}

\begin{proof}  We have $\dim E(\mu) = 2^{\lfloor(n-\gm_0(\mu)+1)/2 \rfloor}$.
For each $\varepsilon_{i}$ $(i=1, \dotsc , n)$, $\dim
V_{\varepsilon_{i}}=2.$  A combinatorial count shows that
\[
\dim \left(V^{\otimes d} \right)_{\lambda - \mu} = \frac{d!}{d_{1}! \dotsb
d_{n}!}2^{d}.
\]
The statement of (i) then follows by Lemma~\ref{L:description}. The statement
of (ii) follows from (i) and Lemma~\ref{L:little verma in big verma}.
\end{proof}

Fix $\ld,\mu\in P$ such that $\ld-\mu\in \Ppos(d)$, and let $d_i=\ld_i-\mu_i$.
Let $\{u_i,u_\oi\}_{i=1,\ldots,n}$ be the standard basis for $V$, let $v_\mu\in
E(\mu)$, and let $u_{\ld-\mu}=u_1^{\otimes d_1}\otimes\cdots\otimes
u_n^{\otimes d_n}\in (V^{\otimes d})_{\ld-\mu}$. Finally, let
\[
m_k=\sum_{i=1}^kd_k,
\]
and define $F_k=\pi_0(f_{kk})$ (see Section~\ref{SS:action}).

\begin{lem}\label{X action} Let $v_{\mu} \in M(\mu)_{\mu}$ be a primitive vector
of weight $\mu,$ and let
$u=u_{\lambda - \mu}= u_{1}^{\otimes d_{1}} \otimes \dotsb  u_{n}^{\otimes d_{n}}.$
For each $1\leq k\leq n$ and
$m_{k-1}<i\leq m_k$,
\[
X_i.v_\mu\otimes u_{\ld-\mu}\equiv
\left(\mu_k+i-m_{k-1}-1-\sum_{m_{k-1}<l<i}C_lC_i -F_{k}C_i\right)v_\mu\otimes
u_{\ld-\mu}
\]
modulo $\n_-(M(\mu)\otimes V^{\otimes d})$. As a consequence,
\[
X_i^2v_\mu\otimes
u_{\ld-\mu}\equiv(\mu_k+i-m_{k-1}-1)(\mu_k+i-m_{k-1})v_\mu\otimes u_{\ld-\mu},
\] again modulo $\n_-(M(\mu)\otimes V^{\otimes d})$.
\end{lem}

\begin{proof}  We first do some preliminary calculations.  Let $1 \leq j < k \leq n$
be fixed, let $m_{k-1} \leq i \leq m_{k}$ be fixed, and consider the vector
\[
v_{\mu}\otimes u_{1}^{\otimes d_{1}} \otimes \dotsb \otimes u_{k}^{\otimes a}
\otimes u_{j} \otimes u_{k}^{\otimes b} \otimes \dotsb \otimes u_{n}^{\otimes
d_{n}},
\]
where the $u_{j}$ is the $i$th tensor and $a+b+1=d_{k}$ (i.e.\ among the
$u_{k}$'s, the one in the $i$th position, recalling that $v_{\mu}$ is in the
zeroth position, is replaced with $u_{j}$). For short, let us write $u =
u_{1}^{\otimes d_{1}} \otimes \dotsb  u_{n}^{\otimes d_{n}}$ and
$\hat{u}=u_{1}^{\otimes d_{1}} \otimes \dotsb \otimes u_{k}^{\otimes a} \otimes
u_{j} \otimes u_{k}^{\otimes b} \otimes \dotsb \otimes u_{n}^{\otimes d_{n}}$.
Then,
\begin{eqnarray*}
e_{kj}(v_{\mu}\otimes\hat{u})&=&(e_{kj}v_{\mu})\otimes\hat{u}\\
&&+ \sum_{r = 1}^{d_{j}} v_{\mu}\otimes u_{1}^{\otimes d_{1}}
\otimes \dotsb \otimes u_{j}^{\otimes r-1 } \otimes u_{k} \otimes u_{j}^{\otimes d_{j}-r}
\otimes \dotsb \otimes u_{k}^{\otimes a} \otimes u_{j} \otimes u_{k}^{\otimes
b} \otimes \dotsb \otimes
u_{n}^{\otimes d_{n}} + v_{\mu} \otimes u\\
 &=&(e_{kj}v_{\mu})\otimes \hat{u} + \sum_{r = 1}^{d_{j}} S_{m_{j-1}+r, i}
 (v_{\mu}\otimes u) + v_{\mu} \otimes u.
\end{eqnarray*}
Similarly, if we write $\check{u}=C_{i}\hat{u}=u_{1}^{\otimes d_{1}} \otimes
\dotsb \otimes u_{k}^{\otimes a} \otimes v_{-j} \otimes v_{k}^{\otimes b}
\otimes \dotsb \otimes v_{n}^{\otimes d_{n}}$, then
\begin{eqnarray*}
f_{kj} (v_{\mu}\otimes \check{u})
&=& (f_{kj}v_{\mu})\otimes \check{u}
    +(-1)^{p(v_{\mu})}\times\\
    &&\times \sum_{r=1}^{d_{j}} v_{\mu}
    \otimes u_{1}^{\otimes d_{1}}
    \otimes \dotsb \otimes u_{j}^{\otimes r-1} \otimes u_{-k} \otimes u_{j}^{\otimes d_{j}-r}
    \otimes \dotsb \otimes  u_{k}^{\otimes a} \otimes u_{-j} \otimes u_{k}^{\otimes b}
    \otimes \dotsb \otimes u_{n}^{\otimes d_{n}}\\
    &&+ (-1)^{p(v_{\mu})}v_{\mu} \otimes u\\
 &=&(f_{kj}v_{\mu})\otimes \check{u} +(-1)^{p(v_{\mu})}
    \sum_{r = 1}^{d_{j}} C_{m_{j-1}+a}C_{i}S_{m_{j-1}+r, i}
    (v_{\mu}\otimes u) + (-1)^{p(v_{\mu})} v_{\mu} \otimes u.
\end{eqnarray*}

We can now consider the first statement of the lemma.  Throughout, we write
$\equiv$ for congruence modulo the subspace  $\n_-(M(\mu)\otimes V^{\otimes
d})$.  Let $1 \leq k \leq n$ be fixed so that $m_{k-1} < i \leq m_{k}$ (ie.\
there is a $u_{k}$ in the $i$th position of $v_{\mu}\otimes u$).  Using that
$v_{\mu}$ is a primitive vector and the equalities given above, we deduce that
\begin{small}
\begin{eqnarray*}
X_i\left( v_\mu\otimes u_{\ld-\mu}\right)&=&\sum_{\ell,j=1}^n e_{\ell
j}v_\mu\otimes u_1^{\otimes d_1}\otimes\cdots\otimes u_k^{\otimes
i-m_{k-1}-1}\otimes \e_{j\ell}u_k\otimes u_k^{\otimes
m_k-i}\otimes\cdots\otimes u_n^{\otimes d_n}\\
    &&-(-1)^{p(v_{\mu})}\sum_{\ell,j=1}^nf_{\ell j}v_\mu\otimes
u_1^{\otimes d_1}\otimes\cdots\otimes u_k^{\otimes i-m_{k-1}-1}\otimes
\f_{j\ell}u_k\otimes u_k^{\otimes
m_k-i}\otimes\cdots\otimes u_n^{\otimes d_n}\\
    &&+\sum_{\ell<i}(1-C_\ell C_i)S_{\ell i}(v_\mu\otimes
    u)\\
    &=&\sum_{j\leq k} e_{kj}v_\mu\otimes
u_1^{\otimes d_1}\otimes\cdots\otimes u_k^{\otimes i-m_{k-1}-1}\otimes
u_j\otimes u_k^{\otimes
m_k-i}\otimes\cdots\otimes u_n^{\otimes d_n}\\
    &&-(-1)^{p(v_\mu)}\sum_{j\leq k}f_{kj}v_\mu\otimes
u_1^{\otimes d_1}\otimes\cdots\otimes u_k^{\otimes i-m_{k-1}-1}\otimes
u_{-j}\otimes u_k^{\otimes
m_k-i}\otimes\cdots\otimes u_n^{\otimes d_n}\\
    &&+\sum_{\ell<i}(1-C_\ell C_i)S_{\ell i}(v_\mu\otimes u)\\
    & \equiv& - \sum_{j < k} \left[\sum_{a = 1}^{d_{j}} S_{m_{j-1}+a, i}
    (v_{\mu}\otimes u) + v_{\mu} \otimes u \right] \\
    &&+ \sum_{j < k} \left[ \sum_{a = 1}^{d_{j}}
    C_{m_{j-1}+a}C_{i}S_{m_{j-1}+a, i} (v_{\mu}\otimes u) + v_{\mu} \otimes u \right] \\
    &&+ \mu_{k}v_{\mu}\otimes u - C_{i}
    \left((f_{kk}v_{\mu})\otimes u \right) +\sum_{\ell<i}(1-C_\ell C_i)S_{\ell i}
    (v_\mu\otimes u)\\
    &=& - \sum_{l \leq m_{k-1}} S_{l,i}v_{\mu}\otimes u  -(k-1)v_{\mu} \otimes u
    + \sum_{l \leq m_{k-1}} C_{l}C_{i}S_{l,i}v_{\mu}\otimes u +(k-1)v_{\mu}\otimes u \\
    &&+ \mu_{k}v_{\mu}\otimes u
    - C_{i}\left((f_{kk}v_{\mu})\otimes u \right) +\sum_{\ell<i}(1-C_\ell C_i)S_{\ell i}(v_\mu\otimes u)\\
    & = &\mu_kv_\mu\otimes
    u_{\ld-\mu}+C_i((f_{kk}v_\mu)\otimes u_{\ld-\mu})+\sum_{m_{k-1}<\ell<i}
    (1-C_\ell C_i)S_{l,i}(v_\mu\otimes u)\\
    &=&\left(\mu_k+i-m_{k-1}-1-\sum_{m_{k-1}<\ell<i}C_\ell C_iS_{l,i}\right)
    (v_\mu\otimes u_{\ld-\mu})+C_i((f_{kk}v_\mu)\otimes u)\\
     &=&\left(\mu_k+i-m_{k-1}-1-\sum_{m_{k-1}<\ell<i}C_\ell C_i -F_{k}C_i\right)
     (v_\mu \otimes u).
\end{eqnarray*}
\end{small}
Note the last equality makes use of the fact that $S_{l,i}
v_{\mu}\otimes u = v_{\mu}\otimes u$ for $m_{k-1}<l<i$ and that as
(odd) linear maps $F_{k}C_{i}=-C_{i}F_{k}.$

Now we consider the second statement of the lemma.  Using the previous
calculation, the fact that $X_{i}$ and the $C$'s satisfy relation \eqref{c&x}
of the degenerate AHCA, and the fact that $f_{kk}v_{\mu} \in
M(\mu)_{\mu}$ is again a primitive vector of weight $\mu$,

\begin{eqnarray*}
X_i^2(v_\mu\otimes u_{\ld-\mu})
    &\equiv& X_i\left(\mu_k+i-m_{k-1}-1-\sum_{m_{k-1}<\ell<i}C_\ell C_i-F_kC_i\right)
    (v_\mu\otimes u_{\ld-\mu})\\
     &=&\left(\mu_k+i-m_{k-1}-1+\sum_{m_{k-1}<\ell<i}C_\ell C_i \right)X_{i}
     (v_{\mu}\otimes u)- C_iX_{i}((f_{kk}v_\mu)\otimes u_{\ld-\mu})\\
     & \equiv & \left(\mu_k+i-m_{k-1}-1+\sum_{m_{k-1}<\ell<i}C_\ell C_i \right)\times\\
     &&\times
     \left(\mu_k+i-m_{k-1}-1-\sum_{m_{k-1}<\ell<i}C_\ell C_i-F_kC_i\right)
     (v_\mu\otimes u_{\ld-\mu}) \\
     &&-C_{i}\left(\mu_k+i-m_{k-1}-1-\sum_{m_{k-1}<\ell<i}C_\ell C_i-F_kC_i\right)
     ((f_{kk}v_\mu)\otimes u_{\ld-\mu}) \\
     &=&\left(\mu_k+i-m_{k-1}-1+\sum_{m_{k-1}<\ell<i}C_\ell C_i \right)
     \left(\mu_k+i-m_{k-1}-1-\sum_{m_{k-1}<\ell<i}C_\ell C_i \right) v_{\lambda}
     \otimes u\\
     &&+ C_{i}F_{k}C_{i}((f_{kk}v_{\mu}) \otimes u) \\
     &=&\left( (\mu_k+i-m_{k-1}-1)^{2}
     - \left( \sum_{m_{k-1}<\ell<i}C_\ell C_i\right)^{2}\right) v_{\mu}\otimes u
     + (f^{2}_{kk}v_{\mu}) \otimes u \\
     &=&\left( (\mu_k+i-m_{k-1}-1)^{2}
     +(\mu_k+i-m_{k-1}-1)\right) v_{\mu}\otimes u.
\end{eqnarray*}
The last equality follows from the fact that in the Clifford algebra
\[
\left( \sum_{m_{k-1}<\ell<i}C_\ell C_i\right)^{2} =
\sum_{m_{k-1}<\ell<i}(C_\ell C_i)^{2} = \sum_{m_{k-1}<\ell<i} -1 = -(i-m_{k-1}
+1)
\]
and that, in $\q (n)$, $f_{kk}^2=e_{kk}$.
\end{proof}

\begin{cor}\label{C:ImageisIntegral} Let $\lambda \in P^{++}$ be a dominant
typical weight, let $\mu\in P$, and let $M(\mu)$ be a Verma module in
$\mathcal{O}(\fq (n))$. Then for $i=1, \dotsc , d$ the element $x_{i}^{2}$ acts on
$F_{\lambda}(M(\mu))$ with generalized eigenvalues of the form $q(a)$ for
various $a \in \Z$. Hence, $F_{\lambda}(M(\mu))$ is integral.
\end{cor}

As a consequence of the previous corollary we see that for $\lambda \in P^{++}$ we have that $F_{\lambda}\left(L(\mu) \right)$ is integral for any simple module $L(\mu)$ in $\mathcal{O}$ and, therefore,
\[
F_\ld:\mathcal{O}(\q(n))\rightarrow\Rep\ASe(d).
\]

\begin{prp}  Let $\ld\in\Pt$ and $\mu\in\ld-\Ppos(d)$. Then,
$F_\ld(\widehat{M}(\mu))\cong\widehat{\M}(\ld,\mu)$.
\end{prp}

\begin{proof}
Let $v_+\in\C_\mu$ be a nonzero vector in the 1-dimensional $\h_\zero$-module $\C_\mu$, let $v_\mu=1\otimes v_+\in C(\mu)_\zero$ be its image and let $u_{\lambda-\mu}$ be as in the prevous lemma. Then
$v_\mu\otimes u_{\ld-\mu}$ is a cyclic vector for $F_\ld(\widehat{M}(\mu))$ as
a $\ASe (d)$-module.

Recall the cyclic vector $\hat{\va}_{\ld,\mu}\in\widehat{\M}(\ld,\mu)$. For $\dt_1,\ldots,\dt_n\in\{0,1\}$, let $\ph_1^{\dt_1}\cdots\ph_n^{\dt_n}\hat{\va}_{\lambda,\mu}=1\otimes\ph_1^{\dt_1}\hat{\va}\otimes\cdots\otimes\ph_n^{\dt_n}\hat{\va}$, cf. \eqref{E:hatcyclicvector}.

Note that $w.(v_\mu\otimes u_{\ld-\mu})=v_\mu\otimes
u_{\ld-\mu}$ for all $w\in S_{\ld-\mu}$. Comparing Lemma \ref{X action} and Proposition \ref{segment representation}, we deduce that, by Frobenious reciprocity, there exists
a surjective $\ASe(d)$-homomorphism $\widehat{\M}(\ld,\mu)\rightarrow
F_\ld(\widehat{M}(\mu))$ sending $\ph_1^{\dt_1}\cdots\ph_n^{\dt_n}\hat{\va}_{\lambda,\mu}\mapsto F_1^{\dt_1}\cdots F_n^{\dt_n}v_\mu\otimes
u_{\ld-\mu}$. That this is an isomorphism follows by comparing dimensions using
Lemmas~\ref{L:standard cyclic dim} and~\ref{C:StandardDim}.
\end{proof}

\begin{cor}\label{C:Image of the little verma} We have
\[
F_\ld M(\mu)\cong\M(\ld,\mu)^{\oplus 2^{\varpi(\mu)}}
\]
where
\[
\varpi(\mu) =\begin{cases}\lfloor\frac{n+1}{2}\rfloor&\mbox{if
}\gm_0(\mu)\mbox{ is even,}\\\lfloor\frac{n}{2}\rfloor&\mbox{if
}\gm_0(\mu)\mbox{ is odd.}\end{cases}.
\]
\end{cor}

\begin{proof} Using the additivity of the functor $F_{\lambda}$, the previous proposition, and Lemmas~\ref{L:little verma in big verma} and~\ref{L:standard cyclic dim} we obtain $F_\ld M(\mu)=2^{n-\lfloor\frac{\gm_0(\mu)+1}{2}\rfloor
-\lfloor\frac{n-\gm_0(\mu)}{2}\rfloor}\M(\ld,\mu)$. It is just left to observe that
\[
n-\lfloor\frac{\gm_0(\mu)+1}{2}\rfloor
-\lfloor\frac{n-\gm_0(\mu)+1}{2}\rfloor=\varpi(\mu).
\]
\end{proof}

\begin{lem}\label{L:IsoStdMod} Assume that $\ld\in\Pt$, $\mu\in P^+[\ld]$,
$\ld-\mu\in\Ppos(d)$, and $\af\in R^+[\ld]$. Then,
$\M(\ld,\mu)\cong\M(\ld,s_\af\mu)$.
\end{lem}

\begin{proof} By Lemma \ref{L:InjHom}, there exists an injective homomorphism
$M(s_\af\mu)\rightarrow M(\mu)$. Since $\varpi(\mu)=\varpi(s_\af\mu)$, there
exists an injective homomorphism
\[
\M(\ld,s_\af\mu)^{\varpi(\mu)}=F_\ld M(s_\af\mu)\to F_\ld M(\mu)=
\M(\ld,\mu)^{\varpi(\mu)}.
\]
Since $\dim\M(\ld,s_\af\mu)=\dim\M(\ld,\mu)$ and by Theorem~\ref{thm:unique irred quotient} $\M(\ld,\mu)$ is indecomposible, it follows that this map is an isomorphism.
\end{proof}

\begin{thm}\label{T:MaxSubmod} Assume $\ld\in\Pt$ and $\mu\in\ld-\Ppos(d)$.
Then, $\M(\ld,\mu)$ has a unique maximal submodule $\mathcal{R}(\ld,\mu)$ and unique
irreducible quotient $\L(\ld,\mu)$.
\end{thm}

\begin{proof}
There exists $w\in S_d[\ld]$ such that $w\mu\in P^+[\ld]$. By Lemma
\ref{L:IsoStdMod}, $\M(\ld,w\mu)\cong\M(\ld,\mu)$. By Theorem \ref{thm:unique
irred quotient}, $\M(\ld,w\mu)$ has a unique maximal submodule and unique
irreducible quotient, so the result follows.
\end{proof}

Given $\mu\in P$, the Shapovalov form on $M(\mu)$ induces a non-degenerate
$\q(n)$-contravariant form on $L(\mu)$, which we will denote
$(\cdot,\cdot)_\mu$. In turn we have a non-degenerate $\q(n)$-contravariant
form on $L(\mu)\otimes V^{\otimes d}$ given by
$(\cdot,\cdot)_\mu\otimes(\cdot,\cdot)_{\ep_1}^{\otimes d}$. Observe that
different weight spaces are orthogonal with respect to this form and different blocks of $\mathcal{O}(\q (n))$ given by central
characters are also orthogonal.  Therefore, when $\lambda\in\Pt$ is dominant and
typical it follows that the bilinear form restricts to a form on
$(L(\mu)\otimes V^{\otimes d})^{[\ld]}_\ld=F_\ld(L(\mu))$, which is
non-degenerate whenever it is nonzero. By Proposition \ref{P:when tau's
collide}, this form is $\ASe(d)$-contravariant.

Similarly, Proposition \ref{P:when tau's collide} implies that the Shapovalov
form on $\widehat{M}(\mu)$ induces an $\ASe(d)$-contravariant form on
$\widehat{\M}(\ld,\mu)$. Now, if $\ld\in\Pt$ and $\mu\in\ld-\Ppos(d)$, then by
Theorem \ref{T:MaxSubmod}, $\widehat{\M}(\ld,\mu)$ posesses a unique
submodule $\widehat{\mathcal{R}}(\ld,\mu)$ which is maximal among those which avoid the generalized
$\zt_{\ld,\mu}$ weight space. Indeed,
\[
\widehat{\mathcal{R}}(\ld,\mu)=\mathcal{R}(\ld,\mu)^{\oplus
2^{n-\lfloor\frac{\gm_0(\mu)+1}{2}\rfloor}}.
\]

\begin{prp}\label{P:ASeRadical} Assume that $\ld\in\Pt$, $\mu\in\ld-\Ppos(d)$,
and $\widehat{\M}(\ld,\mu)$ possesses a nonzero contravariant form
$(\cdot,\cdot)$. Let $\mathcal{R}$ denote the radical of this form. Then,
\[
\mathcal{R}\supseteq\widehat{\mathcal{R}}(\ld,\mu).
\]
\end{prp}

\begin{proof}
First, recall that $\widehat{\M}(\ld,\mu)$ is cyclically generated by
$\hat{\va}_{\ld,\mu}\in\widehat{\M}(\ld,\mu)_{\zt_{\ld,\mu}}$. Now, assume
$v\in\widehat{\mathcal{R}}(\ld,\mu)$ and $v'\in\widehat{\M}(\ld,\mu)$. Then,
$v'=X.\hat{\va}_{\ld,\mu}$ for some $X\in\ASe(d)$. Moreover,
$\tau(X).v\in\widehat{\mathcal{R}}(\ld,\mu)$. Applying Lemma
\ref{L:ASeContraForm} and the definition of $\widehat{\mathcal{R}}(\ld,\mu)$ we
deduce that
\[
(v',v)=(X.\hat{\va}_{\ld,\mu},v)=(\hat{\va}_{\ld,\mu},\tau(X).v)=0.
\]
Hence, $v\in\mathcal{R}$.
\end{proof}

\begin{cor}\label{C:ASeRadical} Given $\ld\in\Pt$ and $\mu\in\ld-\Ppos(d)$,
\[
\mathcal{R}=\mathcal{R}(\ld,\mu)^{\oplus k}\oplus\M(\ld,\mu)^{\oplus
2^{n-\lfloor\frac{\gm_0(\mu)+1}{2}\rfloor}-k}
\]
for some $0\leq k\leq2^{n-\lfloor\frac{\gm_0(\mu)+1}{2}\rfloor}$.
\end{cor}

\begin{thm}\label{T:SimplesToSimples}
Assume $\ld\in\Pt$, and $\mu\in\ld-\Ppos(d)$. If $F_\ld L(\mu)$ is
nonzero, then
\[
F_\ld L(\mu)\cong\L(\ld,\mu)^{\oplus\ell}
\]
for some $0<\ell\leq\varpi(\mu)$.
\end{thm}

\begin{proof}
Let $\widehat{L}(\mu)=L(\mu)^{\oplus
2^{\lfloor\frac{n-\gm_0(\mu)+1}{2}\rfloor}}$, so that
$\widehat{L}(\mu)=\widehat{M}(\mu)/\widehat{R}(\mu)$ where $\widehat{R}(\mu)$
is the radical of the Shapovalov form on $\widehat{M}(\mu)$. Applying the
functor, we see that
\[
F_\ld \widehat{L}(\mu)=\widehat{\M}(\ld,\mu)/ F_\ld\widehat{R}(\mu).
\]
Now, $F_\ld\widehat{R}(\mu)=\mathcal{R}$. Hence, Corollary \ref{C:ASeRadical}
and a calculation similar to Corollary \ref{C:Image of the little verma} gives
the result.
\end{proof}

\begin{prp}\cite[Proposition 18.18.1]{kl} Any finite dimensional irreducible
$\ASe(d)$-module is a composition factor of $\M(\ld,\ld-\ep)$ for some $\ld\in\Pt$.
\end{prp}

\begin{thm} Any finite dimensional simple module for $\ASe(d)$ is isomorphic
$\L(\ld,\mu)$ for some $\mu\in(\ld-\ep)-Q^+$.
\end{thm}

\begin{proof} The functor $F_\ld$ transforms the compostition series for
$M(\ld-\ep)$ into the compostition series for $\M(\ld,\ld-\ep)$. It is now just
left to observe that if $L(\mu)$ is a composition factor for $M(\ld-\ep)$, then
$\mu\in(\ld-\ep)- Q^+$.
\end{proof}

\subsection{Calibrated Representations Revisited}

\begin{thm} If $\ld,\mu\in\Pp$ satisfy $\ld-\mu\in\Ppos(d)$,
then $F_\ld(L(\mu)) \neq 0$ and hence one has a simple module $\L (\lambda, \mu)$.
\end{thm}

\begin{proof} The formal character of $L(\mu)$ when $\mu\in\Pp$ is given by the $Q$-Schur function $Q_\mu$ (c.f.\ \cite{s}). There is a nondegenerate bilinear form, $(\cdot,\cdot)_{\Pp}$ on the subring of symmetric functions spanned by Schur's $Q$-functions given by
\[
\left(Q_{\lambda}, Q_{\mu} \right)_{\Pp} = \Hom_{\fq (n)}\left(L(\lambda), L(\mu) \right).
\]
Furthermore, the basis $Q_\mu$ ($\mu\in\Pp$) is an orthogonal basis. Within this subring are the skew $Q$-Schur functions $Q_{\ld/\mu}$.  We refer the reader to \cite{stem, m} for details.

Under the hypotheses of the theorem, $\ld/\mu$ is a skew shape. Moreover, $F_\ld L(\mu)=0$ implies that
\begin{eqnarray}\label{E:PolynomialRep}
0=\Hom_{\fq  (n)} \left(L(\ld),L(\mu)\otimes V^{\otimes d} \right)=\bigoplus_{\nu\in\Pp(d)}\Hom(L(\ld),L(\mu)\otimes L(\nu))^{\oplus N_\nu}.
\end{eqnarray}
The second equality follows from Sergeev duality which implies that as a $\q(n)$-module
\[
V^{\otimes d}=\bigoplus_{\nu\in\Pp(d)}L(\nu)^{\oplus N_\nu},
\] where $N_\nu$ is the dimension of the Specht module of $\Se(d)$ corresponding to $\nu$ \cite{s2}.

In terms of the bilinear form on symmetric functions, \eqref{E:PolynomialRep} implies
\begin{equation}\label{E:perp}
0 = \left(Q_{\lambda}, Q_{\mu}Q_{\nu} \right)
\end{equation}
for all $\nu \in \Pp (d)$. In fact \eqref{E:perp} holds for all $\nu \in \Pp$ since different graded summands of the symmetric function ring are orthogonal. However,
\[
 \left(Q_{\lambda}, Q_{\mu}Q_{\nu} \right) = \left(Q_{\mu}^{\bot}Q_{\lambda}, Q_{\nu} \right) = 2^{\ell(\mu)}\left( Q_{\lambda/\mu}, Q_{\nu}\right),
\] where $Q_{\mu}^{\bot}$ denotes the adjoint of $Q_{\mu}$ with respect to the form and the second equality follows from $Q_{\mu}^{\bot}Q_{\lambda}= 2^{-\ell(\mu)} Q_{\lambda/\mu}$ (cf.\ \cite[II.8]{m}).
Thus, \eqref{E:PolynomialRep} implies that
\[
(Q_{\ld/\mu},Q_\nu)=0
\]
for all $\nu\in\Pp$.  But the $Q$-functions form an orthogonal basis for this subring. This implies $Q_{\ld/\mu}=0$, which is not true.  Hence, $F_{\lambda}L(\mu) \neq 0$.
\end{proof}

Arguing as in section 7 of \cite{su2} using Sergeev duality \cite{s,s2} we obtain the following result.

\begin{cor}  Let $\ld,\mu\in\Pp$ such that $\ld-\mu\in P_{\geq 0}(d)$. Then the group character of $\L(\ld,\mu)\downarrow_{\Se(d)}$ is a power of $2$ multiple of the skew $Q$-Schur function $Q_{\ld/\mu}$.
\end{cor}

\section{A Classification of Simple Modules}\label{S:Classification}

In \cite{bk2,kl}, it was shown that the Grothendieck group of finite dimensional
integral representations of $\ASe(d)$ is a module for the Kostant-Tits
$\Z$-form of the Kac-Moody Lie algebra $\b_\infty$. Indeed, let
$\n_\infty$ be a maximal nilpotent subalgebra of $\b_\infty$, and let $\U_\Z^*(\n_\infty)$ be the
\emph{minimal} admissible lattice inside the universal envelope of
$\n_\infty$. This lattice is spanned by Lusztig's dual canonical basis,

\begin{thm}\cite[Theorem 20.5.2]{kl} There is an isomorphism of graded Hopf algebras
\[
\U_\Z^*(\n^+_\infty)\cong\bigoplus_{d\geq 0}K(\Rep\ASe(d)).
\]
\end{thm}

and,

\begin{thm}\cite[Theorem 21.0.4]{kl} The set $B(\infty)$ of isomorphism
classes of simple $\ASe(d)$-modules, for all $d$, can be given the structure of
a crystal (in the sense of Kashiwara). Moreover, this crystal is isomorphic to
Kashiwara's crystal associated to the crystal base of $\U_\Q(\n_\infty)$.
\end{thm}

\subsection{Quantum Groups and Shuffle Algebras}\label{SS:ShuffleAlg} Let $\b_r$ be the simple finite
dimensional Lie algebra of type $B_r$ over $\C$, and $\U_q(\b_r)$ the associated
quantum group with Chevalley generators $e_i,f_i$ ($i=0,\ldots,r-1$)
corresponding to the labeling of the Dynkin diagram:

\begin{center}
\begin{picture}(340,30)

\put(100,15){\circle{4}}\put(99,0){$0$}
    \put(100,17){\line(1,0){32}}\put(100,13){\line(1,0){32}}\put(113,12.5){$<$}
\put(133,15){\circle{4}}\put(132,0){$1$}
    \put(135,15){\line(1,0){30}}
\put(167,15){\circle{4}}\put(166,0){$2$}
    \put(169,15){\line(1,0){30}}
\put(201,15){\circle{4}}\put(200,0){$3$}
    \put(210,12){$\cdots$}
\put(235,15){\circle{4}}\put(228,0){$r-2$}
    \put(237,15){\line(1,0){30}}
\put(269,15){\circle{4}}\put(265,0){$r-1$}

\end{picture}
\end{center}

Fix a triangular decomposition $\b_r=\n^+_r\oplus\h_r\oplus\n^-_r$. Let $\Dt$ be the
root system of $\b_r$ relative to this decomposition, $\Dt^+$ the positive roots,
and $\Pi=\{\bt_0,\ldots,\bt_{r-1}\}$ the simple roots. Let $\mathcal{Q}$ be the
root lattice and $\mathcal{Q}^+=\sum_{i=0}^{r-1}\Z_{\geq 0}\bt_i$. Finally, let
$(\cdot,\cdot)$ denote the trace form on $\h^*$. The Cartan matrix of $\b_r$ is
then $A=(a_{ij})_{i,j=0}^{r-1}$, where
\[
a_{ij}=\frac{2(\bt_i,\bt_j)}{(\bt_i,\bt_i)},\;\;\; d_i=\frac{(\bt_i,\bt_i)}2\in\{1,2\}.
\]
Let $q_i=q^{d_i}$. To avoid confusion with notation we will use later, we adopt
the following non-standard notation for $q$-integers and $q$-binomial
coefficients:
\[
(k)_i=\frac{q_i^k-q_i^{-k}}{q_i-q_i^{-1}}.
\]

The algebra $\U_q=\U_q(\n^+_r)$ is naturally $\mathcal{Q}^+$-graded by
assigning to $e_i$ the degree $\bt_i$. Let $|u|$ be the $\curlyQ^+$-degree of a
homogeneous element $u\in\U_q(\n^+_b)$.

There exist $q$-derivations $e_i'$, $i=0,\ldots,r-1$ given by
\[
e_i'(e_j)=\dt_{ij}\andeqn e_i'(uv)=e_i'(u)v+q^{(\bt_i,|u|)}ue_i'(v)
\]
for all homogeneous $u,v\in\U_q^+$.

Now, let $\F$ be the free associative algebra over $\Q(q)$ generated by the set
of letters $\{[0],\ldots,[r-1]\}$. Write
$[i_1,\ldots,i_k]:=[i_1]\cdot[i_2]\cdots[i_k]$, and let $[]$ denote the empty
word. The algebra $\F$ is $\curlyQ^+$ graded by assigning the degree $\bt_i$ to
$[i]$ (as before, let $|f|$ denote the $\curlyQ^+$-degree of a homogeneous
$f\in\F$). Notice that $\F$ also has a \emph{principal grading} obtained by
setting the degree of a letter $[i]$ to be 1; let $\F_d$ be the $d$th graded
component in this grading.

Now, define the (quantum) shuffle product, $*$, on $\F$ inductively by
\begin{align}\label{E:inductiveqshuffle}
(x\cdot[i])*(y\cdot[j])=(x*(y\cdot[j])\cdot[i]+q^{-(|x|+\bt_i,\bt_j)}((x\cdot[i])*y)\cdot[j],\;\;\;x*[]=[]*x=x.
\end{align}
Iterating this formula yields
\[
[i_1,\ldots,i_\ell]*[i_{\ell+1},\ldots,i_{\ell+k}]
=\sum_{w\in D_{(\ell,k)}}q^{-e(w)}[i_{w(1)},\ldots,i_{w(k+\ell)}]
\]
where
\[
e(w)=\sum_{\substack{s\leq\ell<t\\w(s)<w(t)}}(\bt_{i_{w(s)}},\bt_{i_{w(t)}}),
\]
see \cite[$\S2.5$]{lec} for details. The product $*$ is associative and,
\cite[Proposition 1]{lec},
\begin{eqnarray}\label{E:qShuffle}
x*y=q^{-(|x|,|y|)}y\overline{*}x
\end{eqnarray}
where $\overline{*}$ is obtained by replacing $q$ with $q^{-1}$ in the
definition of $*$.

Now, to $f=[i_1,\ldots,i_k]\in\F$, associate $\del_f=e_{i_1}'\cdots
e_{i_k}'\in\End \U_q$, and $\del_{[]}=\operatorname{Id}_{\U_q}$. Then,

\begin{prp}\cite{ro1,ro2,grn} There exists an injective $\Q(q)$-linear homomorphism
\[
\Psi:\U_q\rightarrow(\F,*)
\]
defined on homogeneous $u\in\U_q$ by the formula $\Psi(u)=\sum\del_f(u)f$,
where the sum is over all monomials $f\in\F$ such that $|f|=|u|$.
\end{prp}

Therefore $\U_q^+$ is isomorphic to the subalgebra $\W\subseteq(\F,*)$
generated by the letters $[i]$, $0\leq i<r$.

Let $\A=\Q[q,q^{-1}]$, and let $\U_\A$ denote the $\A$-subalgebra of $\U_q$
generated by the divided powers $e_i^k/(k)_i!$ ($0\leq i<r$, $k\in\Z_{\geq0}$).
Let $(\cdot,\cdot)_K:\U_q\times\U_q\rightarrow\Q(q)$ denote the unique
symmetric bilinear form satisfying
\[
(1,1)_K=1\andeqn(e_i'(u),v)_k=(u,e_iv)_K
\]
for all $0\leq i<r$, and $u,v\in\U_q$. Let
\begin{align}\label{E:DualEnvelope}
\U_\A^*=\{\,u\in\U_q \mid (u,v)_K\in\A\mbox{ for all }v\in\U_\A\,\}
\end{align}
and let $u^*\in\U_\A^*$ denote the dual to $u\in\U_\A$ relative to $(\cdot,\cdot)_K$.

Now, given a monomial
\[
[i_1^{a_1},i_2^{a_2},\ldots,i_k^{a_k}]
    =[\underbrace{i_1,\ldots,i_1}_{a_1},\underbrace{i_2,\ldots,i_2}_{a_2},
    \ldots,\underbrace{i_k,\ldots,i_k}_{a_k}]
\]
 with $i_j\neq i_{j+1}$ for $1\leq j<k$, let
 $c_{i_1,\ldots,i_k}^{a_1,\ldots,a_k}=(a_1)_{i_1}!\cdots(a_k)_{i_k}!$, so that
 $(c_{i_1,\ldots,i_k}^{a_1,\ldots,a_k})^{-1}e_{i_1}^{a_1}\cdots e_{i_k}^{a_k}$ is a
 product of divided powers. Let
\[
\F_\A=\bigoplus\A c_{i_1,\ldots,i_k}^{a_1,\ldots,a_k}
[i_1^{a_1},i_2^{a_2},\ldots,i_k^{a_k}]
\]
and $\W^*_\A=\W\cap\F_\A$. It is known that $\W_\A^*=\Psi(\U_\A^*)$,
\cite[Lemma 8]{lec}.

Define
\[
\F_\C=\C\otimes_\A\F_\A,
\andeqn\W_\C^*=\C\otimes_\A\W_\A^*
\]
where $\C$ is an $\A$-module via $q\rightarrow 1$. Given an element $E\in\W_\A$
(resp. $\F_\A$) let $\underline{E}$ denote its image in $\W_\C$
(resp. $\F_\C$).

Observe that $(\F_\C,*)$ is the classical shuffle algebra and the shuffle
product coincides with the formula for the characters associated to parabolic
induction of $\ASe(d)$-modules (see Lemma \ref{L:ShuffleLemma}).

We close this section by describing the bar involution on $\F$:

\begin{dfn}\label{D:BarInv}\cite[Proposition 6]{lec} Let $-:\F\rightarrow\F$ be the
$\Q$-linear automorphism of $(\F,*)$ defined by $\bar{q}=q^{-1}$ and
\[
\overline{[i_1,\ldots,i_k]}
=q^{-\sum_{1\leq s<t\leq k}(\bt_{i_s},\bt_{i_t})}[i_k,\ldots,i_1].
\]
\end{dfn}

\subsection{Good Words and Lyndon Words}\label{SS:LyndonWords} In what follows,
it is convenient to differ from the conventions in \cite{lec}. In particular,
it is natural from our point of view to order monomial in $\F$
lexicographically reading from \emph{right to left}. Unlike the type $A$ case,
this convention leads to some significant differences in the good Lyndon words
that appear. This section contains a careful explanation of all the changes
that occur.

Fix the ordering on the set of letters in $\F$ (resp. $\Pi$):
$[0]<[1]<\cdots<[r-1]<[]$  (resp. $\bt_0<\bt_1<\cdots<\bt_{r-1}$). Give the set
of monomials in $\F$ the associated lexicographic order read from right to
left. That is,
\[
[i_1,\ldots,i_k]<[j_1,\ldots,j_\ell]\mbox{ if }i_k<j_\ell,\mbox{ or for some
}m, i_{k-m}<j_{\ell-m}\mbox{ and }i_{k-s}<j_{\ell-s}\mbox{ for all }s<m.
\]
Note that since the empty word is larger than any letter, every word is smaller
than all of its right factors:
\begin{align}\label{E:rightfactors}
[i_1,\ldots,i_k]<[i_j,\ldots,i_k],\mbox{ for all }1<j\leq k.
\end{align}
(For those familiar with the theory, this definition is needed to ensure that
the induced Lyndon ordering on positive roots is convex, cf.
$\S$\ref{SS:PBWandCanonical} below.)

For a homogeneous element $f\in\F$, let $\min(f)$ be the smallest monomial
occurring in the expansion of $f$. A monomial $[i_1,\ldots,i_k]$ is called a
\emph{good word} if there exists a homogeneous $w\in\W$ such that
$[i_1,\ldots,i_k]=\min(w)$, and is called a \emph{Lyndon word} if it is larger
than any of its proper left factors:
\[
[i_1,\ldots,i_j]<[i_1,\ldots,i_k],\mbox{ for any }1\leq j<k.
\]
Let $\mathcal{G}$ denote the set of good words, $\L$ the set of Lyndon words,
and $\mathcal{GL}=\L\cap\mathcal{G}\subset\mathcal{G}$ the set of good Lyndon
words.

\begin{lem}\label{L:GoodFactors}\cite[Lemma 13]{lec} Every factor of a good word is
good.
\end{lem}

Because of our ordering conventions, \cite[Lemma 15, Proposition 16]{lec}
become

\begin{lem}\cite[Lemma 15]{lec} Let $l\in\L$, $w$ a monomial such that $w\geq l$. Then, $\min(w*l)=wl$.
\end{lem}
\noindent and
\begin{prp}\label{P:GLproduct}\cite[Proposition 16]{lec} Let $l\in\mathcal{GL}$, and $g\in\mathcal{G}$ with $g\geq l$. Then $gl\in\mathcal{G}$.
\end{prp}

Hence, we deduce from Lemma \ref{L:GoodFactors} and Proposition
\ref{P:GLproduct} \cite[Proposition 17]{lec}:

\begin{prp}\cite{lr,lec} A monomial $g$ is a good word if, and only if,
there exist good Lyndon words $l_1\geq\ldots\geq l_k$ such that
\[
g=l_1l_2\cdots l_k.
\]
\end{prp}

As in \cite{lec}, we have

\begin{prp}\cite{lr,lec} The map $l\rightarrow|l|$ is a bijection
$\mathcal{GL}\rightarrow\Dt^+$.
\end{prp}

Given $\gm\in\Dt^+$, let $\gm\rightarrow l(\gm)$ be the inverse of the above
bijection (called the Lyndon covering of $\Dt^+$).

We now define the \emph{bracketing} of Lyndon words, that gives rise to the
\emph{Lyndon basis} of $\W$. To this end, given $l\in\L$ such that $l$ is not a
letter, define the standard factorization of $l$ to be $l=l_1l_2$ where
$l_2\in\L$ is a proper left factor of maximal length. Define the $q$-bracket
\begin{align}\label{E:qbracket}
[f_1,f_2]_q=f_1f_2-q^{(|f_1|,|f_2|)}f_2f_1
\end{align}
for homogeneous $f_1,f_2\in\F$ in the $\curlyQ^+$-grading. Then, the bracketing
$\la l\ra$ of $l\in\L$ is defined inductively by $\la l\ra=l$ if $l$ is a
letter, and
\begin{align}\label{E:Lyndonbracketing}
\la l\ra=[\la l_1\ra,\la l_2\ra]_q
\end{align}
if $l=l_1l_2$ is the standard factorization of $l$.

\begin{exa} (1) $\la [0]\ra=[0]$;

\noindent(2) $\la [12]\ra=[[1],[2]]_q=[12]-q^{-1}[21]$;

\noindent(3) $\la[012]\ra=[[0],[12]-q^{-1}[21]]_q=[012]-q^{-1}[021]-q^{-2}[120]+q^{-3}[210]$.
\end{exa}

As is suggested in this example, we have

\begin{prp}\label{P:bracketingtriangularity}\cite[Proposition 19]{lec} For $l\in\L$, $\la l\ra=l+r$ where $r$ is
a linear combination of words $w$ such that $|w|=|l|$ and $w<l$.
\end{prp}

Any word $w\in\F$ has a canonical factorization $w=l_1\cdots l_k$ such that
$l_1,\ldots,l_k\in\L$ and $l_1\geq\cdots\geq l_k$. We define the bracketing of
an arbitrary word $w$ in terms of this factorization: $\la w\ra=\la
l_1\ra\cdots\la l_k\ra$. Define a homomorphism $\Xi:(\F,\cdot)\to(\F,*)$ by
$\Xi([i])=[i]$. Then,
$\Xi([i_1,\ldots,i_k])=[i_1]*\cdots*[i_k]=\Psi(e_{i_1}\cdots e_{i_k})$. In
particular, $\Xi(\F)=\W$. We have the following characterization of good words:

\begin{lem}\cite[Lemma 21]{lec} The word $w$ is good if and only if it cannot
be expressed modulo $\ker\Xi$ as a linear combination of words $v<w$.
\end{lem}

For $g\in\mathcal{G}$, set $r_g=\Xi(\la g\ra)$. Then, we have

\begin{thm}\label{T:Lyndonbasis}\cite[Propostion 22, Theorem 23]{lec}
Let $g\in\mathcal{G}$ and $g=l_1\cdots l_k$ be the canonical factorization of
$g$ as a nonincreasing product of good Lyndon words. Then
\begin{enumerate}
\item $r_g=r_{l_1}*\cdots*r_{l_k}$,
\item $r_g=\Psi(e_g)+\sum_{w<g}x_{gw}\Psi(e_w)$ where, for a word
$v=[i_1,\ldots,i_k]$, $e_v=e_{i_1}\cdots e_{i_k}$, and
\item $\{r_g|g\in\mathcal{G}\}$ is a basis for $\W$.
\end{enumerate}
\end{thm}

The basis $\{r_g\mid g\in\mathcal{G}\}$ is called the Lyndon basis of $\W$. An
immediate consequence of Proposition \ref{P:bracketingtriangularity} and
Theorem \ref{T:Lyndonbasis} is the following:

\begin{prp}\label{P:LyndonCoveringProperty}\cite[Proposition 24]{lec} Assume $\gm_1,\gm_2\in \Dt^+$,
$\gm_1+\gm_1=\gm\in\Dt^+$, and $l(\gm_1)<l(\gm_2)$. Then, $l(\gm_1)l(\gm_2)\geq
l(\gm)$.
\end{prp}

This gives an inductive algorithm to determine $l(\gm)$ for $\gm\in\Dt^+$ (cf.
\cite[$\S4.3$]{lec}):

For $\bt_i\in\Pi\subset\Dt^+$, $l(\bt_i)=[i]$. If $\gm$ is not a simple root,
then there exists a factorization $l(\gm)=l_1l_2$ with $l_1,l_2$ Lyndon words.
By Lemma \ref{L:GoodFactors}, $l_1$ and $l_2$ are good, so $l_1=l(\gm_1)$ and
$l_2=l(\gm_2)$ for some $\gm_1,\gm_2\in\Dt^+$ with $\gm_1+\gm_2=\gm$. Assume
that we know $l(\gm_0)$ for all $\gm_0\in\Dt^+$ satisfying
$\height(\gm_0)<\height(\gm)$. Define
\[
C(\gm)=\{\,(\gm_1,\gm_2)\in\Dt^+\times\Dt^+ \mid \gm=\gm_1+\gm_2, \mbox{ and
}l(\gm_1)<l(\gm_2)\,\}.
\]
Then, Proposition \ref{P:LyndonCoveringProperty} implies

\begin{prp}\cite[Proposition 25]{lec} We have
\[
l(\gm)=\min\{\,l(\gm_1)l(\gm_2) \mid (\gm_1,\gm_2)\in C(\gm)\,\}
\]
\end{prp}

In our situation,
\[
\Dt^+=\{\bt_i+\bt_{i+1}+\cdots+\bt_j|0\leq i\leq j<r\}
\cup\{2\bt_0+\cdots+2\bt_j+\bt_{j+1}+\cdots+\bt_k|0\leq j<k<r\}.
\]
A straightforward inductive argument shows that
\[
l(\bt_i+\bt_{i+1}\cdots+\bt_j)=[i,i+1,\ldots,j]\andeqn
    l(2\bt_0+\cdots+2\bt_j+\bt_{j+1}+\cdots+\bt_k)=[j,j-1,\ldots,0,0,\ldots,k-1,k].
\]
Remarkably,
\begin{prp} In the notation of Lemma \ref{L:ShuffleLemma} we have
\[
l(\bt_i+\cdots+\bt_j)=\ch\Phi_{[i,j]}
\]
and
\[
2l(2\bt_0+\cdots+2\bt_j+\bt_{j+1}+\cdots+\bt_k)
    =\ch\Phi_{[-j-1,k]}.
\]
\end{prp}

Observe that we may write any good Lyndon word uniquely in the form
$l=[i,i+1,\ldots,j]$ where $i,j\in\Z$ and $0\leq|i|\leq j<r$. For example,
\begin{align}\label{E:GoodLyndonWordConvention}
l(2\bt_0+\cdots+2\bt_j +\bt_{j+1}+\cdots+\bt_k)=[-j-1,\ldots,k].
\end{align}

In the following definition, we mean for $n$ to vary. Given $\ld\in
P_{>0}^{++}$, let
\begin{align}\label{E:Bdld}
\mathcal{B}_d(\ld)=\{\,\mu\in P^+[\ld] \mid \ld-\mu\in \Ppos(d)\mbox{ and
}|\mu_i|<\ld_i\mbox{ for all }i\,\}
\end{align}
and let
\begin{align}\label{E:Bd}
\mathcal{B}_d=\{\,(\ld,\mu) \mid \ld\in P_{>0}^{++}\mbox{ and
}\mu\in\mathcal{B}_d(\ld)\,\}.
\end{align}

Let $\mathcal{G}_d=\mathcal{G}\cap\F_d$ be the set of good words of principal
degree $d$. We have

\begin{lem}\label{L:BdGd} The map
$(\ld,\mu)\mapsto[\ld-\mu]=[\mu_1,\ldots,\ld_1-1,\ldots,\mu_n,\ldots,\ld_n-1]$
induces a bijection $\mathcal{B}_d\rightarrow\mathcal{G}_d$.
\end{lem}

\begin{pff} By \eqref{E:GoodLyndonWordConvention}, $[\ld-\mu]$ is a
well-defined element of $\F_d$. Since $\ld\in P^{++}_{>0}$ and $\mu\in
P^+[\ld]$, the ordering convention and \eqref{E:rightfactors} imply that
$[\ld-\mu]\in\mathcal{G}_d$. This map is clearly bijective.
\end{pff}

\subsection{PBW and Canonical Bases}\label{SS:PBWandCanonical}
The lexicographic ordering on $\mathcal{GL}$ induces a total ordering on
$\Dt^+$, which is \emph{convex}, meaning that if $\gm_1,\gm_2\in\Dt^+$ with
$\gm_1<\gm_2$, and $\gm=\gm_1+\gm_2\in\Dt^+$, then $\gm_1<\gm<\gm_2$ (cf.
\cite{ro3,lec}).

Indeed, assume $\gm_1,\gm_2,\gm=\gm_1+\gm_2\in\Dt^+$ and $\gm_1<\gm_2$.
Proposition \ref{P:LyndonCoveringProperty} and \eqref{E:rightfactors} imply
that $l(\gm)\leq l(\gm_1)l(\gm_2)<l(\gm_2)$. If $l(\gm)=l(\gm_1)l(\gm_2)$, then
the definition of Lyndon words implies $l(\gm_1)<l(\gm)$. We are therefore left
to prove that $l(\gm_1)<l(\gm)$ even if $l(\gm)<l(\gm_1)l(\gm_2)$. This cannot
happen if $\gm=\bt_i+\cdots+\bt_j$. In the case $\gm=2\bt_0+\cdots+2\bt_j
+\bt_{j+1}+\cdots+\bt_k$, the possibilities for $\gm_1<\gm_2$ are
$\gm_1=\bt_i+\cdots+\bt_j$ and
$\gm_2=2\bt_0+\cdots+2\bt_{i-1}+\bt_i+\cdots+\bt_k$ for $0\leq i\leq j$. In any
of these cases, $[i,\ldots,j]<[j,\ldots,0,0,\ldots,k]$. That is,
$l(\gm_1)<l(\gm)<l(\gm_2)$.

Each convex ordering, $\gm_1<\cdots<\gm_N$, on $\Dt^+$ arises from a unique
decomposition $w_0=s_{i_1}s_{i_2}\cdots s_{i_N}$ of the longest element of the
Weyl group of type $B_r$ via
\[
\gm_1=\bt_{i_1},\;\gm_2=s_{i_1}\bt_{i_2},\;\cdots,\gm_N=s_{i_1}\cdots s_{i_{N-1}}\bt_{i_N}.
\]
Lusztig associates to this data a PBW basis of $\U_\A$ denoted
\[
E^{(a_1)}(\gm_1)\cdots E^{(a_n)}(\gm_N),\;\;\;(a_1,\ldots,a_N)\in\Z_{\geq0}^N.
\]
Leclerc \cite[$\S4.5$]{lec} describes the image in $\W$ of this basis for the
convex Lyndon ordering. We use the same braid group action as Leclerc and the results of \cite[$\S4.5,4.6$]{lec} carry over, making changes in the same manner indicated in the previous section. We describe the relevant facts below.

For $g=l(\gm_1)^{a_1}\cdots l(\gm_k)^{a_k}$, where $\gm_1>\cdots>\gm_k$ and
$a_1,\ldots,a_k\in\Z_{>0}$ set
\[
E_g=\Psi(E^{(a_k)}(\gm_k)\cdots E^{(a_1)}(\gm_1))\in\W_\A
\]
and let $E_g^*\in\W_\A^*$ be the image of $(E^{(a_k)}(\gm_k)\cdots
E^{(a_1)}(\gm_1))^*\in\U_\A^*$. Observe that the order of the factors in the
definition of $E_g$ above are increasing with respect to the Lyndon ordering.
Leclerc shows that if $\gm\in\Dt^+$, then
\begin{align}\label{E:Proportional}
\kp_{l(\gm)}E_{l(\gm)}=r_{l(\gm)},
\end{align}
For some $\kp_{l(\gm)}\in\Q(q)$, \cite[Theorem 28]{lec} (the proof of this theorem in our case is
obtained by reversing all the inequalities and using the standard factorization
as opposed to the costandard factorization).

More generally, let $f\mapsto f^t$ be the linear map defined by
$[i_1,\ldots,i_k]^t=[i_k,\ldots,i_1]$ and $(x*y)^t=y^t*x^t$. Then, $E_g$ is
proportional to $\overline{r_g^t}$ (cf. \cite[$\S4.6$, $\S5.5.2-5.5.3$]{lec}).

As in \cite[$\S5.5.3$]{lec}, we see that there exists an explicit $c_g\in\Z$
such that
\[
E_g^*=q^{c_g} (E_{l_m}^*)*\cdots*(E_{l_1}^*)
\]
if $g=l_1\cdots l_m$ with $l_1>\cdots>l_m$. Using \eqref{E:qShuffle} we deduce
that
\[
E_g^*=q^{C_g}(E_{l_1}^*)\bar{*}\cdots\bar{*}(E_{l_m}^*),
\]
where $C_g=c_g-\sum_{1\leq i<j\leq m}(\bt_i,\bt_j)$. In particular,
\begin{align}\label{E:Eshuffle}
\underline{E_g^*}=\underline{(E_{l_1}^*)*\cdots*(E_{l_m}^*)}.
\end{align}

Using the bar involution (Definition \ref{D:BarInv}), Leclerc constructs the
canonical basis, $\{b_g \mid g\in\mathcal{G}\}$ for $\W_\A$ via the PBW basis
$\{E_g \mid g\in\mathcal{G}\}$. It has the form
\[
b_g=E_g+\sum_{\substack{h\in\mathcal{G}\\h<g}}\chi_{gh}E_h.
\]
The dual canonical basis then has the form
\[
b_g^*=E_g^*+\sum_{\substack{h\in\mathcal{G}\\h>g}}\chi_{gh}^*E_h^*.
\]
In particular, for good Lyndon words, \cite[Corollary 41]{lec},
$b^*_{l}=E^*_{l}$ for every $l\in\mathcal{GL}$. As in \cite[Lemma 8.2]{lec}, we
see that $b^*_{[i,\ldots,j]}=[i,\ldots,j]$ for $0\leq i<j<r$. We now prove

\begin{lem}\label{L:DblSeg} For $0\leq j<k<r$, one has
\[
b^*_{[j,\ldots,0,0,\ldots,k]}=(2)_0[j,\ldots,0,0,\ldots,k].
\]
\end{lem}

\begin{pff}
We prove this by induction on $j$ and $k$ with $j<k$, using \eqref{E:inductiveqshuffle}, \eqref{E:qbracket}, and \eqref{E:Lyndonbracketing} for the computations.

Observe that for $k\geq1$, $r_{[0,1,\ldots,k]}=(q^2-q^{-2})^k[0,1,\ldots,k]$, which can be proved easily by downward induction on $j$, $0\leq j<k$, using \eqref{E:inductiveqshuffle} and
\[
r_{[j,\ldots,k]}=\Xi(\la[j,\ldots,k]\ra)=\Xi([[j],\la[j+1,\ldots,k]]_q)=[j]*r_{[j+1,\ldots,k]}-q^{-2}r_{[j+1,\ldots,k]}*[j].
\]
By \eqref{E:inductiveqshuffle}, we have
\begin{align*}
[0]*[0,1]-[0,1]*[0]&=[0,1,0]+q^2([0]*[0])[1]-([0]*[0])[1]-[0,1,0]\\
&=(q^2-1)([0,0]+q^{-2}[0,0])[1]=(q^2-q^{-2})[0,0,1]
\end{align*}
Therefore, applying \eqref{E:Lyndonbracketing} and the relevant definitions, we deduce that
\begin{align*}
r_{[0,0,1]}&=\Xi(\la[0,0,1]\ra)\\
    &=\Xi([[0],\la[0,1]\ra]_q^2)\\
    &=[0]*r_{[0,1]}-r_{[0,1]}*[0]\\
    &=(q^2-q^{-2})([0]*[0,1]-[0,1]*[0])\\
    &=(q^2-q^{-2})^2[0,0,1]
\end{align*}
Once again, using \eqref{E:inductiveqshuffle}, we deduce that for all $k\geq2$,
\begin{align}\label{E:DblSegReduction0}
[0]*[0,\ldots,k]-[0,\ldots,k]*[0]=([0]*[0,\ldots,k-1]-[0,\ldots,k-1]*[0])[k].
\end{align}
Assume $k\geq2$. Then, $(\bt_0,\bt_0+\cdots+\bt_k)=0$, so iterated applications of \eqref{E:DblSegReduction0} yields
\begin{align*}
r_{[0,0,\ldots,k]}&=[0]*r_{[0,\ldots,k]}-r_{[0,\ldots,k]}*[0]\\ &=(q^2-q^{-2})^k([0]*[0,\ldots,k]-[0,\ldots,k]*[0])\\
    &=(q^2-q^{-2})^k([0]*[0,1]-[0,1]*[0])[2,\ldots, k]\\
    &=(q^2-q^{-2})^{k+1}[0,0,\ldots, k]
\end{align*}

Now, assume that $k\geq 2$, and $0<j<k$. To compute $r_{[j,\ldots,0,0,\ldots,k]}$, we need the following. For $|j-k|>1$,
\begin{align}\label{E:DblSegReduction1}
[j]*[j-1,\ldots,&k]-q^{-2}[j-1,\ldots,k]*[j]\\
    \nonumber&=([j]*[j-1,\ldots,k-1]-q^{-2}[j-1,\ldots,k-1]*[j])[k].
\end{align}
For $j=k-1$,
\begin{align}\label{E:DblSegReduction2}
[j]*[j-1,\ldots,0,&0,\ldots,j+1]-q^{-2}[j-1,\ldots,0,0,\ldots,j+1]*[j]\\
    \nonumber&=(q^2[j]*[j-1,\ldots,0,0,\ldots,j]-q^{-2}[j-1,\ldots,0,0,\ldots,j]*[j])[j+1].
\end{align}
Finally,
\begin{align}\label{E:DblSegReduction3}
q^2[j]*[j-1,\ldots,0,&0,\ldots,j]-q^{-2}[j-1,\ldots,0,0,\ldots,j]*[j]\\
    \nonumber&=([j]*[j-1,\ldots,0,0,\ldots,j-2]-q^{-2}[j-1,\ldots,0,0,\ldots,j-2]*[j])[j,j+1].
\end{align}
Indeed, \eqref{E:DblSegReduction1} and \eqref{E:DblSegReduction2} are straightforward applications of \eqref{E:inductiveqshuffle}. Equation \eqref{E:DblSegReduction3} involves a little more calculation:
\begin{align*}
q^2[j]*[j-1,\ldots,0,&0,\ldots,j]-q^{-2}[j-1,\ldots,0,0,\ldots,j]*[j]\\
=&q^2[j-1,\ldots,0,0,\ldots,j,j]+q^{-2}([j]*[j-1,\ldots,0,0,\ldots,j-1]\\
    &-[j-1,\ldots,0,0,\ldots,j-1]*[j])[j]-q^{-2}[j-1,\ldots,0,0,\ldots,j,j]\\
    =&(q^2-q^{-2})[j-1,\ldots,0,0,\ldots,j,j]+q^{-2}([j-1,\ldots,0,0,\ldots,j]\\&+q^2([j]*[j-1,\ldots,0,0,\ldots,j-2])[j-1]
    -([j-1,\ldots,0,0,\ldots,j-2]*[j])[j-1]\\&-q^4[j-1,\ldots,0,0,\ldots,j])[j]\\
    =&([j]*[j-1,\ldots,0,0,\ldots,j-2]-q^{-2}[j-1,\ldots,0,0,\ldots,j-2]*[j])[j,j+1],
\end{align*}
Note that \eqref{E:DblSegReduction1} holds for both $[j-1,j,\ldots,k]$ and $[j-1,\ldots,0,0,\ldots,k]$.

Now, assume that we have shown that $r_{[j-1,\ldots,0,0,\ldots,k]}=(q^2-q^{-2})^{j+k}[j-1,\ldots,0,0,\ldots,k]$. Then, since $(\bt_j,2\bt_0+\cdots+2\bt_{j-1}+\bt_j+\cdots+\bt_k)=-2$,
\begin{align*}
r_{[j,\ldots,0,0,\ldots,k]}=&[j]*r_{[j-1,\ldots,0,0,\ldots,k]}-r_{[j-1,\ldots,0,0,\ldots,k]}*[j]\\
    =&(q^2-q^{-2})^{j+k}[j]*[j-1,\ldots,0,0,\ldots,k]-q^{-2}[j-1,\ldots,0,0,\ldots,k]*[j]\\
    =&(q^2-q^{-2})^{j+k}([j]*[j-1,\ldots,0,0,\ldots,j+1]\\&-q^{-2}[j-1,\ldots,0,0,\ldots,j+1]*[j])[j+2,\ldots,k]
    \;\;\;\;\;\;\;\;\;\;\;\;\;\;\;\;\;\;\;\;\;\;\;\;\;\;\;\;\;\mbox{by \eqref{E:DblSegReduction1}}\\
    =&(q^2-q^{-2})^{j+k}(q^2[j]*[j-1,\ldots,0,0,\ldots,j]\\&-q^{-2}[j-1,\ldots,0,0,\ldots,j]*[j])[j+1,\ldots,k]
    \;\;\;\;\;\;\;\;\;\;\;\;\;\;\;\;\;\;\;\;\;\;\;\;\;\;\;\;\;\;\;\;\;\mbox{by \eqref{E:DblSegReduction2}}\\
    =&(q^2-q^{-2})^{j+k}([j]*[j-1,\ldots,0,0,\ldots,j-2]\\&-q^{-2}[j-1,\ldots,0,0,\ldots,j-2]*[j])[j,\ldots,k]
    \,\;\;\;\;\;\;\;\;\;\;\;\;\;\;\;\;\;\;\;\;\;\;\;\;\;\;\;\;\;\;\;\;\mbox{by \eqref{E:DblSegReduction3}}\\
    =&(q^2-q^{-2})^{j+k}([j]*[j-1]-q^{-2}[j-1]*[j])[j-2,\ldots,0,0,\ldots,k]\;\;\;\;\mbox{by \eqref{E:DblSegReduction1}}\\
    =&(q^2-q^{-2})^{j+k+1}[j,\ldots,0,0,\ldots,k].
\end{align*}

Finally, the result follows after computing the normalizing coefficient \eqref{E:Proportional} using \cite[Equation (28)]{lec}. We leave the details to the reader.
\end{pff}

\subsection{}In section we give a representation theoretic interpretation of the good
Lyndon words associated to the root vectors
$2\bt_0+\cdots+2\bt_j+\bt_{j+1}+\cdots+\bt_k$ ($0\leq j<k<r$) which appear in
\cite[Lemma 53]{lec}. The corresponding dual canonical basis vectors are given by the formula
\[
[0]\cdot([1,\ldots,j]*[0,\ldots,k]).
\]

\begin{lem} Let $0\leq a<b$, $d=b+a+2$, $\ld=(b+1,a+1)$ and $\af=(1,-1)$.
Then, for $1\leq k\leq a$,
\[
\ch\L(\ld,-k\af)=2\underline{[k-1]\cdot([k-2,k-3,\ldots,1,0,0,1,\ldots,b]*[k,\ldots,a])}
\]
where if $ k=1$, we interpret
\[[k-2,k-3,\ldots,1,0,0,1,\ldots,b]=[0,1,\ldots, b]
\]
\end{lem}

\begin{proof} By \cite[Proposition 11.4]{g}, for each $k\in\Z_{\geq0}$, there exists a
short exact sequence
\[
\xymatrix{0\ar[r]&L(-(k+1)\af)\ar[r]&M(-k\af)\ar[r]&L(-k\af)\ar[r]&0}.
\]
For $k\leq a+1$, applying the functor $F_\ld$ yields the exact sequence
\begin{eqnarray}\label{E:ShortExactSeq}
\xymatrix@1{0\ar[r]&F_\ld L(-(k+1)\af)\ar[r]&2\M(\ld,-k\af)\ar[r]&F_\ld
L(-k\af)\ar[r]&0}.
\end{eqnarray}
Therefore,
\[
\ch F_\ld L(-k\af)=4\underline{[k-1,\ldots,1,0,0,1,\ldots,b]*[k,\ldots,a]}-\ch
F_\ld L(-(k+1)\af).
\]
Note that when $k=a+1$, $F_\ld L(-(k+1)\af)=0$ since $\M(\ld,-(a+2)\af)=0$.
Therefore the sequence \eqref{E:ShortExactSeq} implies $F_\ld
L(-k\af)=2\L(\ld,-(a+1)\af)\cong2\M(\ld,-(a+1)\af)\cong 2\Phi_{[-a-1,b]}$, and
\[
\ch\Phi_{[-a-1,b]}=2\underline{[a,a-1,\ldots,1,0,0,1,\ldots,b]}.
\]

We now prove the lemma by downward induction on $k\leq a$. We have
\begin{align*}
\ch F_\ld L(-a\af)=&4\,\underline{[a-1,\ldots,1,0,0,1,\ldots,b]*[a]-4[a,\ldots,1,0,0,1,\ldots,b]}\\
=&4\,\underline{[a-1]\cdot([a-2,\ldots,1,0,0,1,\ldots,b]*[a])}.
\end{align*}
Hence, $F_\ld L(-a\af)=2\L(\ld,-a\af)$ and the lemma holds for $k=a$. Now,
assume $k<a$, $F_\ld L(-(k+1)\af)=2\L(\ld,-(k+1)\af)$, and
\[
\ch\L(\ld,-(k+1)\af)=2\underline{[k]\cdot([k-1,\ldots,1,0,0,1,\ldots,b]*[k+1,\ldots,a])}.
\]
Then,
\begin{align*}
\ch F_\ld L(-k\af)=&4\underline{[k-1,\ldots,1,0,0,1,\ldots,b]*[k,\ldots,a]}-
    4\underline{[k]\cdot([k-1,\ldots,1,0,0,1,\ldots,b]*[k+1,\ldots,a])}\\
    =&4\underline{[k-1]\cdot([k-2,\ldots,1,0,0,1,\ldots,b]*[k,\ldots,a])}.
\end{align*}
Hence, $F_\ld L(-k\af)\neq 0$, so $F_\ld L(-k\af)=2\L(\ld,-k\af)$ and the lemma
holds.
\end{proof}

\begin{cor}\label{C:LecDblSeg} Let $0\leq a<b$, $d=b+a+2$, $\ld=(b+1,a+1)$ and
$\mu=-\af=(-1,1)$. Then,
\[
\ch\L(\ld,-\af)=2\,\underline{[0]\cdot[0,\ldots,b]*[1,\ldots,a]}.
\]
\end{cor}

\subsection{A Basis for the Grothendieck Group $K(\mbox{Rep} \ASe(d))$}\label{SS:GrothendieckGroup}

\begin{thm}\label{T:GrothendieckBasis1} The set
\[
\{ \left[ \M(\ld,\mu)\right]\mid (\ld,\mu)\in\mathcal{B}_d\}
\]
forms a basis for $K(\Rep\ASe(d))$.
\end{thm}

\begin{proof} By Lemma \ref{L:DblSeg} and \eqref{E:Eshuffle}, it follows that
$\ch\M(\ld,\mu)=\underline{E^*_{[\ld-\mu]}}$. The result now follows from Lemma
\ref{L:BdGd} and the fact that the character map is injective.
\end{proof}

We will now describe a basis for $K(\Rep\ASe(d))$ in terms of the simple
modules $\L(\ld,\mu)$.

\begin{prp}\label{P:StandardSegmentForm} Let $b\geq0$, $\ld=(b+1,b+1)$
and $\af=(1,-1)$. Then,
\[
\Phi_{[-b-1,b]}\cong\L(\ld,b\af).
\]
\end{prp}

\begin{proof} There is a surjective homomorphism $\M(\ld,b\af)\to\Phi_{[-b-1,b]}$. The result follows since
$\Phi_{[-b-1,b]}$ is simple.
\end{proof}

\begin{cor}\label{C:StandardizingWords} Assume that $\ld\in\P_{>0}^{++}$,
$\mu\in P^+[\ld]$, $\ld-\mu\in\Ppos(d)$, and $|\mu_i|\leq\ld_i$ for all $i$.
Then, there exists $(\eta,\nu)\in\mathcal{B}_d$ such that
\[
\L(\ld,\mu)\cong\L(\eta,\nu),
\]
and $[\ld-\mu]\leq[\eta-\nu]$.
\end{cor}

\begin{proof} First, we may assume $\mu_i<\ld_i$ for all $i$, since
the terms for which $\ld_i=\mu_i$ do not contribute to $\L(\ld,\mu)$. Proceed
by induction on $N(\ld,\mu)=|\{i=1,\ldots,n \mid \mu_i=-\ld_i\}|$. If
$N(\ld,\mu)=0$, then $(\ld,\mu)\in\mathcal{B}^+_d$ so there is nothing to do.
If $N(\ld,\mu)>0$, let $j$ be the smallest index such that $\mu_j=-\ld_j$. Set
$\ld^{(1)}=(\ld_1,\ldots,\ld_{j-1},\ld_j,\ld_j,\ld_{j+1},\ldots,\ld_n)$ and
$\mu^{(1)}=(\mu_1,\ldots,\mu_{j-1},\ld_j-1,\mu_j+1,\mu_{j+1},\ldots,\ld_n)$.
Clearly, $\ld^{(1)}\in P_{>0}^{++}$ and $\mu^{(1)}\in\ld^{(1)}-\Ppos(d)$. We
now show $\mu^{(1)}\in P^+[\ld]$. Indeed, $\ld_j>0$, so
$\ld_j-1>1-\ld_j=\mu_j+1$; and, $\mu_j\geq\mu_{j+1}$, so $\mu_j+1>\mu_{j+1}$.
Since $\mu_j<\ld_j-1$, the $j$th twisted good Lyndon word in
$[\ld^{(1)}-\mu^{(1)}]$ is greater than the $j$th twisted good Lyndon word in
$[\ld-\mu]$. Hence, $[\ld-\mu]\leq[\ld^{(1)}-\mu^{(1)}]$.

Now, there exists a surjective homomorphism
\begin{align*}
\Phi_{[\mu_1,\ld_1-1]}\circledast\cdots\circledast\M((\ld_j,\ld_j),(\ld_j-1,\mu_j+1))
    &\circledast\cdots\circledast\Phi_{[\mu_{n},\ld_{n}-1]}\\
&\to\Phi_{[\mu_1,\ld_1-1]}\circledast\cdots\circledast\Phi_{[\mu_j,\ld_j-1]}
    \circledast\cdots\circledast\Phi_{[\mu_{n},\ld_{n}-1]}
\end{align*}
Hence, a surjective homomorphism $\M(\ld^{(1)},\mu^{(1)})\to\L(\ld,\mu)$. It
follows that $\L(\ld^{(1)},\mu^{(1)})\cong\L(\ld,\mu)$.

Since $N(\ld^{(1)},\mu^{(1)})<N(\ld,\mu)$ the result follows.
\end{proof}

Recall that given $\mu\in\ld-\Ppos(d)$ there exists a unique $w\in S_d[\ld]$
such that $w\mu\in P^+[\ld]$. Let $\mu^+$ denote this element. Also, given
$\ld\in\Pt$, and $\mu\in\ld-\Ppos(d)$, let
$[\ld-\mu]^+=[\ld-\mu^+]\in\mathcal{TG}$ be the associated twisted good word.
The following lemma is straightforward.

\begin{lem}\label{L:WordTriangularity} Assume that $\ld\in\Pt$,
$\ld-\mu\in\Ppos(d)$ and $\gamma\in Q^+$. Then,
$[\ld-\mu]\leq[\ld-(\mu-\gm)^+]$.
\end{lem}

\begin{thm} The following is a complete list of pairwise non-isomorphic simple modules
for $\ASe(d)$:
\[
\{\,\L(\ld,\mu)\mid (\ld,\mu)\in \mathcal{B}^+_d\,\}.
\]
\end{thm}

\begin{proof} Every composition factor of $M(\mu)$ is of the
form $L(\mu-\gamma)$ for some $\gamma \in Q^+$. Applying the functor, we deduce
that every composition factor of $\M(\ld,\mu)$ is of the form
$\L(\ld,\mu-\gamma)\cong\L(\ld,(\mu-\gamma)^+)$. Now, putting together
Corollary \ref{C:StandardizingWords} and Lemma \ref{L:WordTriangularity}, we
deduce that in the Grothendieck group
\[
[\M(\ld,\mu)]=\sum_{\substack{\nu \in\mathcal{B}_d(\eta)\\
\eta \in P^{++}_{>0}\\
[\ld-\mu]\leq[\eta-\nu]}}c_{\lambda,\mu, \eta, \nu}[\L(\eta,\nu)],
\] where the $c_{\lambda,\mu,\eta, \nu}$ are integers and where $c_{\lambda,\mu,\lambda, \mu} \neq 0 $.
Therefore, the transition matrix between the basis for $K(\Rep\ASe(d))$ given
by standard modules and that given by simples is triangular.
\end{proof}

\newpage

\section{Table of Notation}\label{SS:TableofNotation}  For the convenience of the reader we provide a table of notation with a reference to where the notation is first defined.

\bigskip

\begin{center}
\begin{tabular}{ccl}
Notation & & First Defined\\ \hline
$\Se (d)$, $\ASe (d)$, $\P_{d}[x]$, $\A(d) $ & & Section~\ref{SS:Saffdef} \\
$q(a)$ &    & Section~\ref{SS:weights},\eqref{E:qdef}\\
$\P_{d}[x^{2}]$ & & Section~\ref{SS:weights} \\
$\ind^{d}_{\mu}$ & & Section~\ref{SS:Mackey} \\
$D_\nu$, $D_{(m,k)}$ & & Section~\ref{SS:Mackey}  \\
$\gamma_{0}=\gamma_{0}(a_{1}, \dots ,a_{d})$ & & Section~\ref{SS:characters}, \eqref{E:gammazerodef}\\
$[a_1,\ldots,a_d]$ & & Section~\ref{SS:characters} \\
$\Cl_{d}$ & & Section~\ref{subsection irred modules}, \eqref{E:Cldef}\\
$\L_i$, $s_{ij}$&& Section~\ref{subsection irred modules}, \eqref{E:JMelt}\\
$[a,b]$ & & Section~\ref{subsection irred modules} \\
$\hat{\Ph}_{[a,b]}$, $\hat{\Ph}_{[a,b]}^{+}$, $\hat{\Ph}_{[a,b]}^{-}$ & & Section~\ref{subsection irred modules} \\
$\Ph_{[a,b]}$ & & Section~\ref{subsection irred modules}, Definition~\ref{segments} \\
$\hat{\va}_{[a,b]}$, $ \varphi\hat{\va}_{[a,b]}$ & & Section~\ref{subsection irred modules} \\
${\va}_{a,b,n}$ & & Section ~\ref{unique simple quotient}\\
$R$, $R^{+}$, $Q$, $Q^{+}$ & & Section~\ref{SS:LieThy} \\
$P$, $P_{\geq 0}$, $P^{+}$, $P^{++}$, $P^{+}_{\text{rat}}$, $P^{+}_{\text{poly}}$ & & Section~\ref{SS:LieThy} \\
$P(d)$, $P_{\geq 0}(d)$, $P^{+}(d)$, $P^{++}(d)$, $P^{+}_{\text{rat}}(d)$, $P^{+}_{\text{poly}}(d)$ & & Section~\ref{SS:LieThy} \\
$S_{n}[\lambda]$, $R[\lambda]$, $P^{+}[\lambda]$, $P^-[\ld]$ & & Section~\ref{SS:LieThy} \\
$\widehat{\Phi}(\lambda, \mu)$, $\Phi(\lambda, \mu)$ & & Section~\ref{SS:inducedmodules} \\
$\widehat{\M}(\lambda, \mu)$, $\M (\lambda, \mu)$ & & Section~\ref{SS:inducedmodules}, \eqref{E:Mhatdef}, \eqref{E:Mdef}\\
${\M}_{a,b,n}$ & & Section ~\ref{unique simple quotient}\\
$S_{n}[\zeta]$ & & Section~\ref{SS:inducedmodules} \\
$\mathcal{R}(\lambda, \mu)$ & & Section~\ref{unique simple quotient} \\
$L(\lambda, \mu)$ & & Section~\ref{unique simple quotient}, Theorem~\ref{thm:unique irred quotient} \\
$\ld/\mu$& & Section~\ref{S:Calibrated}\\
$\mathcal{Y}_{i,L}$& & Section~\ref{S:Calibrated}\\
$H^{\ld/\mu}$& & Section~\ref{S:Calibrated}\\
$e_{i,j}$, $f_{i,j}$, $\bar{e}_{i,j}$, $\bar{f}_{i,j}$ & & Section~\ref{SS:qndfn} \\
$\mathcal{O}$, $\mathcal{O}(\fq (n))$ & & Section~\ref{SS:RootData} \\
$\widehat{M}(\lambda)$, $M(\lambda)$ & & Section~\ref{SS:RootData} \\
$(\cdot, \cdot)_{S}$ & & Section~\ref{SS:ShapovalovForm}
\end{tabular}

\newpage

\begin{tabular}{ccl}
Notation & & First Defined\\ \hline
$C_{i}$, $S_{i,j}$, $F_i$ & & Section~\ref{SS:Sergeev Duality} \\
$\Omega_{i,j}$ & & Section~\ref{SS:action}  \\
$F_{\lambda}$ & &  Section~\ref{SS:Flambda}, \eqref{E:Flambdadef}\\
$(\cdot, \cdot)_{\mu}$ & & Section~\ref{SS:functorimage}\\
$\varpi(\mu)$ & & Section~\ref{SS:functorimage}\\
$\Dt^+$, $\Pi$, $\mathcal{Q}$, $\mathcal{Q}^+$& & Section~\ref{SS:ShuffleAlg}\\
$(\mathcal{F},*)$, $\mathcal{W}$& & Section~\ref{SS:ShuffleAlg}\\
$\F_\A$, $\F_\C$, $\W_\A$, $\W_\C$& & Section~\ref{SS:ShuffleAlg}\\
$\underline{E}\in\W_\C$& & Section~\ref{SS:ShuffleAlg}\\
$\mathcal{GL}$, $\mathcal{G}$ & & Section~\ref{SS:LyndonWords}\\
$\mathcal{B}_d[\ld]$, $\mathcal{B}_d$ & & Section~\ref{SS:LyndonWords}\\
$[\cdot,\cdot]_q$, $\Xi$, $r_g$ & & Section~\ref{SS:LyndonWords}\\
$E_g$, $E_g^*$, $b_g$, $b_g^*$& & Section~\ref{SS:PBWandCanonical}
\end{tabular}
\end{center}

\pagebreak

\end{document}